\documentclass[11pt]{amsart}
\numberwithin{equation}{section}
\usepackage[margin=1in]{geometry}
\usepackage{amsmath,amssymb,amsthm,mathtools}
\usepackage{enumitem}
\usepackage{needspace}
\usepackage{xcolor}
\usepackage{hyperref}
\usepackage{dsfont}
\usepackage{mathrsfs}
\usepackage{mathtools}
\usepackage{cancel}
\usepackage{soul}
\usepackage{graphicx}
\usepackage{tikz}
\usepackage{comment}
\usetikzlibrary{shapes.geometric}
\usetikzlibrary{arrows.meta,decorations.markings,shapes.geometric}

\definecolor{figureink}{HTML}{26364A}
\definecolor{floworange}{HTML}{D76B42}
\definecolor{curveblue}{HTML}{286FB4}
\definecolor{scalegreen}{HTML}{39926E}
\definecolor{pointcyan}{HTML}{28A7B8}

\usepackage{capt-of}
\usepackage[foot]{amsaddr}
\newcommand{\eps}{\epsilon}
\newcommand{\pzi}{\partial_z^{-1}}
\newcommand{\dd}{\,d}
\newcommand{\norm}[1]{\left\|#1\right\|}
\newcommand{\abs}[1]{\left|#1\right|}
\newcommand{\one}{\mathbf 1}
\newcommand{\wh}{w}
\newcommand{\hh}{h}

\renewcommand{\S}{\mathscr{S}}

\newcommand{\R}{\mathds{R}}
\newcommand{\J}{\mathcal{J}}

\newcommand{\tildes}{{\tilde s}}
\newcommand{\tildez}{{\tilde z}}

\newcommand{\tildeepsilon}{{\tilde \epsilon}}

\allowdisplaybreaks[4]
\DeclarePairedDelimiterX\set[1]\lbrace\rbrace{#1}

\makeatletter
\renewcommand\paragraph{%
  \@startsection{paragraph}{4}{\z@}%
  {3.25ex \@plus 1ex \@minus .2ex}%
  {1ex}%
  {\normalfont\normalsize\itshape}}
\makeatother
\theoremstyle{plain}
\newtheorem{theorem}{Theorem}[section]
\newtheorem{lemma}[theorem]{Lemma}
\newtheorem{proposition}[theorem]{Proposition}
\newtheorem{corollary}[theorem]{Corollary}
\newtheorem{remark}[theorem]{Remark}

\usepackage[doi=false,isbn=false,url=false,eprint=false]{biblatex}
\title{Near-blowup asymptotic analysis and analytical extension for the inviscid primitive equations}

\author[1]{Charles Collot$^{1}$}
\address[1]{Laboratoire de Mathematiques AGM, UMR CNRS 8088, Cergy Paris Université, 2 Avenue Adolphe Chauvin, 95302 Cergy-Pontoise Cedex, France}
\email[1]{charles.collot@cyu.fr}

\author[2]{Slim Ibrahim$^{2,3}$}
\address[2]{Department of Mathematics and Statistics
University of Victoria, 3800 Finnerty Road, Victoria
BC V8P 5C2, Canada}
\address[3]{
Pacific Institute for Mathematical Sciences, 4176-2207 Main Mall, Vancouver, BC V6T 1Z4,
Canada}
\email[2]{ibrahims@uvic.ca}

\author[3]{Lingjun Qian$^2$}
\email[3]{ljqian@uvic.ca}

\date{}

\begin{document}

\begin{abstract}

The primitive equations widely approximate the Navier--Stokes equations in large-scale geophysical fluid dynamics. Their finite-time singularity formation-which occurs exclusively in the inviscid case-sheds light on the mechanisms driving small-scale creation. Building on \cite{Collot2023}, which constructed 2D horizontally odd solutions where the on-axis derivative $\partial_X u(t,0,Z)$ blows up at rate $(T-t)^{-1}$ with boundary concentration scale $|\ln (T-t)|^{-1}$, this paper characterizes this blow-up as forming between hyperbolic stagnation points. 

We derive precise asymptotic expansions for on-axis derivatives beyond leading order, identifying a horizontal concentration scale of $(T-t)^{3/2}$ characteristic of generic Burgers' singularities. Furthermore, we justify the 2D solution's expansion near the symmetry axis by proving $(T-t)^{3/2}$ is a universal lower bound on the radius of analyticity.

\end{abstract}

\maketitle

\pagestyle{plain}

\section{Introduction}

\subsection{The inviscid primitive equations}

While the Navier--Stokes equations are the most fundamental model for fluid motion, solving them numerically becomes computationally prohibitive as domain size or resolution increases. For large-scale geophysical flows where the vertical scale is significantly smaller than the horizontal scale, such as in weather forecasting and ocean circulation, the primitive equations provide an accurate, computationally tractable alternative that retains the essential dynamics. Under appropriate rescaling, the primitive equations arise rigorously from the Navier-Stokes equations via the hydrostatic limit \cite{Azrad2001,Gr99, FGHHKW25}. \par

Understanding the well-posedness and singularity formation of fluid models is a central problem in fluid dynamics. For the viscous three-dimensional primitive equations, viscosity exerts a smoothing effect that prevents singularities, as demonstrated by Cao and Titi's proof of global existence and uniqueness of strong solutions \cite{Cao2007, K06, Hieber-Kashiwabara}. However, inviscid PEs with zero temperature-also known as the \textit{hydrostatic Euler equations}-exhibit a severe Kelvin-Helmholtz instability, causing ill-posedness for general data in the scale of Sobolev spaces. See \cite{Renardy2009} for the linear instability, \cite{HanKwan2016} for the nonlinear instability\footnote{The solution map from the Sobolev spaces to $L^2$ fails to be H\"older continuous.}, and \cite{ILT21} for the nonlinear instability in the presence of coriolis rotation. Notice that Masmoudi and Wong \cite{Masmoudi2012} proved local well-posedness in Sobolev spaces $H^s$ under a local Rayleigh condition and a zero-average velocity assumption. 
Consequently, local well-posedness is largely restricted to analytic spaces \cite{Kukavica2011, GILT}, barring specific structural assumptions or damping \cite{Masmoudi2012, CLT19, CT10}.
Given only local well-posedness or ill-posedness for the inviscid primitive equations, it is natural to ask whether finite-time singularities form. The first blow-up result was obtained by Cao et al.\ \cite{Cao2015} for a family of explicit, non-smooth, initial conditions satisfying symmetry assumptions on the tangential velocity. Independently, Wong \cite{Wong2014} proved finite-time blow-up for a class of smooth initial data satisfying explicit structural conditions without requiring symmetry.

\par

When a singularity forms, characterizing its properties, such as the blow-up rate, self-similar profiles, asymptotic expansions, and stability under perturbations, provides deeper insight into the breakdown mechanism. This quantitative description was first established on the vertical symmetry axis by Collot et al.\ \cite{Collot2023}, which serves as the starting point for the present study.\par

By assuming $Y$-independence, the three-dimensional inviscid primitive equations reduce to the two-dimensional system:
\begin{equation}\label{2d primitive intro}
    \begin{cases}
        u_t+u u_X+w u_Z+p_X=0,\\
        p_Z=0,\\
        u_X+w_Z=0,
    \end{cases}
\end{equation}
where $u$ and $w$ denote horizontal and vertical velocities, respectively, and $p$ is the pressure. We consider the strip domain $D=\set{(X,Z):-L\leq X\leq L,\, 0\leq Z\leq 1}$. Although the horizontal domain can be either periodic or unbounded ($L=\infty$; see \cite{Kukavica2011}), for simplicity we assume periodicity in $X$. The system \eqref{2d primitive intro} is supplemented with the non-penetration and periodicity boundary conditions:
\begin{equation}\label{2d primitive boundary condition}
    \begin{split}
        &w(t,X,0)=w(t,X,1)=0,\\
        &u(t,-L,Z) = u(t,L,Z), \qquad Z\in [0,1],
    \end{split}
\end{equation}
The vertical velocity $w$ and pressure $p$ can be reconstructed from $u$ via
\begin{equation}\label{w and p}
\begin{split}
    w(t,X,Z)&=-\int_0^Z u_X(t,X,\tilde{Z})\, d\tilde{Z},\\
    p_{XX}&=-\int_0^1 (u^2)_{XX}\, dZ,
\end{split}
\end{equation}
reducing \eqref{2d primitive intro} to a closed PDE for the horizontal velocity $u$, while \eqref{2d primitive boundary condition} imposes the mean-free condition
\begin{equation}\label{mean-free}
\int_0^1 u_X(t,X,Z)\,dZ =0.
\end{equation}

\subsection{Singularity formation near a stagnation point on a vertical axis} \label{subsec:physical-picture}

A primary mechanism for singularity formation in fluid models is spatial concentration near hyperbolic stagnation points at a scale that vanishes in finite time. To study this phenomenon, we consider solutions to \eqref{2d primitive intro}--\eqref{2d primitive boundary condition} that are odd in $X$:
\begin{equation}\label{x-odd}
    u(t,-X,Z)=-u(t,X,Z),
\end{equation}
a symmetry preserved by the flow. The vertical axis $\{(0,Z): 0\leq Z\leq 1\}$ then plays a distinguished role: the velocity field $(u,w)$ is purely vertical along it, and its endpoints $(0,0)$ and $(0,1)$ are stagnation points where $(u,w)$ vanishes.

Specifically, we consider solutions where $u_X(t,0,0)<0$ (attaining its axial minimum at $Z=0$) and $u_X(t,0,1)>0$, properties propagated by the flow\footnote{This follows directly from the reduced one-dimensional model \eqref{proudman-johnson-equation}; see \cite{proudman-johnson-equation}.}. Physically, this corresponds to horizontal compression at the bottom of the axis and horizontal dilation at the top. By \eqref{w and p}, this implies $w_Z (t,0,0)>0$ and $w_Z(t,0,1)<0$, making $(0,0)$ and $(0,1)$ hyperbolic stagnation points.

Anticipating our main results, the singular solutions to \eqref{2d primitive intro} described in this paper correspond to the physical geometry illustrated in Figure~\ref{fig:characteristic-rescaling}.
\begin{center}
\resizebox{0.8\textwidth}{!}{
\begin{minipage}{16.5cm}
\centering
\begin{tikzpicture}[
  x=3.3cm,y=7.3cm,
  font=\large,
  >=Stealth,
  physical/.style={floworange,line width=1.05pt,-{Stealth[length=2.4mm,width=1.7mm]}},
  characteristic/.style={curveblue,line width=1.45pt,
    postaction={decorate},
    decoration={markings,mark=at position .57 with
      {\arrow{Stealth[length=2.4mm,width=1.8mm]}}}},
  measure/.style={scalegreen,line width=1pt}
]
  \fill[scalegreen!8] (-.27,0) rectangle (.27,1);
  \draw[scalegreen!65,densely dashed,line width=.6pt]
    (-.27,0)--(-.27,1) (.27,0)--(.27,1);

  \def\eps{.035}
  \foreach \xx in {-1.14,-.73,-.32,.32,.73,1.14}{
    \foreach \zz in {.025,.075,.14,.30,.52,.74,.91}{
      \pgfmathsetmacro{\ww}{\zz*(1-\zz)/(\zz+\eps)}
      \pgfmathsetmacro{\wp}{(\eps-2*\eps*\zz-\zz*\zz)/(\zz+\eps)^2}
      \pgfmathsetmacro{\uu}{-\xx*\wp}
      \pgfmathsetmacro{\length}{sqrt(\uu*\uu+4.84*\ww*\ww)}
      \pgfmathsetmacro{\dxx}{.13*\uu/\length}
      \pgfmathsetmacro{\dzz}{.12*\ww/\length}
      \draw[physical,opacity=.9] (\xx,\zz) -- ++(\dxx,\dzz);
    }
  }
  \foreach \zz in {.055,.16,.35,.56,.78,.93}{
    \draw[physical] (0,\zz) -- ++(0,.055);
  }

  \foreach \c in {.06,.115,.17,.23,.30,.37}{
    \pgfmathsetmacro{\bb}{\c/1.47}
    \pgfmathsetmacro{\root}{sqrt((1-\bb)^2-4*\bb*\eps)}
    \pgfmathsetmacro{\zstart}{.015}
    \pgfmathsetmacro{\zend}{(1-\bb+\root)/2}
    \draw[characteristic] plot[domain=\zstart:\zend,
      samples=130,variable=\q] ({\c*(\q+\eps)/(\q*(1-\q))},\q);
    \draw[characteristic] plot[domain=\zstart:\zend,
      samples=130,variable=\q] ({-\c*(\q+\eps)/(\q*(1-\q))},\q);
  }

  \draw[figureink,line width=1.2pt] (-1.48,0)--(1.48,0);
  \draw[figureink,line width=1.2pt] (-1.48,1)--(1.48,1);
  \draw[figureink!75,line width=.95pt] (0,0)--(0,1);
  \foreach \xx in {.43,.86,1.25}{
    \draw[physical] (\xx,0)--++(-.18,0);
    \draw[physical] (-\xx,0)--++(.18,0);
    \draw[physical] (\xx,1)--++(.18,0);
    \draw[physical] (-\xx,1)--++(-.18,0);
  }
  \fill[pointcyan] (0,0) circle[radius=1.8pt];
  \fill[pointcyan] (0,1) circle[radius=1.8pt];

  \node[figureink,anchor=east] at (-1.52,0) {$Z=0$};
  \node[figureink,anchor=east] at (-1.52,1) {$Z=1$};
  \node[figureink,anchor=south] at (0,1.02) {$X=0$};

  \draw[measure,{Stealth[length=2mm]}-{Stealth[length=2mm]}]
    (-.27,-.105)--(.27,-.105);
  \node[scalegreen,anchor=north,font=\Large] at (0,-.126)
    {$(T-t)^{3/2}$};
  \draw[measure,{Stealth[length=2mm]}-{Stealth[length=2mm]}]
    (1.60,.16)--(1.60,1);
  \node[scalegreen,anchor=west,font=\Large] at (1.67,.59) {$1$};
  \draw[measure,{Stealth[length=1.9mm]}-{Stealth[length=1.9mm]}]
    (1.60,0)--(1.60,.16);
  \node[scalegreen,anchor=west,font=\Large] at (1.71,.08)
    {$\frac{1}{|\log(T-t)|}$};

  \draw[curveblue,line width=1.45pt,-{Stealth[length=2.3mm]}]
    (-1.64,-.31)--(-1.42,-.31);
  \node[figureink,anchor=west] at (-1.37,-.31) {Characteristic Lines};
  \draw[physical] (.06,-.31)--(.28,-.31);
  \node[figureink,anchor=west] at (.33,-.31) {Velocity Field (Unscaled)};
  \fill[pointcyan] (-.87,-.43) circle[radius=2.2pt];
  \node[figureink,anchor=west] at (-.78,-.43)
    {Hyperbolic Stagnation Points};

  \path (-2.43,-.54) rectangle (2.43,1.10);
\end{tikzpicture}

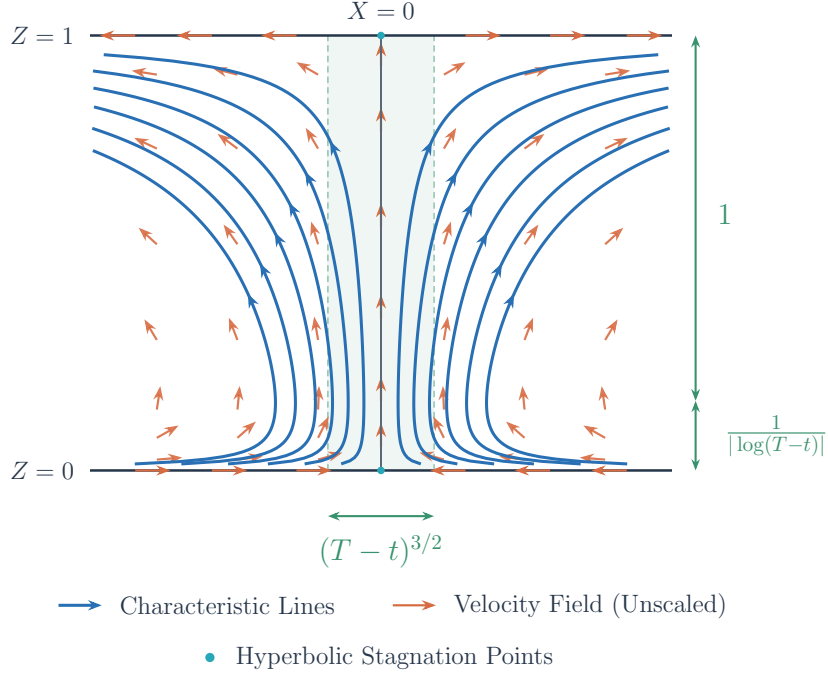
\captionof{figure}{Visualization of the blow-up schema.}
\label{fig:characteristic-rescaling}
\end{minipage}
}
\end{center}

Concentration near hyperbolic stagnation points is a recurring mechanism for small-scale creation across incompressible fluid models. In 2D Euler flows, Kiselev and \v{S}ver\'ak \cite{KiselevSverak2014} constructed smooth solutions whose vorticity gradient grows at the optimal double-exponential rate, providing a smooth dynamical realization of the hyperbolic structure in the stationary Bahouri-Chemin example \cite{BahouriChemin1994}. A related scenario occurs in the 3D axisymmetric Euler equations with boundary, where numerical computations by Luo and Hou \cite{LuoHou2014a,LuoHou2014b} indicated finite-time blow-up near a ring of hyperbolic stagnation points on the boundary. Following studies of reduced models \cite{ChoiHouKiselevLuoSverakYao2017}, Chen and Hou \cite{ChenHou2021,ChenHou2025} established rigorous finite-time blow-up in this axisymmetric setting via a computer-assisted proof based on the numerically observed nearly self-similar profile. Elgindi's finite-time blow-up construction for the 3D Euler equations \cite{Elgindi2021} is related in spirit, though it is driven by a self-similar singularity centered at the origin rather than a smooth hyperbolic stagnation point. Although these works share the theme of strong geometric compression, their precise local mechanisms remain distinct.

\subsection{Blow-up of the first derivative and vertical concentration with anomalous self-similarity}

For $X$-odd solutions \eqref{x-odd} of the system \eqref{2d primitive intro}--\eqref{2d primitive boundary condition}, the trace of the horizontal derivative along the symmetry axis, $a(t,Z)=-u_X(t,0,Z)$, satisfies a closed equation:
\begin{equation} \label{proudman-johnson-equation}
\partial_t a -a^2+\left(\int_0^Z a (\tilde Z)\,d\tilde Z\right) \partial_Z a+2 \int_0^1a^2(\tilde Z)\,d\tilde Z=0,
\end{equation}
with initial data $a(0,Z)=-u_X(0,0,Z)$. Equation \eqref{proudman-johnson-equation} is the Proudman--Johnson equation \cite{proudman-johnson-equation}, which also arises as a 1D reduction of the 2D Euler and incompressible porous medium (IPM) equations. For example, Childress et al.\ \cite{Childress1989} studied stagnation-point-form solutions to the 2D Euler equations that reduce to \eqref{proudman-johnson-equation} and exhibit finite-time singularities. The local well-posedness and blow-up of the Proudman--Johnson equation and its generalizations were further investigated in \cite{Sarria2012,Okamoto2009}. We refer also to \cite{EIS21} for a variant version of active scalar model. More recently, Collot et al.\ \cite{Collot2026} proved the asymptotic stability of a family of steady states for \eqref{proudman-johnson-equation}, resulting in a self-similar blow-up for the IPM equations.

Which singularity mechanisms of the reduced 1D model \eqref{proudman-johnson-equation} are relevant to the full 2D primitive equations? Sarria and Saxton \cite{Sarria2012} showed that for $a(0,\cdot)\in C^2([0,1])$, if the initial datum either reaches its maximum strictly inside $(0,1)$ or reaches its maximum at the boundary $Z\in\{0,1\}$ with a vanishing derivative, the corresponding solution $a$ is global. Although solutions with $C^{1,\alpha}$ regularity ($0\leq \alpha<1$) can blow up, and Cao et al.\ \cite{Cao2015} constructed non-smooth self-similar blow up profiles in this class, such profiles cannot be directly embedded into the 2D problem, which requires real-analytic data for local well-posedness. In contrast, Collot et al.\ \cite{Collot2023} constructed an open set of smooth $C^2$ initial data where $u_X(0,0,\cdot)$ attains its minimum at $Z=0$ with $u_{XZ}(0,0,0)>0$, leading to finite-time blow-up with the leading-order asymptotics
\begin{equation}\label{leading order expansion}
    u_X(t,0,Z)=-\frac{1}{T-t}\left(\phi\left(\frac{Z}{\abs{\log (T-t)}^{-1}}\right)+O\big(\abs{\log (T-t)}^{-\delta}\big)\right),
\end{equation}
where $\phi(z)=e^{-z}$ and $\delta\in (0,1)$. Here and throughout, writing $u_X(t,0,Z)=-a(t,Z)$ is a slight abuse of notation: the 1D solution $a$ to \eqref{proudman-johnson-equation} exists up to $T$, whereas the full 2D solution $u$ could a priori become singular off the axis before time $T$.

The blow-up in \eqref{leading order expansion} is governed by the profile $\phi$, exhibiting \emph{anomalous self-similarity} at the \emph{logarithmic vertical scale} $Z\sim |\ln (T-t)|^{-1}$ driven by subleading but non-negligible nonlocal pressure effects (the $\int_0^1a^2$ term in \eqref{proudman-johnson-equation}; see \cite{Collot2023}). While \eqref{leading order expansion} captures the leading-order behavior, two fundamental questions remain:
\begin{enumerate}
    \item Can the blow-up asymptotics be refined beyond leading order, and can that describe the behavior outside the self-similar zone ($Z\gg |\ln (T-t)|^{-1}$)?
    \item What are the implications for the full 2D solution $u$, and what is the characteristic horizontal concentration scale?
\end{enumerate}
Obtaining higher-order terms in the expansion of $u_X(t,0,Z)$ is crucial for two reasons. First, the subleading corrections of order $|\ln (T-t)|^{-1}\ln |\ln (T-t)|$ and $|\ln (T-t)|^{-1}$ exert a non-negligible influence on the higher horizontal derivatives governing the 2D problem. Second, \eqref{leading order expansion} is too coarse to resolve the upper region $Z\gg |\ln (T-t)|^{-1}$, where $u_X$ becomes positive and the flow is outgoing from the axis.

Our first theorem refines the analysis of \cite{Collot2023} to capture these subleading terms.

\begin{theorem}\label{refined C1 expansion}
There exists an open set of initial data $a_0=-u_X(0,0,\cdot)$ in $C^2([0,1])$ such that the corresponding solution\footnote{As noted above, this theorem concerns the 1D equation \eqref{proudman-johnson-equation}; the identification $a=-u_X(\cdot,0,\cdot)$ holds as long as the 2D solution $u$ exists.} $a=-u_{X}(t,0,Z)$ to \eqref{proudman-johnson-equation} blows up at a finite time $T>0$ with the expansion
    \begin{equation} \label{id:expansion-uX}
    \begin{split}
          u_X(t,0,Z)&=-\frac{1}{T-t}\left(\phi\left(\frac{Z}{\abs{\log (T-t)}^{-1}}\right)+\frac{\log\abs{\log (T-t)} }{\abs{\log (T-t)}}U_{0,\log}\left(\frac{Z}{\abs{\log (T-t)}^{-1}}\right)\right.
          \\&\qquad\qquad \quad  \,\, + \left.\frac{1}{\abs{\log (T-t)}} U_0\left(\frac{Z}{\abs{\log (T-t)}^{-1}}\right)\right)+ R_0(t,Z),
    \end{split}
    \end{equation}
    where $\phi(z)=e^{-z}$, the subleading profiles $U_{0,\log}$ and $U_0$ are given by
    \begin{equation}
            U_{0,\log}(z) =-2z e^{-z}, \qquad U_0(z)=(2+(1+C_\nu)z) e^{-z}-1,
    \end{equation}
    for some constant $C_\nu(a_0)\in \mathbb R$, and the remainder satisfies, for $k=0,1$,
    $$
    \| \partial_Z^k R_0(t,\cdot)\|_{L^\infty ([0,1])} \lesssim  \frac{(\log\abs{\log (T-t)} )^2}{(T-t)(\abs{\log (T-t)})^{2-k}}.
    $$
\end{theorem}

\begin{remark}
Compared to \cite{Collot2023}, Theorem \ref{refined C1 expansion} reveals that above the boundary layer ($Z\gg |\ln (T-t)|^{-1}$), $u_X$ approaches the positive, $Z$-independent value $(T-t)^{-1}|\ln (T-t)|^{-1}$. Consequently, $\lim_{t\uparrow T} u_X(t,0,Z)=\infty$ for every fixed $Z\in (0,1]$, and at the top boundary $u_X(t,0,1)\sim (T-t)^{-1}|\ln (T-t)|^{-1}>0$, confirming the physical picture in Subsection \ref{subsec:physical-picture}.
\end{remark}

\subsection{Blow-up of the third derivative and appearance of the Burgers horizontal scaling}

Having established the vertical scale $|\ln (T-t)|^{-1}$, we turn to the horizontal scale, which is formally comparable to $\sqrt{|u_X /u_{XXX}|}$. For an $X$-odd solution \eqref{x-odd} to \eqref{2d primitive intro}, the trace of the third horizontal derivative on the axis, $b(t,Z)=u_{XXX}(t,0,Z)$, satisfies a closed linear equation coupled to $a(t,Z)=-u_X(t,0,Z)$:
\begin{equation} \label{id:equation-uXXX}
b_t-4ab+\left(\int_0^Z a(\tilde Z)\,d\tilde Z\right)\partial_Z b+3 \partial_Z a \left(\int_0^Z b(\tilde Z)\,d\tilde Z\right)+8 \int_0^1 a(\tilde Z)b(\tilde Z)\,d\tilde Z=0.
\end{equation}
When $a$ satisfies Theorem \ref{refined C1 expansion}, our second theorem shows that $b=u_{XXX}(\cdot,0,\cdot)$ also blows up and provides its asymptotic expansion up to logarithmically subleading terms. Remarkably, unlike the vertical scale, subleading pressure effects do not induce logarithmic corrections in the horizontal scale, which instead exhibits the pure algebraic scaling $(T-t)^{3/2}$ characteristic of generic Burgers singularities.

\begin{theorem}\label{first shell mode expansions}
Let $a=-u_X(\cdot,0,\cdot)$ be a solution\footnote{Again, this result concerns the 1D system \eqref{proudman-johnson-equation} and \eqref{id:equation-uXXX}, matching the 2D trace as long as $u$ exists.} to \eqref{proudman-johnson-equation} satisfying Theorem \ref{refined C1 expansion} with blow-up time $T>0$. Then for any initial datum $b_0=u_{XXX}(0,0,\cdot)\in C^1([0,1])$, the solution $b=u_{XXX}(\cdot,0,\cdot)$ to \eqref{id:equation-uXXX} exists on $[0,T)$, and there exists $f_\infty\in \mathbb R$ such that
    \begin{align*}
        u_{XXX}(t,0,Z)&=\frac{1}{(T-t)^4}\left( c(t)\phi \left(\frac{Z}{|\ln (T-t)|^{-1}}\right)+\frac{\ln |\ln (T-t)|}{|\ln (T-t)|}\Phi_{1,0}\left(\frac{Z}{|\ln (T-t)|^{-1}}\right) \right. \\
        & \qquad \qquad \qquad \qquad +\left.\frac{1}{|\ln (T-t)|} \Psi_{1,0}\left(\frac{Z}{|\ln (T-t)|^{-1}}\right) \right)+R_{1}(t,Z),
    \end{align*}
    where $c(t)=(T-t)^4u_{XXX}(t,0,0)=f_\infty+O(\ln |\ln(T-t)|/|\ln (T-t)|)$, $\phi(z)=e^{-z}$, and
    $$
     \Phi_{1,0}(z)=-2 f_\infty z e^{-z}, \qquad \Psi_{1,0}(z)=f_\infty\left(-1+C_\nu ze^{-z}-\frac{1}{2} e^{-z}+e^{-2z}+\frac{1}{2}e^{-3z}\right),
    $$
    with $C_\nu\in \R$ as in Theorem \ref{refined C1 expansion}, and the remainder satisfies
    $$
    \| R_1(t,\cdot)\|_{L^\infty ([0,1])} \lesssim  \frac{|\log\abs{\log (T-t)} |^{3} }{(T-t)^4(\abs{\log (T-t)})^{2}}.
    $$
\end{theorem}

\begin{remark}
\begin{itemize}
\item\emph{(i) Appearance of the Burgers horizontal scale.} Combining Theorems \ref{refined C1 expansion} and \ref{first shell mode expansions} shows that, for $f_\infty\neq 0$, the characteristic horizontal scale is $\sqrt{|u_X /u_{XXX}|}\sim (T-t)^{3/2}$. This matches the generic blow-up scale of the inviscid Burgers equation $u_t+uu_X=0$, governed by the self-similar solution $u(t,x)=(T-t)^{1/2}\Psi_1(x/(T-t)^{3/2})$ where $\Psi_1^{-1}(y)= -y-y^3$ (see \cite{COLLOT2022_burgers}). {Note that for an odd solution $g$ to the Burgers equation, the condition $g'''(0,0)>0$ is required for the first singularity to occur at the origin. By analogy, it could be possible that $f_\infty>0$ is required for $T$ to be the blow-up time of the original solution $u$ to \eqref{2d primitive intro}, while for $f_\infty<0$ it could become singular off the axis before.}

\smallskip
\item\noindent \emph{(ii) Cancellation of logarithmic corrections.} At first glance, Theorem \ref{refined C1 expansion} suggests that the horizontal scale should acquire logarithmic corrections in $(T-t)$. Indeed, substituting the asymptotics of $a$ into \eqref{id:equation-uXXX} suggests the toy ODE $b_t=4(T-t)^{-1}(1+\alpha |\ln (T-t)|^{-1})b$, which yields $b\sim c (T-t)^{-4}|\ln (T-t)|^{4\alpha}$. The absence of logarithmic factors is due to a surprising cancellation: once properly renormalized, the leading-order dynamics of \eqref{id:equation-uXXX} admits a one-dimensional kernel on which the subleading logarithmic corrections from $u_X$ in \eqref{id:expansion-uX} exactly cancel the nonlocal pressure contribution ($8\int_0^1 ab\,d\tilde Z$ in \eqref{id:equation-uXXX}); see Section \ref{sec:first order mode} for details.
{\smallskip
\item\noindent \emph{(iii)} For an initial datum in the admissible class, its corresponding $f_\infty$ is generically non-zero. (See Remark \ref{rmk:nonzero limit}.)}
\end{itemize}
\end{remark}

\subsection{Asymptotic expansion for higher-order derivatives}

To obtain refined expansions and extend the solution to 2D, we study the full hierarchy of axial derivatives $(u_{i,k})_{i, k\geq 0}$, defined by $u_{i,k}(t,Z)=\partial_X^{2i+1}\partial_Z^k u(t,0,Z)$. These derivatives satisfy the infinite system of PDEs:
\begin{equation} \label{system-hierarchy-original-variables}
\begin{aligned}
\partial_t u_{i,k}
&+\sum_{p=0}^{k}\sum_{q=0}^{i}
  \binom{k}{p}\binom{2i+1}{2q+1}
  u_{q,p}\,u_{i-q,k-p}-\sum_{p=0}^{k}\sum_{q=0}^{i}
  \binom{k}{p}\binom{2i+1}{2q}
  u_{q,p-1}\,u_{i-q,k-p+1}\\
&-\mathbf{1}_{\{k=0\}}
  \sum_{p=0}^{i}\binom{2i+2}{2p+1}
  \int_0^1 u_{p,0}\,u_{i-p,0}\,dZ
=0,
\end{aligned}
\end{equation}
where $u_{\cdot,-1}=\int_0^Z u_{\cdot,0}\,d\tilde Z$ and $\partial_{Z}u_{i,k}=u_{i,k+1}$.
Notice that the modes in the $N$-th shell ($i+k=N$) depend only on lower shells ($i+k\leq N$), allowing the hierarchy to be solved inductively. For $(i,k)=(0,0)$ and $(1,0)$, \eqref{system-hierarchy-original-variables} reduces to \eqref{proudman-johnson-equation} and \eqref{id:equation-uXXX} (up to a sign). Theorem \ref{refined C1 expansion} characterizes $u_{0,0}$ and $u_{0,1}$, while Theorem \ref{first shell mode expansions} characterizes $u_{1,0}$. Much like the lowest-order terms in a Taylor expansion, the zeroth and first shells ($i+k\leq 1$) govern the dynamics of all higher shells ($i+k\geq 2$), making their precise characterization essential.

Due to space constraints, although we do not include the proof in this paper which can be found in \cite{lingjun}, the same method can be applied to obtain the following result, describing the asymptotic expansions for each fixed function of the hierarchy.

\begin{remark}\label{general shell mode expansions}
Let $N>1$, and let $a=-u_{0,0}=-u_X(0,0,\cdot)$ and $b=u_{1,0}=u_{XXX}(0,0,\cdot)$ be solutions\footnote{As before, this theorem concerns the 1D system \eqref{system-hierarchy-original-variables}.} to \eqref{proudman-johnson-equation} and \eqref{id:equation-uXXX} satisfying the conclusions of Theorems \ref{refined C1 expansion} and \ref{first shell mode expansions} with blow-up time $T>0$ and parameters $C_\nu,f_\infty \in \mathbb R$, assuming further that their initial data belong to $C^{N+1}([0,1])$ and $C^{N}([0,1])$, respectively. Then it can be proved, see \cite{lingjun}, that for any initial data $u_{i,0,0}=\partial_{X}^{2i+1}u(0,0,\cdot)\in C^{N+1-i}([0,1])$ with $1<i\leq N$, the solution $(u_{i,k})_{i+k\leq N}$ to \eqref{system-hierarchy-original-variables} with initial data $u_{i,k}(0,\cdot)=\partial_{Z}^k u_{i,0,0}$ exists on $[0,T)$ and satisfies
\begin{align*}
\partial_X^{2i+1} u(t,0,Z) & = \frac{1}{(T-t)^{1+3i}}\left( (1+\alpha_i(t))\phi_i\left(\frac{Z}{|\ln (T-t)|^{-1}}\right)+\frac{\ln |\ln (T-t)|}{|\ln (T-t)|}U_{i,\log} \left(\frac{Z}{|\ln (T-t)|^{-1}}\right)\right. \\
&\qquad \qquad \qquad \left.+\frac{1}{|\ln (T-t)|}U_{i} \left(\frac{Z}{|\ln (T-t)|^{-1}}\right) \right)+R_i(t,Z),
\end{align*}
where the profiles $\phi_i,U_{i,\log},U_{i}\in C^\infty([0,\infty)) \cap L^\infty([0,\infty))$ depend on $f_\infty$ and $C_\nu$, and the remainders satisfy, for $0\leq k \leq N-i$,
    \begin{equation}
        \abs{\alpha_i(t)}=O\left(\frac{\ln |\ln (T-t)|}{|\ln (T-t)|}\right),\qquad \norm{\partial_Z^k R_{i}(t,\cdot)}_{L^\infty([0,1])} =O\left(\frac{\ln^3|\ln (T-t)|}{(T-t)^{1+3i}|\ln (T-t)|^{2-k}}\right).
    \end{equation}
\end{remark}

Remark \ref{general shell mode expansions} extends our asymptotic expansions to any finite truncation $(\partial_X^{2i+1}\partial_Z^k u(t,0,Z))_{0\leq i+k\leq N}$ of the hierarchy. However, because the implicit constants may grow uncontrollably with $i$ and $k$, this finite-shell expansion does not by itself yield convergence of the Taylor series in $X$. To reconstruct the 2D solution, we must establish uniform bounds across the entire infinite hierarchy via a suitable real-analytic norm.

\subsection{Rigorous two-dimensional extension and lower bound on the analyticity radius}

In a related study of the 2D unsteady Prandtl equations, Collot et al.\ \cite{Collot2022} showed that the axial trace of the horizontal derivative blows up in finite time $T$ and controlled the infinite hierarchy of axial derivatives in an analytic norm, proving that the 2D solution remains real-analytic up to time $T$ in a neighborhood of the symmetry axis of width at least $(T-t)^{7/4}$. Recall that for the inviscid primitive equations, local well-posedness requires real-analytic initial data \cite{Kukavica2011}.\par

Our fourth theorem establishes an analogous 2D extension for the inviscid primitive equations. While our approach shares the overarching philosophy of \cite{Collot2022}-providing, to our knowledge, the only other quantitative stability analysis of singularity formation in an analytic setting-the underlying mechanisms differ substantially due to the nonlocal pressure in the primitive equations versus vertical viscosity in the Prandtl system. Consequently, we uncover a different horizontal analyticity radius scaling of order $(T-t)^{3/2}$.

\begin{theorem} \label{thm:physical-x-radius_intro}
Let $a$ and $b$ satisfy the conclusions of Theorems \ref{refined C1 expansion} and \ref{first shell mode expansions}, and fix $t_0$ sufficiently close to $T$, setting $a_0=a(t_0)$ and $b_0=b(t_0)$. Let $u_0$ be real-analytic on $\mathbb R\times [0,1]$ and $2L$-periodic, $X$-odd \eqref{x-odd}, and satisfy the mean-free condition \eqref{mean-free}, with axial traces $\partial_X u_0(0,\cdot)=-a_0$ and $\partial_{X}^3 u_{0}(0,\cdot)=b_0$. The following conclusions hold:
\begin{enumerate}[label=\textup{(\roman*)},leftmargin=2.3em]
\item \emph{Physical derivative bounds.}
The solution\footnote{As before, this concerns the system \eqref{system-hierarchy-original-variables}, coinciding with the 2D derivatives while $u$ exists.} $(u_{i,k})_{i,k\geq 0}$ to \eqref{system-hierarchy-original-variables} with initial data $u_{i,k}(0,\cdot)=\partial_X^{2i+1}\partial_Z^k u_0(0,\cdot)$ exists on $[0,T)$. Moreover, there exist constants $M>0$ and $\rho>0$ such that for all $t\in[0,T)$, $Z\in [0,1]$, and $i,k\geq 0$,
\begin{equation}\label{eq:physical-derivative-conclusion_intro}
\left|\partial_X^{2i+1}\partial_Z^k u(t,0,Z)\right| \leq M\rho^{-2i-k-1} (2i+1)! (T-t)^{-3i-1}|\ln (T-t)|^{k}.
\end{equation}
\item \emph{Analytic extension in \(X\).} The Taylor series
\begin{equation}\label{eq:u-taylor-series_intro}
  \widetilde u(t,X,Z)
  :=\sum_{i=0}^{\infty}
  \frac{u_{i,0}(t,Z)}{(2i+1)!}X^{2i+1}
\end{equation}
has a radius of convergence of at least
\begin{equation}\label{eq:physical-radius_intro}
  R_X(t):=(T-t)^{3/2}\rho,
\end{equation}
and defines a solution $\widetilde u$ to the 2D primitive equations \eqref{2d primitive intro} on the domain $\{(t,X,Z): 0\leq t<T, \, |X|<R_X(t), \, 0\leq Z\leq 1\}$. Furthermore, $\widetilde u$ provides the unique real-analytic continuation of $u(t,\cdot,Z)$ in $|X|<R_X(t)$, satisfying
\begin{equation}\label{eq:physical-reconstruction_intro}
  \widetilde u(t,X,Z)=u(t,X,Z)
  \qquad\text{for }|X|<R_X(t)
\end{equation}
for as long as the 2D solution $u$ to \eqref{2d primitive intro} exists.
\end{enumerate}
\end{theorem}

Part (i) of Theorem \ref{thm:physical-x-radius_intro} complements the finite-shell asymptotic expansions by establishing uniform Gevrey/analytic bounds across the entire hierarchy, demonstrating the local persistence of analyticity up to the blow-up time $T$ at the expense of detailed profile expansions. Part (ii) shows that the full 2D solution $u$ can be rigorously reconstructed from its axial derivatives inside a shrinking strip of width $|X|< \rho(T-t)^{3/2}$, and that the local solution $\tilde u$ rigorously displays the physical blow-up geometry depicted in Subsection \ref{subsec:physical-picture}.

\subsection{Aspects of the proof and methodology}

The proofs of our asymptotic expansion results (Theorems \ref{refined C1 expansion}, \ref{first shell mode expansions}, and Remark \ref{general shell mode expansions}) share a common framework inspired by \cite{Collot2022}. We work with the hierarchy of renormalized axial derivatives $(F_{i,k})$ (defined in \eqref{setting-renormalization}) in fixed self-similar coordinates $(s,z)$ (defined in \eqref{self-similar-coordinates}), where $s=-\log(T-t)$ and $z=Z/|\ln(T-t)|^{-1}=sZ$, factoring out the expected leading-order algebraic blow-up rates. Here, $i$ indexes the $(2i+1)$-th horizontal derivative and $k$ indexes the $k$-th vertical derivative. A central theme is to decompose each mode into a neutral component spanning the kernel of the leading-order linear operator and a remainder that vanishes at the origin $z=0$ (which fixes the modulation parameter), and then control the remainder via the decaying dynamics on the complementary stable subspace. In self-similar coordinates, the remainder PDEs take the form of a linear evolution driven by a ``transport + nonlocal'' operator plus nonlinear source terms.\par

To bound the remainders, we first establish decay estimates for the linear evolution in a carefully designed weighted $L^\infty$ norm of the form $\norm{(\cdot)(B Y_\theta)^{-1}}_{L^\infty}$. Thanks to the mean-free condition \eqref{mean-free}, we can construct a supersolution and apply the maximum principle to show exponential decay for the pure transport evolution. Here, the weight $B(z)\geq 0$ vanishes sufficiently fast at $z=0$ to exploit the remainder's vanishing condition at the origin, ensuring strong damping from the transport operator. Meanwhile, the weight $Y_\theta(z)\approx 1$ is bounded away from zero but grows rapidly near the origin for large $\theta>0$, making the nonlocal operator's contribution arbitrarily small in the weighted norm. Combining these ingredients via Duhamel's formula and Gronwall's lemma yields exponential decay for the full linear operator.\par

A key distinction in the proof of Theorem \ref{refined C1 expansion} is that, rather than working directly in the fixed self-similar coordinates $(s,z)$ with the remainder $\eps$ from \eqref{leading order expansion}, we perform the bulk of the analysis in the dynamical self-similar coordinates $(\tilde s, \tilde z)$ of \cite{Collot2023}. In these coordinates, $-u_X(t,0,Z)$ is decomposed via a crude remainder $\tilde \eps(\tilde s,\tilde z)$ linked to physical coordinates through dynamical modulation parameters $(\lambda,\nu)$. Imposing higher-order vanishing conditions at the origin determines the modulation dynamics and aligns $\tilde \eps$ with our weighted linear theory. Specifically, we first identify the leading-order profile of $\tilde\eps$ matching its behavior near the origin, and then peel off this profile to obtain a finer remainder, which we bound using the weighted linear evolution method. Finally, translating from $(\tilde s, \tilde z)$ back to the fixed self-similar coordinates $(s,z)$ via the asymptotic expansions of the modulation parameters $(\lambda,\nu)$ yields the expansion for $\eps$. Working in dynamical coordinates is essential because the leading subleading timescale turns out to be $\frac{\log s}{s}$ rather than $\frac{1}{s}$, a logarithmic correction generated dynamically by the modulation parameters.\par 

To verify that our horizontal renormalization scale is optimal, Theorem \ref{first shell mode expansions} shows that the renormalized first-shell mode $F_{1,0}$ (corresponding to $u_{XXX}\rvert_{X=0}$) converge to non-trivial O(1) profiles. To avoid the tedious algebra of dynamical coordinates $(\tilde s,\tilde z)$, we enforce the vanishing-at-zero condition via a two-step decomposition in fixed self-similar coordinates $(s,z)$. First, noting that the leading-order profile of $F_{1,0}$ is a multiple of $\phi$, we decompose $F_{1,0}(s,z) = F_{1,0}(s,0)\phi(z) + \tilde r(s,z)$ so that $\tilde r(s,0)=0$, yielding a coupled system for the trace $F_{1,0}(s,0)$ and the remainder $\tilde r$. We then construct a combined energy functional $H(s)$ that simultaneously controls the boundedness of $F_{1,0}(s,0)$ and the decay of $\tilde r$. Crucially, the expansion of $\eps(s,0)$ from Theorem \ref{refined C1 expansion} triggers an exact cancellation in the linear $s^{-1}$ coefficient of the ODE for $F_{1,0}(s,0)$, rendering the perturbation integrable in $s$. Once the convergence of $F_{1,0}(s,0)$ and the coarse decay of $\tilde r$ are established, the second step peels off the explicit subleading profiles at scales $\frac{\log s}{s}$ and $\frac{1}{s}$ (both vanishing at $z=0$) to isolate a finer remainder $r_{1,0}$. Because $r_{1,0}$ satisfies a single uncoupled PDE, the linear evolution method directly proves that it decays strictly faster than $\frac{1}{s}$. \par

The generalized expansion in Remark \ref{general shell mode expansions} is proved \cite{lingjun} via a nested induction on the shell index $N=i+k$, exploiting the lower-triangular structure of the hierarchy \eqref{system-hierarchy-original-variables} with Theorem \ref{first shell mode expansions} serving as the base case. For each shell the analysis extends that of the proof of Theorems \ref{refined C1 expansion} and \ref{first shell mode expansions}: we determine the leading and subleading profiles, decompose the modes, and bound a shell-specific energy functional.

\par

In contrast, the proof of the 2D analytic extension (Theorem \ref{thm:physical-x-radius_intro}) requires a fundamentally different approach: because the Duhamel--Gronwall method loses uniformity as $i+k\to\infty$, we must establish uniform bounds across all shells simultaneously. Treating the zeroth and first shells (Theorems \ref{refined C1 expansion} and \ref{first shell mode expansions}) as given background coefficients, we induct on shells $i+k\geq 2$ for the normalized variables
$$
U_{i,k} = \frac{(1+i+k)^{a+1}}{(2i+k+1)!} \rho^{2i+k-1} h_k^{-1} F_{i,k},
$$
where $a=3$ is a Sobolev-type polynomial correction, $\rho(s)$ is the time-dependent analyticity radius in self-similar coordinates, and $h_k(s,z)$ is a spatiotemporal weight tailored to the profiles found in Theorem \ref{refined C1 expansion} and Theorem \ref{first shell mode expansions}. Specifically, since for these two modes, they decay like $e^{-z}$ for $z=O(1)$ but plateau at order $\frac{1}{s}$ (for $k=0$) or $(\frac{\log s}{s})^2$ (for $k\geq 1$) in the far field $z\sim s$. As we expect this behavior to persist for all modes, the weight $h_k$ is then constructed as a sum interpolating between these two spatial regimes. Since our horizontal renormalization is optimal, $\rho(s)$ remains bounded below by a strictly positive constant as $s\to\infty$, while allowing $\rho(s)$ to decrease slowly provides crucial extra damping in the evolution equations for $U_{i,k}$.\par

Methodologically, the uniform control of $U_{i,k}$ relies on two pillars: a modified maximum/comparison principle and combinatorial nonlinear source estimates. For each mode $U_{i,k}$, the comparison principle shows that if the initial datum is bounded and the nonlinear source $S_{i,k}$ is small relative to the linear damping rate, then $U_{i,k}$ remains uniformly bounded for all time. For the $k=0$ modes, this maximum principle must be modified to account for the nonlocal operator, which competes with the local linear damping; we prove that the net damping rate (local damping minus the worst-case nonlocal contribution) remains uniformly positive. Finally, assuming the inductive bound holds for lower shells, we use combinatorial estimates on the quadratic interactions in $S_{i,k}$ to show that choosing the analyticity radius scale $\rho$ sufficiently small makes the nonlinear source arbitrarily small relative to the linear damping, thereby closing the induction.

\medskip

\noindent\textbf{Organization of the paper.} In Section 2, we introduce the self-similar renormalization, derive the hierarchy of equations for $(F_{i,k})$, and outline the main notation. Section \ref{sec: asymptotic expansion of eps} establishes the refined subleading asymptotics of $u_X\rvert_{X=0}$, proving Theorem \ref{refined C1 expansion}. Section \ref{sec:first order mode} derives the asymptotic expansion of the renormalized third horizontal derivative $F_{1,0}$, proving Theorem \ref{first shell mode expansions}.
Finally, Section \ref{sec: analytical extension} establishes the uniform analytic bounds and 2D Taylor reconstruction, proving Theorem \ref{thm:physical-x-radius_intro}.

\section{Notation}

{The profile for the first order tangential derivative on the axis is
\[
  \phi(z)=e^{-z}.
\]
We use the following notation for the antiderivative that vanishes at the origin,
$$
  \pzi f(z)=\int_0^z f(y)\dd y.
$$
{For notational simplicity, we adopt the convention
\begin{equation}
    \pzi f g:= (\pzi f) g.
\end{equation}
}
Derivatives will sometimes be denoted by subscripts to ease notation
$$
f_x=\partial_x f.
$$
We will use indicator functions 
$$
\one_{k=0}=\left\{ \begin{array}{l l} 1 & \mbox{if }k=0,\\ 0 & \mbox{if }k\geq 1, \end{array} \right.  \quad \mbox{and} \quad \one_{i\geq 1}=\left\{ \begin{array}{l l} 1 & \mbox{if }i\geq 1,\\ 0 & \mbox{if }i=0, \end{array} \right. 
$$
and straightforward adaptations of this notation.}

{We introduce $\chi$ a smooth positive cut-off function on $[0,\infty)$ such that
\begin{equation} \label{def-cut-off}
\chi(y)=1 \mbox{ for }y\leq 1 \quad \mbox{and} \quad \chi(y)=0 \mbox{ for }y\geq 2.
\end{equation}
}
{For simplicity, we also introduce  $L_\theta$ as
\begin{equation}\label{def-L_theta}
     L_\theta =2\log(\theta).
\end{equation}
Together with $\chi$ defined in \eqref{def-cut-off}, they will be used to construct $L^\infty$ weights, as mentioned in introduction.
}

{For $\tau_1\leq\tau_2$, let
$\mathscr S_{\mathcal T}(\tau_1,\tau_2)$ denote the linear evolution
operator associated with $\mathcal T$. Given data $f$ at time $\tau_1$,
the function $\mathscr S_{\mathcal T}(\tau_1,\tau_2)f$ is the solution
at time $\tau_2$ of
\(\partial_\tau v=\mathcal T v,\, v(\tau_1)=f.\)}

{Throughout the paper, $C>0$ denotes a universal constant whose value may vary from line to line. More generally, $C_{\alpha_1,\ldots,\alpha_m}>0$ denotes a constant depending only on the parameters $\alpha_1,\ldots,\alpha_m$, whose value may likewise vary from line to line. We write $A\lesssim B$ if $A\leq CB$ for some universal constant $C>0$, and $A\lesssim_{\alpha_1,\ldots,\alpha_m} B$ if $A\leq C_{\alpha_1,\ldots,\alpha_m}B$.}

\section{Renormalized 2D primitive equations in self-similar coordinates}\label{sec:equations and goal}

\subsection{The self-similar renormalization}

The starting point of the present article is the finite-time blow-up construction for $X$-odd solutions $u(t,X,Z)$ to the 2D primitive equations \eqref{2d primitive intro} established in \cite[Theorem 1.1]{Collot2023}. Specifically, \cite{Collot2023} showed that the axial trace of the horizontal derivative, $-u_X(t,0,Z)$, blows up at time $T$ at rate $(T-t)^{-1}$ and concentrates near $(X,Z)=(0,0)$ at the vertical scale $Z\sim \vert{}\ln (T-t)\vert{}^{-1}$. Anticipating Theorem \ref{first shell mode expansions}, which shows that $u_{XXX}(t,0,Z)$ blows up at rate $(T-t)^{-4}$, the characteristic horizontal scale is of order $(T-t)^{3/2}$. We therefore introduce the self-similar variables
\begin{equation} \label{self-similar-coordinates}
s=-\log(T-t), \quad z=\frac{Z}{\abs{\log(T-t)}^{-1}}, \quad x=\frac{X}{(T-t)^{3/2}}, \quad U(s,x,z)=(T-t)^{-1/2}u(t,X,Z).
\end{equation}
In these coordinates, the 2D inviscid primitive equations \eqref{2d primitive intro} take the form
\begin{equation}\label{2d primitive boundary condition-renormalized}
\begin{split}
& U_s-\frac{1}{2} U+\left(U+\frac{3}{2} x\right)U_x+\left(-\partial_z^{-1} U_x+\frac{z}{s}\right) U_z+P_x=0,\\
& P_z=0, \quad P_{xx}=-\frac{1}{s}\int_0^s (U^2)_{xx} \,dz,
\end{split}
\end{equation}
on the domain $(x,z)\in [-e^{3s/2}L, e^{3s/2}L]\times [0,s]$, where $P(s,x)=(T-t)^{-1}p(t,X)$, supplemented by the periodic boundary and mean-free conditions
$$ U(s,-e^{3s/2}L,z) =  U(s,e^{3s/2}L,z) \quad \mbox{and} \quad \int_0^s U_x(s,x,z)\,dz=0.$$

\subsection{Trace of derivatives on the axis}

For this $X$-odd blow-up solution, we first derive the governing equations for the higher-order axial derivatives of $u(t,0,Z)$ in the self-similar coordinates $(s,z)$, renormalized by their expected blow-up rates:
\begin{equation}\label{setting-renormalization}
\begin{split}
     &\partial_X^{2i+1} \partial_Z^{k} u(t,0,Z)=(T-t)^{-3i-1}\abs{\log(T-t)}^{k} F_{i,k}(s,z),\\
     &F_{0,0}(s,z)=(T-t)\partial_X u(t,0,Z)=-\phi(z)-\eps(s,z),
\end{split}
\end{equation}
where $\phi(z)=e^{-z}$, and the remainder satisfies $\norm{\eps}_{L^\infty}\lesssim s^{-\delta}$ for some $\delta>0$ by \cite{Collot2023}. Equivalently, $F_{i,k}(s,z)= \partial_x^{2i+1} \partial_z^k U(s,0,z)$, with $F_{0,0}(s,z)=\partial_{x} U(s,0,z)$.
\subsection{The system of partial differential equations for the hierarchy $(F_{i,k})_{i,k\geq 0}$}

The renormalized primitive equations \eqref{2d primitive boundary condition-renormalized} for $x$-odd solutions, restricted on the axis $\{x=0, \ 0\leq z \leq s\}$, take the form of the following system (below $i+k\geq 1$)
\begin{equation}\label{eq:raw-hierarchy}
\partial_sF_{i,k}
=
\mathcal T_{i,k}F_{i,k}+N_{i,k}
-\one_{k=0}(1+\one_{i\ge1})(2i+2)s^{-1}
\int_0^s(\phi+\eps)F_{i,0}\dd z, 
\end{equation}
where the linear operator satisfies
\begin{equation}\label{eq:Tik}
\mathcal T_{i,k}F_{i,k}
=
L_{i,k}F_{i,k}
+
\one_{i\ge1}\one_{k=0}(2i+1)(\phi-\eps_z)\pzi F_{i,0},
\end{equation}
\begin{equation}\label{eq:Lik}
\begin{aligned}
L_{i,k}F_{i,k}
={}&
\left(-\frac{k}{s}-3i-1+(2i+2-k)(\phi+\eps)\right)F_{i,k}
\\
&-
\left(\pzi(\phi+\eps)+\frac{z}s\right)\partial_zF_{i,k},
\end{aligned}
\end{equation}
and the nonlinear source is defined as
\begin{equation}\label{eq:Nik}
  \begin{aligned}
N_{i,k}
={}&
\one_{i\ge1}\one_{k\ge1}(2i+1)\pzi F_{i,0}\,F_{0,k+1}
+
\one_{i\ge2}\sum_{q=1}^{i-1}\binom{2i+1}{2q}\pzi F_{q,0}\,F_{i-q,k+1}
\\
&-
\sum_{\substack{0\le p\le k,\ 0\le q\le i\\(q,p)\ne(0,0),(i,k)}}
\binom{k}{p}\binom{2i+1}{2q+1}F_{q,p}F_{i-q,k-p}
\\
&+
\one_{k\ge1}
\sum_{\substack{0\le p\le k-1,\ 0\le q\le i\\(q,p)\ne(0,0)}}
\binom{k}{p+1}\binom{2i+1}{2q}F_{q,p}F_{i-q,k-p}
\\
&+
\one_{k=0}\one_{i\ge2}
\sum_{q=1}^{i-1}\binom{2i+2}{2q+1}s^{-1}
\int_0^sF_{q,0}F_{i-q,0}\dd z.
\end{aligned} 
\end{equation}
{
Here, the spatial domain becomes $0\le z\le s$ and 
System \eqref{eq:raw-hierarchy} is subject to the average-free compatibility conditions
\begin{equation}\label{eq:background-mass-compat}
  \pzi(\phi+\eps)(s,s)
  =\int_0^s(\phi+\eps)(s,y)\dd y=0 \quad \mbox{and} \quad \int_0^s F_{i,0}(s,z)dz=0 \mbox{ for }i\geq 1,
\end{equation}
}
and it inherits 
\begin{equation}\label{eq: derivative-compatability}
    \partial_z F_{i,k}=F_{i,k+1}.
\end{equation}
Moreover, the transport coefficient
satisfies
\begin{equation}\label{eq:transport-tangent-boundary}
  b(s,z):=\pzi(\phi+\eps)(s,z)+\frac{z}s,\quad b(s,0)=0,
  \qquad
  b(s,s)=1.
\end{equation}

\section{The first order tangential derivative on the axis} \label{sec: asymptotic expansion of eps}

\subsection{Setting in the dynamical self-similar coordinate system}

{
Recall from the introduction that, for odd in $X$ solutions to the $2D$ primitive equations \eqref{2d primitive intro}, the trace of the tangential derivative on the axis $-u_X(t,X=0,Z)$ solves the Proudman-Johnson equation 
\begin{equation} \label{proudman-johnson-equation-2}
a_t-a^2+\left(\int_0^Z a(\tilde Z)d\tilde Z\right)\partial_Z a +2 \int_0^1 a^2(\tilde Z)d\tilde Z=0
\end{equation}
on the set $\{0\leq Z\leq 1\}$, supplemented with the mean-free condition $\int_0^1 a(Z)dZ=0$. While the result of \cite{Collot2023} provides an open set of initial data in $C^2([0,1])$ for \eqref{proudman-johnson-equation-2} whose solutions blow up with the leading order form of $(T-t)^{-1}\phi(|\ln (T-t)|Z)$, the estimates they provide for the remainder do not permit the analysis for the hierarchy $F_{i,k}$. This is because the next terms in the expansion do not decay fast enough (like $(T-t)^{-1}|\ln (T-t)|^{-1}$ and $(T-t)^{-1}|\ln \ln (T-t)||\ln (T-t)|^{-1}$ which are non-integrable functions on $[0,T]$) and our first task is therefore to single them out.
}

{In \cite{Collot2023}, instead of the fixed self-similar coordinates \eqref{self-similar-coordinates}, the authors adopted a dynamic self-similar coordinate system $(\tildes,\tildez)$ which forces the remainder term $\tilde \epsilon$ and its derivative to vanish at $\tildez=0$. This setting is more suitable for the present section, as we shall see that the weighted norm or linear evolution estimate requires the function to vanish at the origin at a certain rate. The precise result of \cite{Collot2023}, stated in these dynamic self-similar coordinates, is the following.}

{
\begin{theorem}[Section 5 in \cite{Collot2023}] \label{th:base-result-CIL}
Let $\kappa=2/3$. For any sufficiently large $\tilde s_0^*\gg 1$, there exists a mean-free initial condition $a_0\in C^2([0,1])$ such that the corresponding solution
$-u_X(t,X=0,Z)$ to \eqref{proudman-johnson-equation-2} blows up in finite time $T>0$ and satisfies the following. It can be decomposed as
\begin{equation}
    -u_X(t,X=0,Z)=\frac{1}{\lambda}(\phi(\tilde z)+\tilde \epsilon(\tilde s,\tilde z))
\end{equation}
where the renormalization variables are
\begin{equation} \label{dynamic-self-similar-variables}
\lambda=-\frac{1}{u_X(t,0,0)}>0, \quad \tilde z =\frac{Z}{\nu}, \quad \nu= \frac{-u_X(t,0,0)}{\partial_Z u_X(t,0,0)}>0, \quad \tilde s=\tilde s_0+\int_0^t \frac{dt}{\lambda}.
\end{equation}
Here $\lambda(t)=(T-t)(1+o_{t\uparrow T}(1))$ and $\nu(t)=|\ln (T-t)|^{-1}(1+o_{t\uparrow T}(1))$ and $\tilde\epsilon$ and $\nu$ satisfy in renormalized variables for $\tilde s \in [\tilde s_0^*,\infty)$,
\begin{equation}\label{a priori bound for tilde epsilon new-1}
    \begin{split}
        & \| \tilde\epsilon(\tildes) \|_{L^\infty (0, \nu^{-1}(\tilde s))}=O(\tildes^{-\kappa}) \\
        &\tildeepsilon (\tilde z)=O(\tildez^2) \text{ as }\tildez\rightarrow 0^+\\
        &\abs{\frac{\nu_\tildes}{\nu}}\lesssim \tildes^{-1},\, \nu = \tildes^{-1} (1+o_{\tilde s \to \infty}(1)), \nu>0\in C^1.
    \end{split}
\end{equation}
Moreover, for any small enough perturbation in $C^2([0,1])$ of $a_0$, the same conclusions hold true.
\end{theorem}
}

{The main result of this section is the following refined expansion in fixed self-similar coordinates \eqref{self-similar-coordinates}, providing a detailed expansion of the first element $F_{0,0}$ of the hierarchy $(F_{i,k})_{i,k\geq 0}$.}

\begin{theorem}\label{asymptotic expansion of epsilon}
{Let $u_X$ satisfy the conclusions of Theorem \ref{th:base-result-CIL} and let $F_{0,0}(s,z)=(T-t)u_X(t,X=0,Z)$ with $(s,z)$ given by \eqref{self-similar-coordinates}. Recall by \eqref{setting-renormalization} the decomposition $F_{0,0}(s,z)=-\phi(z)-\epsilon(s,z)$. Then $\epsilon$ satisfies}
    \begin{equation}
        \norm{\epsilon-\frac{\log s}{s}U_{1,\log}-\frac{1}{s}U_1}_{C_z^1([0,s])}=O((\frac{\log s}{s})^2),
    \end{equation}
    where the profiles are defined to be
    \begin{equation}
            U_{1,\log}(z) =-2z e^{-z}, \qquad U_1(z)=(2+(1+C_\nu)z) e^{-z}-1,
    \end{equation}
    for some $C_\nu\in \R$. Consequently, the mode $F_{0,1}$ defined in \eqref{setting-renormalization} has the following decomposition
    \begin{equation}
          F_{0,1}(s,z)=(1+\epsilon_0(s))\phi(z)
+\frac{\log s}{s}\Phi_{0,1}(z)
+\frac{1}s\Psi_{0,1}(z)+r_{0,1}(s,z),
    \end{equation}
    where the profiles are defined as 
    \begin{equation}
             \Phi_{0,1}(z)
=2e^{-z}-2ze^{-z}, \qquad \Psi_{0,1}(z)
=\bigl((1+C_\nu)z-C_\nu\bigr)e^{-z}
    \end{equation}
{and the remainders $\eps_0 =\epsilon\rvert_{z=0},r_{0,1}$ satisfy
\begin{equation}
    \begin{split}
       \eps_0 &=\frac{1}{s}-2\frac{\log s}{s^2}+O(\frac{1}{s^2}),\quad r_{0,1}=O(\frac{(\log s)^2}{s^2}).
    \end{split}
\end{equation}
}
\end{theorem}

{The rest of the section is devoted to the proof of Theorem \ref{asymptotic expansion of epsilon}. We now fix for the rest of the section $u_X$ as given by Theorem \ref{th:base-result-CIL}. Our ultimate goal is to show that $\epsilon$ given by \eqref{setting-renormalization} has a more precise asymptotic decomposition. Our plan is to show that $\tilde \epsilon$ given in Theorem \ref{th:base-result-CIL}} has a more precise expansion {within the dynamic self-similar coordinate system $(\tildes,\tildez)$ given by \eqref{dynamic-self-similar-variables}}, then recover the expansion of $\epsilon$ by converting the result back to the original fixed self-similar coordinates $(s,z)$ at the end. We shall show that the expansion of $\eps$ is in $C^1$, as this information not only enables us to conclude an asymptotic expansion on $F_{0,1}$ (via \eqref{eq: derivative-compatability} and \eqref{setting-renormalization}), but it is also needed for the subsequent analysis of $F_{1,0}$.

In the dynamical self-similar coordinates $(\tildes,\tildez)$ of Theorem \ref{th:base-result-CIL}, we have that $\tildeepsilon$ satisfies the following evolution equation (see \cite{Collot2023} for the computations),
\begin{equation}\label{equation of tilde epsilon new}
    \begin{split}
        &\tildeepsilon_\tildes-\frac{\lambda_\tildes}{\lambda}\tildeepsilon-\frac{\nu_\tildes}{\nu}\tildez\tildeepsilon_\tildez-2\phi\tildeepsilon+\partial_{\tilde z}^{-1}\phi \tildeepsilon_\tildez-\partial_\tildez^{-1}\tildeepsilon \phi-\tildeepsilon^2+\partial_\tildez^{-1}\tildeepsilon \tildeepsilon_\tildez\\
        &\qquad=(\frac{\lambda_\tildes}{\lambda}+1)\phi-\frac{\nu_\tildes}{\nu}\tildez\phi-2\nu\int_0^{\nu^{-1}} (\phi+\tildeepsilon)^2 d\tilde z,\\
        &{\int_0^{\nu^{-1}} ( \phi+\tildeepsilon )d\tildez =0,}
    \end{split}
\end{equation}
{on the domain $\set{(\tildes,\tildez):\tildes\geq s_0^*,\, 0\leq \tildez\leq \nu^{-1}(\tildes)}$, while the modulation equations of $(\lambda,\nu)$ are}
\begin{equation}\label{lambda and nu new}
\begin{split}
        &\frac{\lambda_\tildes}{\lambda}+1=2\nu\int_0^{\nu^{-1}} (\phi+\tildeepsilon)^2 d\tilde z\\
     &\frac{\lambda_\tildes}{\lambda}+\frac{\nu_\tildes}{\nu}=-1.
\end{split}
\end{equation}
{Note that this enforces that
\begin{equation}
\begin{split}
     & \tildeepsilon(\tildez=0)=\tildeepsilon_\tildez(\tildez=0)=0\\
    &\tildeepsilon=O(\tildez^2),\tildeepsilon_\tildez=O(\tildez)\text{ as }\tildez\rightarrow 0^+.
\end{split}
\end{equation}
}
We begin by decomposing $\tilde \epsilon$ in the {following ansatz}
{
\begin{equation}\label{decomp of tilde epsilon new}
    \tilde \epsilon (\tilde s,\tilde z)=\nu (\tilde s) \tilde U(\tilde z)+\tilde v(\tilde s,\tilde z).
\end{equation}
}
{By matching the first power of $\nu$ in Equation \eqref{equation of tilde epsilon new}, expecting $\tilde v=O(\nu^2)$, one sees that the suitable choice for the $\nu$-profile $\tilde U$ is}
\begin{equation}\label{nu profile new}
\begin{split}
    \tilde U(\tilde z)&=(1+\tilde z)e^{-\tildez}-1,
\end{split}    
\end{equation}
{
which satisfies the ODE
\begin{equation}
        \left((-1+2\phi)-\partial_\tildez^{-1} \phi \partial_\tildez+\phi\partial_\tildez^{-1}\right)\tilde U+\phi+\tildez\phi-1=0
\end{equation}
with boundary conditions $\tilde U(0)=\tilde U'(0)=0$.} If we substitute \eqref{nu profile new} into the equation of \eqref{equation of tilde epsilon new}, upon using \eqref{lambda and nu new} we obtain the equation of $\tilde v$ as
\begin{align}\label{equation of remainder new}
\partial_{\tilde s}\tilde v
&=
\left(-1-\frac{\nu_{\tilde s}}{\nu}+2\phi+2\nu \tilde U+\tilde v\right)\tilde v 
+\left(\frac{\nu_{\tilde s}}{\nu}\tilde z-\partial_{\tilde z}^{-1}(\phi+\tilde\epsilon)\right)\tilde v_{\tilde z}+\left(\phi-\nu \tilde U_{\tilde z}\right)\partial_{\tilde z}^{-1}\tilde v \notag\\
&\quad -\left(\frac{\nu_{\tilde s}}{\nu}+\nu\right)
\left(\tilde U+\nu(2\tilde U-\tilde z\,\tilde U_{\tilde z})\right) +\nu^2\left(\tilde U^2-(\partial_{\tilde z}^{-1}\tilde U)\tilde U_{\tilde z}
+2\tilde U-\tilde z\,\tilde U_{\tilde z}\right), \notag\\[0.75em]
\frac{\nu_{\tilde s}}{\nu}+\nu
&=
-\nu^2+\frac{11}{2}\nu^3
-4\nu\int_0^{\nu^{-1}}\phi\,\tilde v\,d\tilde z -4\nu^2\int_0^{\nu^{-1}}\tilde U\,\tilde v\,d\tilde z
-2\nu\int_0^{\nu^{-1}}\tilde v^2\,d\tilde z \\
&\quad
-8\nu^2(1+\nu)e^{-1/\nu}
+\left(4\nu+6\nu^2+\frac52\nu^3\right)e^{-2/\nu}, \notag\\[0.75em]
\int_0^{\nu^{-1}}\tilde v\,d\tilde z
&=
-2\nu+2(1+\nu)e^{-1/\nu}.
\end{align}
Differentiating this equation gives

\begin{equation}\label{remainder derivative new}
    \begin{split}
        \partial_{\tilde s}\tilde v_{\tilde z}
&=
\left(-1+\phi+\tilde\epsilon\right)\tilde v_{\tilde z} 
+\left(
\frac{\nu_{\tilde s}}{\nu}\tilde z
-\partial_{\tilde z}^{-1}(\phi+\tilde\epsilon)
\right)\tilde v_{\tilde z\tilde z} +\left(\nu \tilde U_{\tilde z}-\phi\right)\tilde v
-\left(\phi+\nu \tilde U_{\tilde z\tilde z}\right)\partial_{\tilde z}^{-1}\tilde v \\
&\quad
-\left(\frac{\nu_{\tilde s}}{\nu}+\nu\right)
\left(
\tilde U_{\tilde z}
+\nu(\tilde U_{\tilde z}-\tilde z\,\tilde U_{\tilde z\tilde z})
\right)  +\nu^2\left(
\tilde U\tilde U_{\tilde z}
-(\partial_{\tilde z}^{-1}\tilde U)\tilde U_{\tilde z\tilde z}
+\tilde U_{\tilde z}
-\tilde z\,\tilde U_{\tilde z\tilde z}
\right).
    \end{split}
\end{equation}
By Theorem \ref{th:base-result-CIL}, we see that $\tilde v$ satisfies the following estimates
{
\begin{equation}\label{a priori bound for tilde epsilon new}
    \begin{split}
        &\| \tilde v(\tildes)\|_{L^\infty([0,\nu^{-1}(\tilde s))}=O(\tildes^{-\kappa}), \quad \tilde v=O(\tildez^2) \text{ as }\tildez\rightarrow 0^+, \quad \abs{\frac{\nu_\tildes}{\nu}}\lesssim \tildes^{-1},\quad \nu = \tildes^{-1} (1+o_{\tilde s \to \infty}(1))
    \end{split}
\end{equation}
Our ideal goal is to show that $\abs{\tilde v}$ can be improved to be of order $\nu^2$ or $\tildes^{-2}$.} {For simplicity, let us fix, once for all that $\kappa=\frac{2}{3}$, based on Theorem \ref{th:base-result-CIL}.}

\subsection{$L^\infty$ estimate of the remainder in dynamical self-similar coordinates}
\subsubsection{Settings}
We begin by decomposing
\begin{equation}\label{tilde epsilon source new}
    \begin{split}
        \partial_s \tilde v&=\mathcal{H} \tilde v+F\\
        F&=-\left(\frac{\nu_{\tilde s}}{\nu}+\nu\right)
\left(\tilde U+\nu(2\tilde U-\tilde z\,\tilde U_{\tilde z})\right)  +\nu^2\left(\tilde U^2-(\partial_{\tilde z}^{-1}\tilde U)\tilde U_{\tilde z}
+2\tilde U-\tilde z\,\tilde U_{\tilde z}\right), \\
    \end{split}
\end{equation}
where the linear operator $\mathcal{H}$ is further decomposed into ``transport+non-local''. We shall include the non-linear term within the transport operator, as it is partially known to us, thanks to the a priori estimate 
\eqref{a priori bound for tilde epsilon new}.
\begin{equation}\label{decomp of op for tilde epsilon new} 
\begin{split}
      \mathcal{H}&=\mathcal{T}+\mathcal{K}\\
      \mathcal{T}&=\left(-1-\frac{\nu_{\tilde s}}{\nu}+2\phi+2\nu \tilde U+\tilde v\right)\,\mathds{1}+
\left(\frac{\nu_{\tilde s}}{\nu}\tilde z-\partial_{\tilde z}^{-1}(\phi+\tilde\epsilon)\right)\,\partial_{\tilde z}\\
\mathcal{K}&=\left(\phi-\nu \tilde U_{\tilde z}\right)\partial_{\tilde z}^{-1}
\end{split}
\end{equation}
For the equation of $\nu$, we simplify this as
\begin{equation}\label{tilde epsilon nu simplified new}
    \begin{split}
        \frac{\nu_\tildes}{\nu}+\nu&=
-\nu^2+O(\nu^3)
-4\nu\int_0^{\nu^{-1}}\phi\,\tilde v\,d\tilde z 
-4\nu^2\int_0^{\nu^{-1}}\tilde U\,\tilde v\,d\tilde z
-2\nu\int_0^{\nu^{-1}}\tilde v^2\,d\tilde z, \\
    \end{split}
\end{equation}
where $\tilde U$ is defined in \eqref{nu profile new}.
We begin by introducing weights $X_\theta:\R^{\geq 0}\mapsto \R^{>0}$, $\theta>0$ such that
\begin{equation}\label{weights for tilde epsilon new}
        X_\theta(\tildez)=B(\tildez)Y_\theta(\tildez), \quad B=(1-e^{-\tildez})^2, \quad Y_\theta=\exp(\theta \int_0^\tildez \chi(\frac{\xi}{L_\theta}) d\xi),
\end{equation}
where the cut-off $\chi$ is defined by \eqref{def-cut-off} and $L_\theta$ in \eqref{def-L_theta}.
These weights have the following properties
\begin{enumerate}
    \item {$0\leq X_\theta \in C^1[0,\infty)$}
    \item $Y'_\theta(\tilde z)=\theta Y_\theta \chi(\frac{\tildez}{L_\theta})\geq 0$, and therefore $Y_\theta,X_\theta$ is increasing.
    \item $\text{supp}(Y'_\theta(\tildez)),\text{supp}(\frac{Y'_\theta(\tildez)}{Y_\theta(\tildez)})\subseteq [0,2L_\theta]$.
    \item $Y_\theta=e^{\theta\tildez}$ if $0\leq \tildez\leq L_\theta$.
    \item $B,X_\theta= O(\tildez^2)$ as $\tildez \rightarrow 0^+$ and $X_\theta=0\iff \tildez=0.$
    \item $Y_\theta\approx_\theta 1$ for all $\tildez\geq 0$, $X_\theta\approx_\theta \tildez^2$ for $\tildez\leq 1$  and $X_\theta\approx_\theta 1$ for $\tildez\geq 1$ .
\end{enumerate}
\subsubsection{Linear evolution estimates}
\paragraph{Linear evolution estimates for the Transport operator}
We begin by attempting to derive an estimate on the linear evolution of the transport operator $\mathcal T$.
\begin{lemma}\label{tilde epsilon transport new}
    There exists $\theta^*$ universal, such that for any $0<p<1$, $\theta>\theta^*$, there exists $\tildes^*_1(p,\theta)$ so that when {$u_0 X_\theta^{-1}\in L^\infty$} and $\tildes\geq \tildes_1\geq\tildes_1^*$, then
    \begin{equation}
        \norm{\frac{\S_\mathcal{T}(\tilde s_1,\tilde s) u_0}{X_\theta}}_{L^\infty([0,\frac{1}{\nu(\tildes)}])}\leq   \norm{\frac{ u_0}{X_\theta}}_{L^\infty([0,\frac{1}{\nu(\tildes_1)}])}e^{-p(\tilde s-\tildes_1)}.
    \end{equation}
\end{lemma}
\begin{proof}
 We compute, by definition, that {for any $0<p<1$,}
\begin{equation}\label{weights for tilde epsilon new}
    \begin{split}
        & (\partial_\tildes-\mathcal{T})(X_\theta e^{-p\tildes})=X_\theta e^{-p\tildes}(-p-\frac{\mathcal T X_\theta}{X_\theta})\\
        &=X_\theta e^{-p\tildes} \left(-p+1-2\phi+\partial^{-1}_\tildez \phi \frac{B'}{B}\right. +\frac{\nu_\tildes}{\nu}-2\nu \tilde U-\tilde v-(\frac{\nu_\tildes}{\nu} \tilde z-\partial_\tildez ^{-1} \tildeepsilon)\frac{B'}{B} \left. +(\partial_\tildez^{-1}\phi-\frac{\nu_\tildes}{\nu}\tildez+\partial_\tildez^{-1}\tildeepsilon)\frac{Y_\theta'}{Y_\theta}\right),
    \end{split}
\end{equation}
where we have used the definition of decomposition of $\tildeepsilon$ \eqref{decomp of tilde epsilon new}, definition of $\mathcal{T}$ \eqref{decomp of op for tilde epsilon new}  and \eqref{weights for tilde epsilon new} for the weights and $\frac{B'}{B}=\frac{2\phi}{\partial_z^{-1}\phi}$. We estimate
\begin{equation*}
    \begin{split}
        -p+1-2\phi+\partial^{-1}_\tildez \phi \frac{B'}{B}&=-p+1-2\phi+\partial^{-1}_\tildez \phi\frac{2\phi}{\partial^{-1}_\tildez \phi}=-p+1\\
        \frac{\nu_\tildes}{\nu}-2\nu \tilde U-\tilde v&\geq -C\tildes^{-1}-C\tildes^{-1}-C\tildes^{-\frac{2}{3}}\\
        -(\frac{\nu_\tildes}{\nu} \tilde z-\partial_\tildez ^{-1} \tildeepsilon)\frac{B'}{B}&=-\frac{\nu_\tildes}{\nu} \frac{2\phi\tilde z}{\partial^{-1}_\tildez \phi}+\partial_\tildez ^{-1} \tildeepsilon \frac{2\phi}{\partial^{-1}_\tildez \phi}\\
        &\geq -2\norm{ \frac{\phi\tilde z}{\partial^{-1}_\tildez \phi}}_{L^\infty}(\abs{\frac{\nu_\tildes}{\nu} }+\norm{ \tildeepsilon}_{L^\infty[0,\nu^{-1}(\tildes)])}) \geq -2\cdot 1\cdot(C \tildes^{-1}+C\tildes^{-\frac{2}{3}})\\
        (\partial_\tildez^{-1}\phi-\frac{\nu_\tildes}{\nu}\tildez+\partial_\tildez^{-1}\tildeepsilon)\frac{Y_\theta'}{Y_\theta}&=(\partial_\tildez^{-1}\phi-\frac{\nu_\tildes}{\nu}\tildez+\partial_\tildez^{-1}\tildeepsilon) \theta \chi(\frac{\tildez}{L_\theta})\\
         &\geq -\theta(\abs{\frac{\nu_\tildes}{\nu}}+\norm{\tildeepsilon}_{L^\infty([0,\frac{1}{\nu(\tildes)}])}) \norm{\tildez  \chi(\frac{\tildez}{L_\theta})}_{L^\infty(\R^{\geq 0})} \geq -\theta (C\tildes^{-1}+C\tildes^{-\frac{2}{3}}) (2L_\theta),
    \end{split}
\end{equation*}
where we have used the \eqref{weights for tilde epsilon new} to simplify $\frac{B'}{B}$, \eqref{a priori bound for tilde epsilon new} for $\tildeepsilon$, $\frac{\nu_\tildes}{\nu}$ and $\nu$, the fact that $U$ is uniformly bounded, the computation of $Y'_\theta(z)$, as well as $\norm{\tildez  \chi(\frac{\tildez}{L_\theta})}_{L^\infty(\R^{\geq 0})}=\norm{\tildez  \chi(\frac{\tildez}{L_\theta})}_{L^\infty[0,2L_\theta]}\leq 2L_\theta$, as $\chi$ is smooth cutoff function. Consequently, for $0<p<1$, we have
\begin{equation}
\begin{split}
      (\partial_\tildes-\mathcal{T})(X_\theta e^{-p\tildes})&\geq (1-p-C_\theta \tildes^{-\frac{2}{3}})(X_\theta e^{-p\tildes})\\
      &\geq\frac{1-p}{2}X_\theta e^{-p\tildes}\geq 0,\text{ for }\tildes\geq \tildes^*(p,\theta) \text{ large enough.}
\end{split}
\end{equation}
We verify that under the same condition, 
{
\begin{equation}
    \norm{\frac{u_0}{X_\theta}}_{L^\infty([0,\frac{1}{\nu(\tildes_1)}])} e^{-p\tildes} X_\theta\pm e^{-p \tildes_1} \S_\mathcal{T}(\tilde s_1,\tilde s)u_0,
\end{equation}
}
is a supersolution for $\tildes\geq\tilde s_1\geq \tildes^*(p,\theta)$ and $v_1=O(z^2)$ as $z\rightarrow 0^+$. At the boundary, this can be argued similarly as in \cite{Collot2023}. In particular, we notice that the transport speed $\partial_\tildez^{-1}\phi-\frac{\nu_\tildes}{\nu}\tildez+\partial_\tildez^{-1}\tildeepsilon$ evaluated at boundary $\tildez^*=\nu^{-1}$ coincides with $\frac{d\tildez^*}{d\tildes}=-\frac{\nu_\tildes}{\nu^2}$, thanks to the compatibility condition (second equation of \eqref{equation of tilde epsilon new}). {Since we now have a supersolution in the interior, as well as at the boundary $\tilde z=0,\nu^{-1}(\tilde s)$, the desired result follows by applying the maximum principle.}
\end{proof}

\paragraph{Bound estimate for the Non-local operator}
Next, we wish to find a bound on the operator $K$ with respect to this weighted norm.
\begin{lemma}\label{tilde epsilon nonlocal new}
    There exists a universal $\theta^*>0$ and an universal time $s^*$, such that when $\theta\geq \theta^*$ and $s\geq s^*$ the following holds. For any $\tildes\geq \tildes_1\geq s^*$,
    \begin{equation}
        \norm{\frac{\mathcal K u}{X_\theta}}_{L^\infty([0,\frac{1}{\nu(\tildes)}])}\leq 2\theta^{-1}  \norm{\frac{u}{X_\theta}}_{L^\infty([0,\frac{1}{\nu(\tildes)}])},
    \end{equation}
    whenever {$u X_\theta^{-1}\in L^\infty$}.
\end{lemma}
\begin{proof}
we first notice that
\begin{equation}
\begin{split}
     \abs{\frac{\mathcal Ku}{X_\theta}}&=\abs{\phi-\nu \tilde U_\tildez} \abs{\frac{\partial_\tildez^{-1} u}{X_\theta}}\leq\abs{\phi-\nu \tilde U_\tildez}\frac{\partial_\tildez^{-1} X_\theta}{X_\theta} \norm{\frac{u}{X_\theta}}_{L^\infty([0,\frac{1}{\nu(\tildes)}])},
\end{split}
\end{equation}
for $u=O(\tildez^2)$ as $z\rightarrow 0^+$. We estimate that
\begin{equation}\label{integral quotient weight new}
\begin{split}
    \frac{\partial_\tildez^{-1} X_\theta}{X_\theta}=&\begin{cases}
       \frac{\int_0^\tildez e^{\theta \hat z} B(\hat z) d\hat z}{e^{\theta \tildez} B(\tildez) } &\quad \text{ if }\tildez\leq L_\theta\\
        \frac{\int_0^\tildez Y_\theta(\hat z) B(\hat z) d\hat z}{Y_\theta(\tildez) B(\tildez) }&\quad \text{ if }\tildez> L_\theta
    \end{cases}
    \leq\begin{cases}
       \frac{B(\tilde z)\int_0^\tildez e^{\theta \hat z}  d\hat z}{e^{\theta \tildez} B(\tildez) } &\quad \text{ if }\tildez\leq L_\theta\\
        \frac{Y_\theta(\tildez) B(\tildez)\int_0^\tildez 1 d\hat z}{Y_\theta(\tildez) B(\tildez) }&\quad \text{ if }\tildez> L_\theta
    \end{cases}=\begin{cases}
       \frac{1}{\theta}\frac{e^{\theta \tildez}-1}{e^{\theta \tildez}} &\quad \text{ if }\tildez\leq L_\theta\\
        \tildez&\quad \text{ if }\tildez> L_\theta
    \end{cases}
\end{split}
\end{equation}
where we have used $\tildez\leq L_\theta$, $Y_\theta=e^{\theta\tildez}$ and $B,Y_\theta$ is increasing. By direct computation, we have
\begin{equation}\label{inner weight new}
\begin{split}
     \frac{\partial_\tildez^{-1} X_\theta}{X_\theta}\abs{(\phi-\nu \tilde U_\tildez)}&\leq \frac{1}{\theta }(e^{-\tildez}+\nu \tildez e^{-\tildez})
     \leq \frac{1}{\theta}(1+\nu e^{-1}) \text{ for }\tilde z\leq L_\theta\\
\end{split}
\end{equation}
 Since $L_\theta=2\log \theta$, we have $e^{-L_\theta/2}=\frac{1}{\theta}$ and if we choose $L_\theta \geq 20$, then
\begin{equation}
    L_\theta^{1} e^{-L_\theta}\leq L_\theta^{2} e^{-L_\theta}\leq  e^{-\frac{L_\theta}{2}}=\frac{1}{\theta}
\end{equation}
On the other hand, if $\tildez\geq L_\theta\geq 2$, then $\tildez^{1} e^{-\tildez},\tildez^{2} e^{-\tildez}$ are decreasing. Therefore, for $\theta$ large enough such that $L_\theta\geq 20$
\begin{equation}\label{outside weight new}
\begin{split}
     \frac{\partial_\tildez^{-1} X_\theta}{X_\theta}\abs{(\phi-\nu \tilde U_\tildez)}&\leq (\tildez e^{-\tildez}+\nu \tildez^2 e^{-\tildez})\leq \frac{1}{\theta}(1+\nu), \text{ for }\tilde z> L_\theta.
\end{split}
\end{equation}
If $\theta $ and $L_\theta$ satisfy these assumptions, then combining \eqref{inner weight new} in the inner region and \eqref{outside weight new} in the outer region, we find that 
\begin{equation}
    \begin{split}
       \abs{(\phi-\nu \tilde U_\tildez)}\frac{\partial_\tildez^{-1} X_\theta}{X_\theta}\leq (1+\nu)\theta^{-1}.
    \end{split}
\end{equation}
The bound can now be obtained by using \eqref{a priori bound for tilde epsilon new}.
\end{proof}

\paragraph{Combined linear evolution estimate}
We now have an estimate on the linear evolution of the combined operator.
\begin{corollary}\label{tilde epsilon combined new}
    There exists a universal $\theta^{**}$, such that the following holds for the linear evolution of $\mathcal H$. For any $\theta\geq \theta^{**}$, there exists $\tildes^{**}(\theta)$ large enough, such that 
    \begin{equation}
         \norm{\frac{\S_\mathcal{H}(\tilde s_1,\tilde s) u_0}{X_\theta}}_{L^\infty([0,\frac{1}{\nu(\tildes)}])}\leq   \norm{\frac{ u_0}{X_\theta}}_{L^\infty([0,\frac{1}{\nu(\tildes_1)}])}e^{-\frac{1}{2}(\tildes-\tildes_1)},
    \end{equation}
    whenever $\tildes\geq \tildes_1\geq \tildes^{**}$ and {$u_0 X_\theta^{-1}\in L^\infty$}.
\end{corollary}
\begin{proof}
    Following a standard Duhamel-Gronwall argument, this corollary is proved by combining Lemma \ref{tilde epsilon transport new} and Lemma \ref{tilde epsilon nonlocal new}
\end{proof}
\subsubsection{Estimate in the weighted $L^\infty$ norm via a \textit{bootstrap} argument}

\begin{lemma}\label{tilde epsilon iterative}
    {Let $\tilde v$ be the solution of \eqref{equation of remainder new}.} There exists $\theta^{**}$ large enough such that for all $\theta\geq \theta^{**}$, an $\bar s_0^*(\theta)\gg 1$ exists, such that the following holds. For any fixed $\tilde s_0^*\geq \bar s_0^*(\theta)$ and $\alpha>\frac{1}{2}$, if there exists a constant $C_{\tilde s_0^*,\alpha,\theta}>0$ such that
    \begin{equation}\label{tilde epsilon iterative simple original state}
    \norm{{\tilde v(\tilde s)}}_{L^\infty [0,\nu^{-1}(\tildes)]}\leq C_{\tilde s_0^*,\alpha,\theta}\tildes^{-\alpha} \text{ for }\tildes\geq s_0^*,
\end{equation}
then there exists another constant $\tilde C_{\tilde s_0^*,\alpha,\theta}>0$ such that
 \begin{equation}\label{tilde epsilon iterative simple imrpoved state}
    \norm{\frac{\tilde v(\tilde s)}{X_\theta}}_{L^\infty [0,\nu^{-1}(\tildes)]}\leq \tilde C_{\tilde s_0^*,\alpha,\theta}\tildes^{-\min(2,1+\alpha,2\alpha)} \text{ for }\tildes\geq s_0^*,
\end{equation}
\end{lemma}
\begin{proof}

Fix $\theta\geq \theta^{**}$ as in Corollary \ref{tilde epsilon combined new}. Set $$\bar s_0^*(\theta)=\max(\tilde s^{**}(\theta),50,\hat s^*), $$where $\tilde s^{**}$ is given in Corollary \ref{tilde epsilon combined new}, $\frac{1}{2}\tildes ^{-1}\leq \nu\leq 2\tildes^{-1}$ whenever $\tildes\geq \hat s^*$, as well as \eqref{tilde epsilon nu simplified new} can be written. Let $\tilde s_0^*\geq \bar s_0^*(\theta)$ be fixed. For $\alpha>\frac{1}{2}$, given that 
\begin{equation}\label{tilde epsilon iterative simple}
    \norm{{\tilde v(\tilde s)}}_{L^\infty [0,\nu^{-1}(\tildes)]}\leq C_{\tilde s_0^*,\alpha,\theta}\tildes^{-\alpha} \text{ for }\tildes\geq s_0^*.
\end{equation}
We begin by estimating the source term, from \eqref{tilde epsilon nu simplified new}
\begin{equation}\label{tilde epsilon nu simplified estimate new}
    \begin{split}
        \abs{\frac{\nu_\tildes}{\nu}+\nu}
&\leq 4\tildes^{-2}+C\tildes^{-3}+8 \tildes^{-1}\norm{{\tilde v(\tilde s)}}_{L^\infty [0,\nu^{-1}(\tildes)]}\int_0^{\infty}\phi\,d\tilde z + 16 \tildes^{-2}\norm{{\tilde v(\tilde s)}}_{L^\infty [0,\nu^{-1}(\tildes)]}\norm{\tilde U}_{L^\infty }\int_0^{\nu^{-1}}\,d\tildez\\
&\quad+4\tildes^{-1} \norm{{\tilde v(\tilde s)}}_{L^\infty [0,\nu^{-1}(\tildes)]}^2\int_0^{\nu^{-1}}1\,d\tildez\\
&\leq (4+C)\tildes^{-2}+8 C_{\tilde s_0^*,\alpha,\theta}\tildes^{-\alpha-1}+16\tildes^{-2} C_{\tilde s_0^*,\alpha,\theta}\tildes^{-\alpha}\cdot 1 \cdot 2\tildes+4\tildes^{-1} C_{\tilde s_0^*,\alpha,\theta}^2\tildes^{-2\alpha}\cdot 2\tildes\\
&\leq 3\max (4+C,40 C_{\tilde s_0^*,\alpha,\theta},8 C_{\tilde s_0^*,\alpha,\theta}^2)\tildes^{-\min(2,1+\alpha,2\alpha)},
    \end{split}
\end{equation}
for $\tildes \geq \tildes_0^*$, where we have used the definition of $\bar s_0^*$ in the first and second line, \eqref{tilde epsilon iterative simple} and definition of $\tilde U$ in the third line.

For $\tildes \geq \tilde s_0^*$, we estimate the source term \eqref{tilde epsilon source new} in weighed norm as
\begin{equation}
    \begin{split}
       &  \norm{\frac{F}{X_\theta}}_{L^\infty[0,\nu^{-1}(s)]} \leq \left|\frac{\nu_{\tilde s}}{\nu}+\nu\right|\norm{\frac{\tilde U}{X_\theta}}_{L^\infty }
+\abs{\nu} \left|\frac{\nu_{\tilde s}}{\nu}+\nu\right|\norm{\frac{2\tilde U-\tilde z\,\tilde U_{\tilde z}}{X_\theta}}_{L^\infty} 
+\nu^2 \norm{\frac{ \tilde U^2-(\partial_{\tilde z}^{-1}\tilde U)\tilde U_{\tilde z}
+2\tilde U-\tilde z\,\tilde U_{\tilde z}}{X_\theta}}_{L^\infty} \\
&\leq C'_{s_0^*,\alpha,\theta} C_\theta \tildes^{-\min(2,1+\alpha,2\alpha)} +2 \tilde s^{-1} C_{\tilde s_0^*,\alpha,\theta}' C_\theta \tildes^{-\min(2,1+\alpha,2\alpha)}+4\tilde s^{-2} C_\theta \leq C_{\tilde s_0^*,\alpha,\theta}''  \tildes^{-\min(2,1+\alpha,2\alpha)},
    \end{split}
\end{equation}
where we have used \eqref{tilde epsilon nu simplified estimate new} and 
$\tilde U,\tilde z \tilde U_\tildez,\tilde U^2,(\partial_{\tilde z}^{-1}\tilde U)\tilde U_{\tilde z},
2\tilde U,\tilde z\,\tilde U_{\tilde z}$ are $O(\tildez^2)$ at $0$ and $O(1)$ at $\infty$. Consequently, when $\tildes \geq \tilde s_0^*$, Duhamel's principle and linear evolution estimate (Corollary \ref{tilde epsilon combined new}) gives
\begin{equation}\label{tilde epsilon eq 1}
    \begin{split}
       & \norm{\frac{\tilde v(s)}{X_\theta}}_{L^\infty[0,\nu^{-1}(\tilde s)]} \leq \norm{ \frac{\S_\mathcal{H}(\tildes_0^*,\tildes)\tilde v (\tilde s_0^*)}{X_\theta}}_{L^\infty[0,\nu^{-1}(\tilde s)]}+\int_{\tildes_0^*}^{\tildes} \norm{\frac{\S_\mathcal{H}(\tau,\tildes) F}{X_\theta}}_{L^\infty[0,\nu^{-1}(\tilde s)]} d\tau\\
        &\leq e^{-\frac{1}{2}(\tildes-\tildes_0^*)}\norm{ \frac{\tilde v (\tildes_0^*)}{X_\theta}}_{L^\infty[0,\nu^{-1}(\tilde s_0^*)]}+\int_{\tildes_0^*}^{\tildes} e^{-\frac{1}{2}(\tildes-\tau)} \norm{\frac{ F}{X_\theta}}_{L^\infty[0,\nu^{-1}(\tau)]} d\tau\\
        &\leq e^{\frac{1}{2} \tildes_0^*}\left( \norm{ \frac{\tilde \epsilon (\tildes_0^*)}{X_\theta}}_{L^\infty[0,\nu^{-1}(\tilde s_0^*)]}+\nu(\tilde s_0^*)\norm{ \frac{\tilde U}{X_\theta}}_{L^\infty}\right)e^{-\frac{1}{2} \tildes} +C_{\tilde s_0^*,\alpha,\theta}'' \int_{1}^{\tildes} e^{-\frac{1}{2}(\tildes-\tau)}  \tau^{-\min(2,1+\alpha,2\alpha)} d\tau\\
         &\leq e^{\frac{1}{2} \tildes_0^*}\left( \norm{ \frac{\tilde \epsilon (\tildes_0^*)}{X_\theta}}_{L^\infty[0,\nu^{-1}(\tilde s_0^*)]}+\frac{2C_\theta}{{\tildes_0^{*}}}+C_{\tilde s_0^*,\alpha,\theta}'''\right)\tildes^{-\min(2,1+\alpha,2\alpha)} =  \tilde C_{\tilde s_0^*,\alpha,\theta}\tildes^{-\min(2,1+\alpha,2\alpha)}
    \end{split}
\end{equation}
where we have used \eqref{tilde epsilon nu simplified estimate new} and \eqref{decomp of tilde epsilon new} in the third line, the fact that for $\tildes\geq \tildes_0^*\geq\bar s_0^* \geq 50$,
$\norm{ \frac{\tilde \epsilon (\tildes_0^*)}{X_\theta}}_{L^\infty[0,\nu^{-1}(\tilde s_0^*)]}$ is bounded, as $\epsilon (\tildes_0^*)$ is $O(\tildez^2)$ at $\tildez=0$ and uniformly bounded when $\tildez\geq 0$ (See the definition of initial closeness of \cite{Collot2023}),
\begin{equation}\label{tilde epsilon convenient inequality}
    e^{-\frac{1}{2}\tilde s}\leq e^{-\frac{1}{4}\tildes}\leq \tildes^{-2}\leq \tildes^{-\min(2,1+\alpha,2\alpha)}.
\end{equation}
as well as $\alpha>1/2$ and $\int_{1}^{\tildes} e^{-\frac{1}{2}(\tildes-\tau)}  \tau^{-\min(2,1+\alpha,2\alpha)} d\tau \lesssim C_\alpha  \tildes^{-\min(2,1+\alpha,2\alpha)}$.
\end{proof}

\begin{corollary}\label{tilde epsilon remainder estimate new}
    There exists $\theta^{**}$ large enough such that for all $\theta\geq \theta^{**}$, an $\bar s_0^*(\theta)\gg 1$ exists, such that the following holds. For any fixed $\tilde s_0^*\geq \bar s_0^*(\theta)$, there exists a constant $C_{\tilde s_0^*,\theta}>0$ such that
     \begin{equation}
          \norm{\frac{\tilde v(\tilde s)}{X_\theta}}_{L^\infty [0,\nu^{-1}(\tildes)]}\leq C_{\tilde s_0^*,\theta}\tildes^{-2} \text{ for }\tildes\geq s_0^*.
     \end{equation}
\end{corollary}
\begin{proof}
    Fix $\theta\geq 2\theta^{**}$ and set initial $\alpha$ to be $1>\alpha_0=\frac{2}{3}>\frac{1}{2}$, thanks to the fact that $\norm{\cdot}_{L^\infty[0,\nu^{-1}(\tildes)]}\lesssim_\theta \norm{\frac{\cdot}{X_\theta}}_{L^\infty[0,\nu^{-1}(\tildes)]}$, we can iteratively apply Lemma \ref{tilde epsilon iterative} to obtain a sequence of $(\alpha_n)$ that is increasing and reaches to $2$ in finite steps.
\end{proof}
\subsection{$C^1$ estimate of the remainder in dynamical self-similar coordinates}
{We now establish a bound on the derivative of the remainder term.}
\subsubsection{Settings}
To simplify notations, we denote $\tilde w=\tilde v_\tildez$, from \eqref{remainder derivative new}, we introduce
\begin{equation}\label{tilde epsilon source new derivative}
\begin{split}
    \partial_s \tilde w&={\mathcal  J} \tilde w+G\\
    G&=\left(\nu \tilde U_{\tilde z}-\phi\right)\tilde v
-\left(\phi+\nu \tilde U_{\tilde z\tilde z}\right)\partial_{\tilde z}^{-1}\tilde v 
-\left(\frac{\nu_{\tilde s}}{\nu}+\nu\right)
\left(
\tilde U_{\tilde z}
+\nu(\tilde U_{\tilde z}-\tilde z\,\tilde U_{\tilde z\tilde z})
\right) \\
&\quad
+\nu^2\left(
\tilde U\tilde U_{\tilde z}
-(\partial_{\tilde z}^{-1}\tilde U)\tilde U_{\tilde z\tilde z}
+\tilde U_{\tilde z}
-\tilde z\,\tilde U_{\tilde z\tilde z}
\right),
\end{split}
\end{equation}
and the operator $\mathcal J$ is defined to be
\begin{equation}\label{decomp of op for tilde epsilon new derivative}
    \mathcal J= \left(-1+\phi+\tilde\epsilon\right)\mathds{1} 
+\left(
\frac{\nu_{\tilde s}}{\nu}\tilde z
-\partial_{\tilde z}^{-1}(\phi+\tilde\epsilon)
\right)\partial_\tildez 
\end{equation}
\subsubsection{Transport operator estimate}
We begin by introducing a weight $\bar X_\theta:\R^{\geq 0}\mapsto \R^{>0}$, $\theta>1$ such that
\begin{equation}\label{weights for tilde epsilon new derivative}
        \bar X_\theta(\tildez)=\bar B(\tildez)Y_\theta(\tildez), \quad \bar B=(1-e^{-\tildez}), \quad Y_\theta \text{ the same defined in \eqref{weights for tilde epsilon new}.} 
\end{equation}

\begin{lemma}\label{tilde epsilon transport new derivative}
    For any $\theta>1$ and {$u_0 \bar X_\theta^{-1}\in L^\infty$}, there exists $\tildes^*_1(\theta)$ such that when $\tildes\geq \tildes_1\geq\tildes_1^*$,
    \begin{equation}
        \norm{\frac{\S_\mathcal{J}(\tildes_1,\tildes) u_0}{\bar X_\theta}}_{L^\infty([0,\frac{1}{\nu(\tildes)}])}\leq   \norm{\frac{ u_0}{\bar X_\theta}}_{L^\infty([0,\frac{1}{\nu(\tildes_1)}])}e^{-\frac{1}{2}(s-\tildes_1)},
    \end{equation}
\end{lemma}
\begin{proof}
We compute, by definition, that {for any $0<q<1$,}
\begin{equation}\label{weights for tilde epsilon new derivative}
    \begin{split}
        & (\partial_\tildes-\mathcal{J})(\bar X_\theta e^{-q\tildes}) =\bar X_\theta e^{-q\tildes}(-q-\frac{\mathcal J \bar X_\theta}{\bar X_\theta})\\
        &=\bar X_\theta e^{-q\tildes} \left(-q+1-\phi+\partial^{-1}_\tildez \phi \frac{\bar B'}{\bar B}\right. -\tildeepsilon-(\frac{\nu_\tildes}{\nu} \tilde z-\partial_\tildez ^{-1} \tildeepsilon)\frac{\bar B'}{\bar B} \left. +(\partial_\tildez^{-1}\phi-\frac{\nu_\tildes}{\nu}\tildez+\partial_\tildez^{-1}\tildeepsilon)\frac{Y_\theta'}{Y_\theta}\right),
    \end{split}
\end{equation}
where we have used the definition of $\mathcal{J}$ \eqref{decomp of op for tilde epsilon new derivative}.
\begin{equation*}
    \begin{split}
        -q+1-\phi+\partial^{-1}_\tildez \phi \frac{\bar B'}{\bar B}&=-q+1-\phi+\partial^{-1}_\tildez \phi\frac{\phi}{\partial^{-1}_\tildez \phi}=-q+1\\
        -\tildeepsilon-(\frac{\nu_\tildes}{\nu} \tilde z-\partial_\tildez ^{-1} \tildeepsilon)\frac{B'}{B}&=-\tildeepsilon-\frac{\nu_\tildes}{\nu} \frac{2\phi\tilde z}{\partial^{-1}_\tildez \phi}+\partial_\tildez ^{-1} \tildeepsilon \frac{2\phi}{\partial^{-1}_\tildez \phi}\\
        &\geq-\norm{\tildeepsilon}_{L^\infty[0,\frac{1}{\nu(\tildes)}]}  -2\norm{ \frac{\phi\tilde z}{\partial^{-1}_\tildez \phi}}_{L^\infty}(\abs{\frac{\nu_\tildes}{\nu} }+\norm{ \tildeepsilon}_{L^\infty[0,\nu^{-1}(\tildes)])})\\
        &\geq -C\tildes^{-\frac{2}{3}} -2\cdot 1\cdot(C \tildes^{-1}+C\tildes^{-\frac{2}{3}})\\
        (\partial_\tildez^{-1}\phi-\frac{\nu_\tildes}{\nu}\tildez+\partial_\tildez^{-1}\tildeepsilon)\frac{Y_\theta'}{Y_\theta}&=(\partial_\tildez^{-1}\phi-\frac{\nu_\tildes}{\nu}\tildez+\partial_\tildez^{-1}\tildeepsilon) \theta \chi(\frac{\tildez}{L_\theta})\\
         &\geq -\theta(\abs{\frac{\nu_\tildes}{\nu}}+\norm{\tildeepsilon}_{L^\infty([0,\frac{1}{\nu(\tildes)}])}) \norm{\tildez  \chi(\frac{\tildez}{L_\theta})}_{L^\infty(\R^{\geq 0})}\\
        &\geq -\theta (C\tildes^{-1}+C\tildes^{-\frac{2}{3}}) (2L_\theta),
    \end{split}
\end{equation*}
where we have used $\frac{\bar B'}{\bar B}=\frac{\phi}{\partial_\tildez^{-1}\phi}$ and all the other reasons are similar to the proof in {Lemma \ref{tilde epsilon transport new}}.
Consequently, for $0<q<1$, we have
\begin{equation}
\begin{split}
      (\partial_\tildes-\mathcal{J})(\bar X_\theta e^{-p\tildes})&\geq (1-q-C_\theta \tildes^{-\frac{2}{3}})(\bar X_\theta e^{-q\tildes})\\
      &\geq\frac{1-q}{2}\bar X_\theta e^{-q\tildes}\geq 0,\text{ for }\tildes\geq \tildes^*(q,\theta) \text{ large enough.}
\end{split}
\end{equation}
Therefore, under the same condition, following the same procedure as before (Lemma \ref{tilde epsilon transport new}), we see
\begin{equation}
    \norm{\frac{v_1}{\bar X_\theta}}_{L^\infty([0,\frac{1}{\nu(\tildes_1)}])} e^{-q\tildes} \bar X_\theta\pm e^{-q \tildes_1} \S_\mathcal{J}(\tildes_1,\tildes) v_1,
\end{equation}
is also a supersolution. {The desired result follows by applying maximum principle and setting $q=1/2$.}
\end{proof}

\subsubsection{Estimate in the weighed $L^\infty$ norm}
\begin{lemma}\label{tilde epsilon remainder in C1}
    There exists $\theta^{**}$ large enough such that for all $\theta\geq \theta^{**}$, an $\bar s_0^*(\theta)\gg 1$ exists, such that the following holds. For any fixed $\tilde s_0^*\geq \bar s_0^*(\theta)$, there exists a constant $C_{\tilde s_0^*,\theta}>0$ such that
     \begin{equation}
          \norm{\frac{\tilde v(\tilde s)}{X_\theta}}_{ [0,\nu^{-1}(\tildes)]},\norm{\frac{\tilde v_\tildez(\tilde s)}{\bar X_\theta}}_{ [0,\nu^{-1}(\tildes)]}\leq C_{\tilde s_0^*,\theta}\tildes^{-2} \text{ for }\tildes\geq \tilde s_0^*.
     \end{equation}
     In particular,
      \begin{equation}
          \norm{{\tilde v(\tilde s)}}_{C^1 [0,\nu^{-1}(\tildes)]}\lesssim_{\tilde s_0^*,\theta}\tildes^{-2} \text{ for }\tildes\geq \tilde s_0^*.
     \end{equation}
\end{lemma}
\begin{proof}
Recall that $\tilde v_\tildez=\tilde w$. Fix $\theta\geq \theta^{**}$ as in Corollary \ref{tilde epsilon remainder estimate new}. Let $\theta\geq \theta^{**}$ and $$\bar s_0^{**}(\theta)=\max(\bar s_0^*(\theta),\tilde s_1^*(\theta),50,\hat s^{*})$$ where $\bar s_0^*(\theta)$ is given by Corollary \ref{tilde epsilon remainder estimate new}, $\tilde s_1^*(\theta)$ is given by Lemma \ref{tilde epsilon transport new derivative}, and $\frac{1}{2}\tildes ^{-1}\leq \nu\leq 2\tildes^{-1}$ whenever $\tildes\geq \hat s^*$, as well as \eqref{tilde epsilon nu simplified new} can be written. Fix $\tilde s_0^*\geq \bar s_0^{**}(\theta)$, then whenever $\tilde s\geq \tilde s_0^*$ holds, we have, by Corollary \ref{tilde epsilon remainder estimate new},
$$  \norm{\frac{\tilde v(\tilde s)}{X_\theta}}_{L^\infty [0,\nu^{-1}(\tildes)]}\leq C_{\tilde s_0^*,\theta}\tildes^{-2},$$
 and the linear evolution estimate Lemma \ref{tilde epsilon transport new derivative} holds. We begin by estimate each term appears in the source term \eqref{tilde epsilon source new derivative},
\begin{equation}
    \begin{split}
        \norm{\frac{\left(\nu \tilde U_{\tilde z}-\phi\right)\tilde v}{\bar X_\theta}}_{L^\infty[0,\frac{1}{\nu(\tildes)}]}&\leq (\abs{\nu} \norm{\tilde U_\tildez}_{L^\infty}+\norm{\phi}_{L^\infty})  \norm{\frac{\tilde v}{ X_\theta}\frac{X_\theta}{\bar X_\theta}}_{L^\infty[0,\frac{1}{\nu(\tildes)}]}\\
        &\leq (2C\tilde s^{-1}+1)C_{\tilde s_0^*,\theta}\tildes^{-2}\lesssim_{\tilde s_0^*,\theta} \tildes^{-2}\\
        \abs{\frac{\partial_\tildez^{-1}\tilde v}{\bar X_\theta}}=\abs{\frac{\partial_\tildez^{-1}(\frac{\tilde v}{X_\theta}{X_\theta})}{\bar X_\theta}}&\leq \norm{\frac{\tilde v}{X_\theta}}_{L^\infty[0,\frac{1}{\nu(\tildes)}]}\frac{\partial_\tildez^{-1} X_\theta}{\bar X_\theta}
        \leq \norm{\frac{\tilde v}{X_\theta}}_{L^\infty[0,\frac{1}{\nu(\tildes)}]} \frac{X_\theta}{\bar X_\theta}\tildez\leq\norm{\frac{\tilde v}{X_\theta}}_{L^\infty[0,\frac{1}{\nu(\tildes)}]} \tildez\\
        \norm{\frac{\left(\phi+\nu \tilde U_{\tilde z\tilde z}\right)\partial_{\tilde z}^{-1}\tilde v }{\bar X_\theta}}_{L^\infty[0,\frac{1}{\nu(\tildes)}]}&\leq (\norm{\tilde z\phi}_{L^\infty}+\abs{\nu}\norm{\tilde z\tilde U_{\tildez\tildez}}_{L^\infty})\norm{\frac{\tilde v}{X_\theta}}_{L^\infty[0,\frac{1}{\nu(\tildes)}]}\\
        &\leq (C+2C\tildes^{-1})C_{\tilde s_0^*,\theta} \tildes^{-2}\lesssim_{\tilde s_0^*,\theta} \tildes^{-2}\\
        \norm{\frac{\tilde U_{\tilde z}
+\nu(\tilde U_{\tilde z}-\tilde z\,\tilde U_{\tilde z\tilde z})}{\bar X_\theta}}_{L^\infty}&\leq C_\theta +2\tildes^{-1} C_\theta \lesssim_\theta 1,\quad
\norm{\frac{
\tilde U\tilde U_{\tilde z}
-(\partial_{\tilde z}^{-1}\tilde U)\tilde U_{\tilde z\tilde z}
+\tilde U_{\tilde z}
-\tilde z\,\tilde U_{\tilde z\tilde z}}{\bar X_\theta}}_{L^\infty}\lesssim_\theta 1
    \end{split}
\end{equation}
where we have used $X_\theta$ is increasing, $\frac{X_\theta}{\bar X_\theta}=\frac{B}{\bar B}=1-e^{-\tildez}$ and 
$\tilde U_\tildez, \tilde z \tilde U_{{\tilde z} {\tilde z}}$, $\tilde U\tilde U_{\tilde z}
,(\partial_{\tilde z}^{-1}\tilde U)\tilde U_{\tilde z\tilde z}$ are $O(\tildez)$ at $\tildez=0$ and $O(e^{-\tildez/2})$ at infinity. On the other hand, we have $\nu^2\leq 4\tildes^{-2}$, and by reusing the computation of \eqref{tilde epsilon convenient inequality}, we have
\begin{equation}
    \abs{\frac{\nu_\tildes}{\nu}+\nu}\lesssim_{\tilde s_0^*,\theta}\tildes^{-2}.
\end{equation}
Therefore, for $\tildes\geq \tildes_0^*$, the source term satisfies
\begin{equation}
    \norm{\frac{G}{\bar X_\theta}}_{L^\infty[0,\frac{1}{\nu(\tildes)}]}\lesssim_{s_0^*,\theta} \tildes^{-2}.
\end{equation}
Finally, combining the linear evolution estimate Lemma \ref{tilde epsilon transport new derivative} and Duhamel's principle, we have for $\tildes\geq \tildes_0^*$,
\begin{equation}
    \begin{split}
        \norm{\frac{\tilde w(s)}{\bar X_\theta}}_{L^\infty[0,\nu^{-1}(\tilde s)]}&\leq \norm{ \frac{\S_\mathcal{J}(\tildes_0^*,\tildes)\tilde w (\tilde s_0^*)}{\bar X_\theta}}_{L^\infty[0,\nu^{-1}(\tilde s)]}+\int_{\tildes_0^*}^{\tildes} \norm{\frac{ \S_\mathcal{J}(\tau,\tildes) G}{\bar X_\theta}}_{L^\infty[0,\nu^{-1}(\tilde s)]} d\tau\\
        &\leq e^{-\frac{1}{2}(\tildes-\tildes_0^*)}\norm{ \frac{\tilde w (\tildes_0^*)}{ \bar X_\theta}}_{L^\infty[0,\nu^{-1}(\tilde s_0^*)]}+\int_{\tildes_0^*}^{\tildes} e^{-\frac{1}{2}(\tildes-\tau)} \norm{\frac{ G}{\bar X_\theta}}_{L^\infty[0,\nu^{-1}(\tau)]} d\tau\\
        &\leq e^{\frac{1}{2} \tildes_0^*}\left( \norm{ \frac{\tilde \epsilon_\tildez (\tildes_0^*)}{\bar X_\theta}}_{L^\infty[0,\nu^{-1}(\tilde s_0^*)]}+\nu(\tilde s_0^*)\norm{ \frac{\tilde U_\tildez}{X_\theta}}_{L^\infty}\right)e^{-\frac{1}{2} \tildes} \\
        &\quad+C_{\tilde s_0^*,\theta}'' \int_{1}^{\tildes} e^{-\frac{1}{2}(\tildes-\tau)}  \tau^{-2} d\tau\lesssim_{\tilde s_0^*,\theta}\tildes^{-2},
    \end{split}
\end{equation}
where the last line follows the same idea, as in the proof of Lemma \ref{tilde epsilon iterative}.
\end{proof}

\subsection{The asymptotic expansion of modulation parameters in fixed self-similar coordinates}
{In the previous section, we have established a $C^1$ asymptotic expansion on $\tilde \eps$, using $\nu$ as a {spatial} scale. To convert this result for $\eps$ in fixed self-similar coordinates, one needs to first establish asymptotic expansions on the modulation parameters, which are used to establish a link between the physical and the dynamical self-similar coordinates. On the other hand, the fixed self-similar is explicitly defined using the physical ones. One shall see that the leading order of $\eps$: $\frac{\log s}{s}$ is intrinsically related to these modular asymptotic expansions, as the computation progresses.}
We recall the profile in the fixed self-similar variable setting:
\begin{equation}
    s=\abs{\log(T-t)},\, z=Zs.
\end{equation}
The following corollary gives the asymptotic expansion of modulation parameters in fixed self-similar coordinates.
\begin{corollary}\label{tilde epsilon modular expansion new}
    {Let $u_X$ be given by Theorem \ref{th:base-result-CIL}, then in the fixed self-similar coordinates $(s,z)$ given by \eqref{self-similar-coordinates}, {there exist $C_\tildes,C_\nu\in \R$ such that} the following asymptotic relations hold:}
    \begin{equation}
    \begin{split}
    \tilde s&=s+\log s+C_\tildes+O(\frac{\log s}{s})\\
        \frac{1}{\nu}&=s+2\log s-C_\nu+O(\frac{\log s}{s})\\
        \nu&=\frac{1}{s}-2\frac{\log s}{s^2}+\frac{C_\nu}{s^2}+O(\frac{(\log s)^2}{s^3})\\
        \tildez&=\frac{z}{\nu s}=z(1+\frac{2\log s}{s}-\frac{C_\nu}{s}+O(\frac{\log s}{s^2}))\\
        \lambda&=e^{-s}(1-\frac{1}{s}+2\frac{\log s}{s^2}+\frac{1-C_\nu}{s^2}+O(\frac{(\log s)^2}{s^3}))\\
        \frac{1}{\lambda}&=e^s(1+\frac{1}{s}-2\frac{\log s}{s^2}+\frac{C_\nu}{s^2}+O(\frac{(\log s)^2}{s^3})
    \end{split}
\end{equation}
\end{corollary}
\begin{proof}

 If $\norm{\tilde v}_{L^\infty[0,\nu^{-1}(\tildes)]}\lesssim_{\tildes_0^*,\theta} \tildes^{-2}$ is satisfied, then for sufficiently large $\tildes$, we reuse the computation of \eqref{tilde epsilon nu simplified estimate new}, which shows
\begin{equation*}
    4\nu\int_0^{\nu^{-1}}\phi\,\abs{\tilde v}\,d\tilde z ,\,4\nu^2\int_0^{\nu^{-1}}\abs{\tilde U\,\tilde v}\,d\tilde z=O(\tildes^{-3}),\,2\nu\int_0^{\nu^{-1}}\tilde v^2\,d\tilde z=O(\tildes^{-4}), 
\end{equation*}
combining with the definition of source term \eqref{tilde epsilon nu simplified new}, \eqref{a priori bound for tilde epsilon new} and $\nu\simeq \tildes^{-1}$ we see
\begin{equation}\label{1 over nu new}
    \begin{split}
        (\frac{1}{\nu})_\tildes&=1+\nu+O(\nu^2)=1+O(\tildes^{-1})\\
    \end{split}
\end{equation}
Integrating both hand sides and using geometric series gives
\begin{equation}
    \nu-\tildes^{-1}=\tildes^{-1}+O(\frac{\log\tildes}{\tildes^2})-\tildes^{-1}=O(\frac{\log\tildes}{\tildes^2}),
\end{equation}
and plug this back to the first equation of $1/\nu$ \eqref{1 over nu new}, using $\nu\simeq \tildes^{-1}$, we see
\begin{equation}
    (\frac{1}{\nu})_\tildes=1+\tildes^{-1}+O(\frac{\log\tildes}{\tildes^2}),
\end{equation}
integrating once again and using geometric series gives
\begin{equation}\label{tilde epsilon 1 over nu and nu in tilde s new}
\begin{split}
    \nu
    &=\frac{1}{\tildes}-\frac{\log\tildes}{\tildes^2}-\frac{\tilde C}{\tildes^2}+O(\frac{(\log \tildes)^2}{\tildes^3}).
\end{split}
\end{equation}
Since $-1=\frac{\lambda_\tildes}{\lambda}+\frac{\nu_\tildes}{\nu}=\frac{d}{d\tildes}(\log \lambda+\log \nu)=\frac{d}{d\tildes}(\log (\lambda\nu))$, we find that $\lambda\nu=K e^{-\tilde s} $, and thus
\begin{equation}
    \lambda=Ke^{-\tildes}\tildes (1+\frac{\log\tildes
    }\tildes+{\tilde C}{\tildes}^{-1}+O(\frac{\log\tildes}{\tildes^2})).
\end{equation}
As in \cite{Collot2023}, one introduces $t=\int_{\tildes_0}^\tildes \lambda d\bar s$ and set $T=\int_{\tildes_0}^\infty \lambda d\bar s$, then
\begin{equation}\label{T-t comp tilde epsilon new}
\begin{split}
        T-t&= K \int_\tildes^\infty e^{-\bar s} \bar s+e^{-\bar s}\log \bar s+\tilde Ce^{-\bar s}+O(e^{-\bar s} \frac{\log \tildes}{\tildes}) d\bar s\\
        &=K e^{-\bar s}\tildes (1+\frac{\log\tildes}{\tildes}+(\tilde C+1)\tildes^{-1}+O(\frac{\log \tildes}{\tildes^2})).
\end{split}
\end{equation}
Taking $\log$ on both sides of \eqref{T-t comp tilde epsilon new} shows
\begin{equation}\label{tilde epsilon log'd new}
\begin{split}
     \log(T-t)
     &=\log K-\tildes+\log(\tildes)+O(\frac{\log \tildes}{\tildes})
\end{split}
\end{equation}
and this shows $\tildes \simeq \abs{\log(T-t)}$ and thus using the previous equation and taking log gives
\begin{equation}
    \log \tildes=\log\abs{\log(T-t)}+O(\frac{\log\abs{\log(T-t)}}{\abs{\log(T-t)}}),
\end{equation}
which upon using \eqref{tilde epsilon log'd new}, one obtains
\begin{equation}
    \tildes=\abs{\log(T-t)}+\log\abs{\log(T-t)}+\log K+O(\frac{\log\abs{\log(T-t)}}{\abs{\log(T-t)}}).
\end{equation}
Combining this with \eqref{tilde epsilon 1 over nu and nu in tilde s new}, we see
\begin{equation}\label{tilde epsilon 1/nu and nu in T-t new}
\begin{split}
     \frac{1}{\nu}
     &=\abs{\log(T-t)}+2\log\abs{\log(T-t)}+\log K+\tilde C+O(\frac{\log\abs{\log(T-t)}}{\abs{\log(T-t)}})\\
     \nu&=\frac{1}{\abs{\log(T-t)}}-\frac{2\log \abs{\log(T-t)}}{\abs{\log(T-t)}^2}-\frac{\log K+\tilde C}{\abs{\log(T-t)}^2}+O(\frac{(\log\abs{\log(T-t)})^2}{\abs{\log(T-t)}^3})
\end{split}
\end{equation}
To compute the asymptotic expansion of $\lambda$ with respect to $T-t$, using \eqref{T-t comp tilde epsilon new} and combining with the equation of $\nu^{-1}$ \eqref{tilde epsilon 1 over nu and nu in tilde s new}, we find that
\begin{equation}\label{2 tilde new}
    T-t=Ke^{-\tildes}(\nu^{-1}+1+O(\frac{\log \tildes}{\tildes^2})).
\end{equation}
Next, combining $\lambda=Ke^{-\tildes}\nu^{-1}$ with \eqref{2 tilde new} one gets
\begin{equation}
\begin{split}
    \lambda&=(T-t)\frac{1}{(1+\nu+O(\frac{\log \tildes}{\tildes^3}))}\\
    &=(T-t)(1-\nu+(-\nu)^2+O(\nu^3)+O(\frac{\log \tildes}{\tildes^3}))\\
    &=(T-t)\left(1-\frac{1}{\abs{\log(T-t)}}+\frac{2\log \abs{\log(T-t)}}{\abs{\log(T-t)}^2}+\frac{\log K+\tilde C+1}{\abs{\log(T-t)}^2}+O\left(\frac{(\log\abs{\log(T-t)})^2}{\abs{\log(T-t)}^3}\right)\right),
\end{split}
\end{equation}
where we have first used $\nu \simeq \frac{1}{\abs{\log(T-t)}}\simeq \tildes^{-1}$ and then \eqref{tilde epsilon 1/nu and nu in T-t new} in the last equation.
This corollary is proved upon renaming constants $C_\tildes=\log K$ and $C_\nu=-\log K-\tilde C$, and using the definition of fixed self-similar variable.
\end{proof}

\subsection{Computation of asymptotic expansion}

We can now end the proof of Theorem \ref{asymptotic expansion of epsilon}.

\begin{proof}[Proof of Theorem \ref{asymptotic expansion of epsilon}]

We recall that in \cite{Collot2023} $\epsilon$ and $\tildeepsilon$ are resulted by different decompositions of the same $a=-u_X(X=0)$: {$\frac{1}{\lambda}(\phi(\tilde z)+\tilde \epsilon(\tilde s,\tilde z))$ and $\frac{1}{T-t}(\phi(z)+\eps(s,z))$}, which gives the formula
\begin{equation}\label{tilde epsilon---epsilon relation new}
\begin{split}
     \epsilon(s,z)&= \phi(\tilde z)-\phi(z)+\bar\lambda \phi(\tilde z)-\bar\lambda \phi(z)+\bar\lambda \phi(z)+(1+\bar\lambda) \tilde\epsilon(\tilde s,\tilde z)\\
     &=(1+\bar \lambda) (\phi(\tilde z)-\phi(z))+\bar\lambda \phi(z)+(1+\bar\lambda) \tilde\epsilon(\tilde s,\tilde z)\\
     &=(1+\bar \lambda) (\phi(\tilde z)-\phi(z))+\bar\lambda \phi(z)+(1+\bar\lambda) (\nu \tilde U(\tilde z)+\tilde v(\tildes,\tildez))
\end{split}
\end{equation}
where $\lambda^{-1}=e^s(1+\bar\lambda)$ or, by the previous Corollary,
\begin{equation}\label{bar lambda refined new}
    \epsilon(s,z=0)=\bar\lambda=\frac{1}{s}-2\frac{\log s}{s^2}+\frac{C_\nu}{s^2}+O(\frac{(\log s)^2}{s^3}).
\end{equation}
Similarly, by differentiating $a$ and seting $Z=0$, if one define $\frac{1}{\nu}=s(1+\bar \nu)$, then
\begin{equation}
    \begin{split}
        \epsilon_z\rvert_{z=0}&=-(1+\bar\nu)(1+\bar\lambda)+1\\
        &=-2\frac{\log s}{s}+\frac{C_\nu-1}{s}+O(\frac{\log s}{s^2})
    \end{split}
\end{equation}
By definition of $\phi$ and Corollary \ref{tilde epsilon modular expansion new}, we have
\begin{equation}\label{phi difference new}
    \begin{split}
        \phi(\tilde z)-\phi(z)&=e^{-\frac{z}{s\nu}}-e^{-z}\\
        &=e^{-z(1+2\frac{\log(s)}{s}-C_\nu s^{-1}+O(\frac{\log s}{s^2}))}-e^{-z}\\
        &=e^{-z}\left(\exp\left({-z\left(2\frac{\log(s)}{s}-C_\nu s^{-1}+O(\frac{\log s}{s^2})\right)}\right)-1\right)\\
    \end{split}
\end{equation}
Using Taylor series with Lagrangian remainder, there exists $0<\alpha<1$ (it may depend on $z,s$, but it is uniformly bounded) such that whenever $f(s)=o(1)$,
\begin{equation}\label{taylor new}
\begin{split}
     e^{-zf(s)}&=1-zf(s)+\frac{\left(zf(s)\right)^2}{2}\exp\left(-\alpha zf(s)\right),\\
\end{split}
\end{equation}
and thus one computes
\begin{equation}
    \begin{split}
        \abs{e^{-z}(e^{-zf(s)}-1)+ze^{-z}f(s)}&=\abs{\frac{z^2 e^{(-1-\alpha f(s))z}}{2}f(s)^2}\\
     &\leq \frac{1}{2}f(s)^2 z^2e^{-\frac{1}{2}z} \text{ for }s\geq s_1 \text{ large enough}\\
     &\lesssim f(s)^2 \text{ for }s\geq s_1 \text{ large enough, uniformly in }z.
    \end{split}
\end{equation}
where we have used $\alpha$ is uniformly bounded.
Letting $f(s)=2\frac{\log(s)}{s}-C_\nu s^{-1}+O(\frac{\log s}{s^2})$, we see that
\begin{equation}\label{tilde epsilon phi diff new}
\begin{split}
    \phi(\tildez)-\phi(z)
    &=-z e^{-z}\left(2\frac{\log(s)}{s}-C_\nu s^{-1}\right)+O\left((\frac{\log s}{s})^2\right),
\end{split}
\end{equation}
where the remainder is uniform in spatial variables.

Using Corollary \ref{tilde epsilon modular expansion new}, we see for sufficiently large $s\geq s_1$, $\tildez\geq z$ and $\tildez\simeq z$ and by Taylor series with Lagrangian remainder, there exists $\xi$ such that $z<\xi< \tildez\leq 2z$
\begin{equation}
    \begin{split}
         &\tilde U(\tildez)=\tilde U(z)+\tilde U'(z)(\tilde z-z)+\frac{1}{2}(\tilde z-z)^2 \tilde U''(\xi)\\
     &\abs{\tilde U''(\xi)}=\abs{(\xi-1)e^{-\xi}}\leq(1+\xi)e^{-\xi}\leq(1+2z)e^{-z}\\
     &\abs{\frac{1}{2}(\tilde z-z)^2 \tilde U''(\xi)}\lesssim z^2(1+2z)e^{-z}(\frac{2\log s}{s}-\frac{C_\nu}{s}+O(\frac{\log s}{s^2}))^2\lesssim (\frac{\log s}{s})^2\\
     &\tilde U(\tildez)=\tilde U(z)+\tilde U'(z)z(\frac{2\log s}{s}-\frac{C_\nu}{s}+O(\frac{\log s}{s^2}))+O((\frac{\log s}{s})^2)\\
     &\qquad\,\,=\tilde U(z)+z\tilde U'(z)(\frac{2\log s}{s}-\frac{C_\nu}{s})+O((\frac{\log s}{s})^2),
    \end{split}
\end{equation}
where each $O()$ is uniform in $z$ and we have used the definition of $U$ to conclude $zU'$ is uniformly bounded, as well as Lemma \ref{tilde epsilon modular expansion new} in the third and last equation to rewrite $(\tilde z-z)^2,(\tilde z-z)$. 

Consequently, assuming the assumption of Lemma \ref{tilde epsilon remainder estimate new} holds, and using the definition of $\tilde\epsilon $,
\begin{equation}\label{tilde epsilon in fixed variable new}
    \begin{split}
        \tildeepsilon(\tildes,\tilde z)&=\nu \tilde U(\tildez)+\tilde v(\tildes,\tildez)\\
        \nu \tilde U(\tildez)&=\frac{1}{s}\tilde U(z)+\frac{2\log s-C_\nu}{s^2}(z\tilde U'(z)-\tilde U(z))+O(\frac{(\log s)^2}{s^3})\\
        \tilde\epsilon(\tildes,\tildez)&=\frac{1}{s}\tilde U(z)+O(\frac{\log s}{s^2})
    \end{split}
\end{equation}
where the remainder is uniform in spatial variables and we have used Lemma \ref{tilde epsilon remainder in C1}, Corollary \ref{tilde epsilon modular expansion new} and the fact that $\tilde U, z\tilde U'$ is uniformly bounded. 

Finally, using \eqref{tilde epsilon---epsilon relation new}, thanks to \eqref{tilde epsilon in fixed variable new}, \eqref{tilde epsilon phi diff new} and \eqref{bar lambda refined new}, we compute
\begin{equation}
    \begin{split}
        (1+\bar\lambda)(\phi(\tildez)-\phi(z))&=-z e^{-z}\left(2\frac{\log(s)}{s}-C_\nu s^{-1}\right)+O\left((\frac{\log s}{s})^2\right)\\
        \bar\lambda\phi(z)&=\frac{1}{s}\phi(z)+O(\frac{\log s}{s^2})\\
        (1+\bar\lambda)\tildeepsilon(\tildes,\tildez)&=\frac{1}{s}U(z)+O(\frac{\log s}{s^2})\\
        \epsilon(s,z)&=-2ze^{-z}\frac{\log s}{s}+\frac{1}{s}(\tilde U(z)+\phi(z)+C_\nu z e^{-z})+O\left((\frac{\log s}{s})^2\right).
    \end{split}
\end{equation}
Proceeding in a similar manner, in light of Lemma \ref{tilde epsilon remainder in C1}, it is not difficult to see that
\begin{equation}
    \epsilon_z(s,z)=2(z-1)e^{-z}\frac{\log s}{s}+\frac{1}{s}\partial_z(\tilde U(z)+\phi(z)+C_\nu z e^{-z})+O\left((\frac{\log s}{s})^2\right).
\end{equation}
where the $O(\cdot)$ is uniform in $z$. For the second part, by derivative compatibility condition $\partial_z F_{i,k}=F_{i,k+1}$, we have 
\begin{equation}
\label{eq:F0k-refined-decomposition2}
\begin{split}
    & F_{0,1}(s,z)
    =
    (-1)^{1+1}(1+\epsilon_0(s))\phi(z)
    +
    \frac{\log s}{s}\Phi_{0,1}(z)
    +
    \frac1s\Psi_{0,1}(z)
    +
    r_{0,1}(s,z),  \\
 &\Phi_{0,1}:=\partial_z^1\Phi_{0,0}
    =
    -\partial_z^1U_{1,\log},\\    
&\Psi_{0,1}:=\partial_z^1\Psi_{0,0}
    =
    (-1)^1\phi-\partial_z^1U_1,\\
    &r_{0,1}:=\partial_z^1 r_{0,0}
    =
    -\partial_z^1\eta+(-1)^1\eta_0\phi,
\end{split}
\end{equation}
where $\eps=\eps\rvert_{z=0}$ and $\eta$ is the remainder in the asymptotic expansion of $\eps$.
\end{proof}

\section{The third order renormalized tangential derivative $F_{1,0}$}\label{sec:first order mode}
As mentioned earlier, our goal is to show that $F_{1,0}$ has an asymptotic expansion in the self-similar variable $(s,z)$ as $s\rightarrow \infty$. As the first order Taylor corrections, the dynamics of the original hierarchy $F_{i,k}$ is overwhelmingly determined by these modes. Being able to precisely understand the long term asymptotic behaviors of these first-order modes $F_{1,0}$ and $F_{0,1}$ is thus of great importance. In particular, they serve as base cases for the general asymptotic expansion for all the modes (Remark \ref{general shell mode expansions}) and we shall find substantial methodological resemblance with the next section. This section itself also shares some similarity with the analysis of $\eps$ (Section \ref{sec: asymptotic expansion of eps}), except that we do not proceed in the dynamical self-similar coordinates, as the equation of $F_{1,0}$ depend on $\eps,\eps_z$ and we already have obtained information on $\eps$ (and hence $F_{0,1}$) via Theorem \ref{asymptotic expansion of epsilon}.

{Notice that in this section, we only establish the desired behavior for solutions of the one-dimensional equation associated with the third horizontal derivative along the symmetry axis. That is, in original physical coordinate system, the 1D equation of the third order horizontal derivative restricted to the symmetry axis:
\begin{equation*}
\begin{split}
        &b_t-4ab+\left(\int_0^Z a(\tilde Z)d\tilde Z\right)\partial_Z b+3 \partial_Z a \left(\int_0^Z b(\tilde Z)d\tilde Z\right)+8 \int_0^1 a(\tilde Z)b(\tilde Z)d\tilde Z=0,\\
        &a(Z,t):=-u_X(X=0,Z,t),\quad b(Z,t):= u_{XXX}(X=0,Z,t)
\end{split}
\end{equation*}
which is the equation of \eqref{id:equation-uXXX}. Whenever a two-dimensional solution exists, its third horizontal derivative along the symmetry axis satisfies the same equation, and thus, with matching initial data, uniqueness then implies that this derivative coincides with the one-dimensional solution for as long as the full two-dimensional solution exists.}

Using the equation of hierarchy \eqref{eq:raw-hierarchy}, the renormalized third order tangential derivative, {restricted to the symmetry axis}, $F_{1,0}$ (defined via \eqref{setting-renormalization}) takes the form of
\begin{align}\label{F_1,0 equation}
& \partial_s F_{1,0}+4(1-\phi-\epsilon)F_{1,0}+\left( \partial_z^{-1}\phi+\partial_{z}^{-1}\epsilon+\frac{z}{s} \right) \partial_z F_{1,0}+3(\partial_z \phi+\partial_z \epsilon) \partial_z^{-1}F_{1,0}\\
&+\frac{8}{s}\int_0^s (\phi(\hat z)+\epsilon(s,\hat z))F_{1,0}(\hat z) d\hat z=0\notag.
\end{align}
{The main result of this section is the following refined expansion in fixed self-similar coordinates \eqref{self-similar-coordinates}, providing a detailed expansion of $F_{1,0}$ of the hierarchy $(F_{i,k})_{i,k\geq 0}$, corresponding to the renormalized third order tangential derivative of $u_{XXX}$, restricted to the symmetry axis.
\begin{theorem}\label{asymptotic expansion of F_1,0 recall}
   {Let $F_{0,0}(s,z)=(T-t)u_X(t,X=0,Z)$ and $F_{1,0}(s,z)=(T-t)^4u_{XXX}(t,X=0,Z)$ with $(s,z)$ given by \eqref{self-similar-coordinates}. Recall by \eqref{setting-renormalization} the decomposition $F_{0,0}(s,z)=-\phi(z)-\epsilon(s,z)$. Assuming $u_X$ satisfy the conclusion of Theorem \ref{th:base-result-CIL} and $\eps$ satisfies the conclusion of Theorem \ref{asymptotic expansion of epsilon}. Then $F_{1,0}$ satisfies}
    \begin{equation}
        F_{1,0}(s,z)=F_{1,0}\rvert_{z=0}(s)\phi(z)+\frac{\log s }{s}V_{1,\log}(z)+\frac{1}{s}V_1(z)+\tilde r(s,z),
    \end{equation}
    and there exists $f_\infty$ such that for all $s\geq s_1$
    \begin{equation}
    \begin{split}
         &\abs{F_{1,0}\rvert_{z=0}(s)-f_\infty}=O(\frac{\log s}{s}),\quad\norm{\frac{\tilde r(s,\cdot)}{\partial_z^{-1}\phi}}_{L^\infty[0,s]}=O( \frac{(\log s)^3}{s^2}),
    \end{split}
    \end{equation}
    where $V_{1,\log}$ and $V_1$ are defined as (with $C_\nu$ given in Theorem \ref{asymptotic expansion of epsilon})
    \begin{equation}
             V_{1,\log}(z)=-2 f_\infty z e^{-z}, \quad V_1(z) =f_\infty(-1+C_\nu ze^{-z}-\frac{1}{2} e^{-z}+e^{-2z}+\frac{1}{2}e^{-3z}).
    \end{equation}
\end{theorem}
}
\subsection{Modular decomposition}
{We begin by noticing that $\phi$ itself is inside the kernel of the leading order linearized operator of \eqref{F_1,0 equation}. The structure of the leading-order linearized operator thus motivates the
decomposition
\begin{equation}\label{eq: third tangential crude decomp}
    F_{1,0}(s,z)=f(s)\phi(z)+\tilde v(s,z).
\end{equation}
Indeed, $\phi$ spans the kernel of this leading order linear operator, so the leading-order
linear dynamics leaves the amplitude of this mode unchanged. We choose
$\tilde v$ in a complementary space on which the associated
linearized dynamics decays exponentially. The parameter $f(s)$ then
tracks the projection of $F_{1,0}$ onto the kernel subspace
$\operatorname{Span}\{\phi\}$. Since the leading-order operator vanishes
on this component, the lower-order terms cannot be neglected in
determining its evolution. This motivates introducing $f(s)$ as a
modulation parameter, while using the decay of the linearized dynamics
to control the remainder $\tilde v$.
}

Upon setting
\begin{equation}
    \tilde v(s,z=0)\equiv 0\implies f(s)=F_{1,0}(s,z=0),
\end{equation}
and combining with \eqref{eq:raw-hierarchy}, we arrive at the equations for $f$ and $\tilde v$  as
\begin{equation}\label{modular decomp third order}
    \begin{split}
        &{f_s}= 4f\rho-8s^{-1}\int_0^s (\phi+\epsilon)\tilde v dz\\
         \tilde v_s&=\mathcal{J} \tilde v+f F-8 s^{-1}\partial_z^{-1}\phi\left(f\xi+\int_0^s (\phi+\epsilon)\tilde v dz\right),
    \end{split}
\end{equation}
where each term is defined as ({Below, $\eps_0=\eps\rvert_{z=0}$, as defined in Theorem \ref{asymptotic expansion of epsilon}.})
\begin{equation}\label{definitions of functions in source new}
\begin{split}
    \epsilon_0(s)-s^{-1}&=-2\frac{\log s}{s^2}+\frac{C_\nu}{s^2}+O(\frac{(\log s)^2}{s^3})\\
  \rho(s)&=\left(\epsilon_0(s)-s^{-1}+s^{-1}e^{-2s}-2s^{-1}\int_0^s \epsilon \phi dz\right)\\
    \xi(s)&=\frac{1}{2}(1-e^{-2s})+\int_0^s \epsilon\phi dz\\
    F(s,z)&=4(\epsilon-\epsilon_0) \phi +(\partial_z^{-1}\epsilon+s^{-1}z)\phi-3\epsilon_z\partial_z^{-1}\phi,\\
\end{split}
\end{equation}
and the operator $\mathcal{J}$ is defined as 
\begin{equation}\label{definitions of operators recall}
    \begin{split}
        &\J v=Lv+3Kv\\
        &Lv=(-4\partial_z^{-1}\phi+4\epsilon) v+(-\partial_z^{-1} \phi-\partial_z^{-1}\epsilon-s^{-1}z) \partial_z v\\
        &Kv=(\phi-\epsilon_z)\partial_z^{-1}v,
    \end{split}
\end{equation}
where we have used \eqref{bar lambda refined new}.
By Duhamel's principle, the solution of $f,\tilde v$ can be written as
\begin{equation}\label{the equations of d and tilde v after gronwall new}
    \begin{split}
        f(s)&=f(s_1)\mathcal{R}(s)-8\mathcal R(s)\int_{s_1}^s \left(\tau^{-1} \int_0^\tau (\phi+\epsilon(\tau,\cdot)) \tilde v(\tau,\cdot)dz\right)\mathcal R(\tau)^{-1}d\tau\\
        \tilde v(s,\cdot)&= \S_\mathcal{J}(s_1,s)\tilde v(s_1,\cdot)+\int_{s_1}^s f \S_\mathcal{J}(\tau,s)F(\tau,\cdot) d\tau\\
        &-8\int_{s_1}^s \tau^{-1}\left(f\xi+\int_0^\tau (\phi+\epsilon(\tau,\cdot))\tilde v(\tau,\cdot)dz\right)\S_\mathcal{J}(\tau,s) \partial_z^{-1}\phi d\tau
    \end{split}
\end{equation}
where $ \mathcal R(s)$ is the inverse of integrating factor 
\begin{equation}\label{definition of integrating factor new}
    \mathcal R(s)=\exp\left(4\int_{s_1}^s \rho d\tilde s\right)
\end{equation}
{It can be directly computed (in \ref{subsec:the first iteration}) that the linear coefficient of the ODE of modulation parameter $f$ has leading order $O(\frac{\log s}{s^2})$, thanks to the fact that the leading order of  $\eps_0=\eps\rvert_{z=0}$ cancels with $s^{-1}$. Therefore, the first possible contribution to the rate of change of modulation parameter from the kernel component
occurs only at order $O(s^{-2}\log s)$. Since this rate is integrable
in $s$, such higher-order terms produce only a small accumulated
change when starting at sufficiently large $s$.}
\subsection{Linear evolution estimate}
Recall the weight $\bar X_\theta$ defined by \eqref{weights for tilde epsilon new derivative}.

\begin{lemma}
For each $\theta>1$, there exists a time $s^*(\theta)$ such that whenever $s\geq s_1\geq s^*(\theta)$
    \begin{equation}\label{transport linear semigroup for third order new}
    \norm{\frac{ \S_L(s_1,s) v_1}{\bar X_\theta}}_{L^\infty([0,s]) }\leq \norm{\frac{v_1}{\bar X_\theta}}_{L^\infty ([0,s_1])} e^{-\frac{3}{4}(s-s_1)},
\end{equation}
\end{lemma}
\begin{proof}
We compute that, based on the definition of transport operator \eqref{definitions of operators recall}
    \begin{equation}
   \begin{split}
       (\partial_s-L)(e^{-ps}\bar X_\theta)&=e^{-ps}\bar X_\theta\left(-p+4-4\phi+\frac{\bar B'}{\bar B}\partial_z^{-1}\phi -4\epsilon+\frac{\bar B'}{\bar B}\left(\partial_z^{-1}\epsilon+s^{-1}z\right)\right.\\
       &\quad \qquad \qquad \qquad \left.+\frac{Y_\theta'}{Y_\theta}(s^{-1}z+\partial_z^{-1}\epsilon+\partial_z^{-1}\phi)\right)\\
   \end{split}
\end{equation}
Using Theorem \ref{asymptotic expansion of epsilon} and the properties of weights, we estimate similarly
\begin{equation*}
    \begin{split}
        -p+4-4\phi+\frac{\bar B'}{\bar B}\partial_z^{-1}\phi&=-p+4-4\phi+\partial^{-1}_z \phi\frac{\phi}{\partial^{-1}_\tildez \phi} =-p+4-3\phi\geq1-p\\
       -4\epsilon+\frac{\bar B'}{\bar B}\left(\partial_z^{-1}\epsilon+s^{-1}z\right)
        &\geq-4\norm{\epsilon}_{L^\infty[0,s]}  -\norm{ \frac{\phi z}{\partial^{-1}_z \phi}}_{L^\infty}(s^{-1}+\norm{ \epsilon}_{L^\infty[0,s])}) \geq -C\frac{\log s}{s} \\
        \frac{Y_\theta'}{Y_\theta}(s^{-1}z+\partial_z^{-1}\epsilon+\partial_z^{-1}\phi)&=(\partial_z^{-1}\phi+s^{-1}z+\partial_z^{-1}\epsilon) \theta \chi(\frac{z}{L_\theta})\\
        &\geq -\norm{\epsilon}_{L^\infty([0,s])}\theta \norm{z\chi(\frac{z}{L_\theta})}_{L^\infty}\geq -C\theta \frac{\log s}{s} (2L_\theta),
    \end{split}
\end{equation*}
for sufficiently large $s$, where we have used $\frac{\bar B'}{\bar B}=\frac{\phi}{\partial_z^{-1}\phi}$ and the other reasons are similar to the proof in the previous section. Consequently, whenever $p<1$, for large $s$, we have 
\begin{equation}
      (\partial_s-L)(e^{-ps}\bar X_\theta)\geq e^{-ps}\bar X_\theta (1-p-C_\theta \frac{\log s}{s})\geq 0.
\end{equation}
We arrive at the desired result by observing

\begin{equation}
    \norm{\frac{v_1}{\bar X_\theta}}_{L^\infty([0,s_1])} e^{-ps}\bar X_\theta\pm e^{-p s_1} \S_L(s_1,s) v_1,
\end{equation}
is also a supersolution inside the domain, at the boundary (thanks to the compatibility condition $\int_0^s \phi+\epsilon dz\equiv 0 $), and at initial time $s_1$.
\end{proof}

\begin{lemma}\label{nonlocal bound}
    For each $\theta>1$, there exists $s^{**}(\theta)$, such that whenever $s\geq s_1\geq s^{**}(\theta)$,
    \begin{equation}
        \norm{\frac{K v}{\bar X_\theta}}_{L^\infty[0,s]}\leq 2\theta^{-1}  \norm{\frac{v}{\bar X_\theta}}_{L^\infty[0,s]} 
    \end{equation}
\end{lemma}
\begin{proof}
{This proof is similar to the proof of Lemma \ref{tilde epsilon nonlocal new}.}
\end{proof}

\begin{corollary}\label{semigroup estimate third order}
    There exists a universal $\theta^*$, such that whenever $\theta\geq \theta^*$, the following holds. There exists $s_0^*(\theta)>0$, and whenever  $s\geq s_1\geq s_0^*(\theta) $,
    \begin{equation}
          \norm{\frac{e^{(s-s_1)\mathcal J} v_1}{\bar X_\theta}}_{L^\infty([0,s]) }\leq \norm{\frac{v_1}{\bar X_\theta}}_{L^\infty ([0,s_1])} e^{-\frac{1}{2}(s-s_1)}.
    \end{equation}
\end{corollary}
\begin{proof}
    Apply Duhamel principle and Gronwall's inequality.
\end{proof}

\subsection{The first iteration}\label{subsec:the first iteration}
The asymptotic expansion is divided into two parts. In this subsection, we first show that $F_{1,0}$ converges to some {constant of} profile, while in the next subsection, we illustrate that one can obtain a more precise expansion, by introducing more corrections in $s$.
\subsubsection{Preliminary estimates}
Based on the definitions \eqref{definitions of functions in source new} and \eqref{definition of integrating factor new}, for sufficiently large $s_1$, if one plugs in the asymptotic expansion of $\epsilon$, one obtains
\begin{equation}\label{rho and xi expansion}
    \begin{split}
        \int_0^s \epsilon \phi dz&=-\frac{1}{2}\frac{\log s}{s}+\frac{1+C_\nu}{4}\frac{1}{s}+O((\frac{\log s}{s})^2)\\
        \rho(s)&=-\frac{\log s}{s^2}+\frac{C_\nu-1}{2}\frac{1}{s^2}+O(\frac{(\log s)^2}{s^3})\\
        \xi(s)&=\frac{1}{2}-\frac{1}{2}\frac{\log s}{s}+\frac{1+C_\nu}{4}\frac{1}{s}+O((\frac{\log s}{s})^2),
    \end{split}
\end{equation}
and thus we see
\begin{equation}\label{asy of integrating factor}
    \begin{split}
        \mathcal R_\infty(s_1):=\lim_{s\rightarrow \infty} \mathcal R(s)&=\exp(4\int_{s_1}^\infty \rho(\tau)d\tau)\text{ exists},\\
        \int_s^\infty\rho d\tau&=-\frac{\log s}{s}+\frac{C_\nu-3}{2}\frac{1}{s}+O((\frac{\log s}{s})^2)\\
        \mathcal R(s)&=\mathcal R_\infty(s_1)\exp(-4\int_{s}^\infty \rho(\tau)d\tau)\\
        &=\mathcal R_\infty(s_1)\left(1+4\frac{\log s}{s}+(6-2C_\nu)\frac{1}{s}+O((\frac{\log s}{s})^2)\right)
    \end{split}
\end{equation}
where we have used that $\phi\in L^1[0,\infty)$ and uniform bound in Theorem \ref{asymptotic expansion of epsilon}. As a result,  we see that as long as $s_1$ is large enough,
\begin{equation}\label{eq:bound on R}
     \frac{1}{2}\mathcal R_\infty(s_1)\leq \mathcal R(s)\leq 2\mathcal R_\infty(s_1) \text{ for all }s\geq s_1.
\end{equation}
Using the same asymptotes, we have
\begin{equation}\label{eq: asymptote of R_inf}
    \mathcal R_\infty(s_1)=\exp\left(-4\frac{\log s_1}{s_1}-(6-2C_\nu)\frac{1}{s_1}+O((\frac{\log s_1}{s_1})^2)\right),
\end{equation}
and therefore $\lim_{s_1\rightarrow \infty} \mathcal R_\infty(s_1)=1$ and
\begin{equation}\label{eq:bound on R_inf}
    \frac{1}{2}\leq \mathcal R_\infty(s_1)\leq 2 \text{ for any }s_1 \text{ large enough}.
\end{equation}
Now we estimate terms that may depend on $\tilde v$. Since $\epsilon$ can be written as
\begin{equation}
    \epsilon(s,z)=-\frac{1}{s}+\frac{1}{s}e^{-z}(2+(1+C_\nu)z-2z \log (s))+r(s,z),\, \norm{r(s,\cdot)}_{L^\infty[0,s]}=O((\frac{\log s}{s})^2),
\end{equation}
using triangular inequality $\abs{\abs{a+b}-\abs{a}}\leq \abs{b}$, we estimate
\begin{equation}\label{the estimate of int phi+epsilon tilde v}
\begin{split}
    \abs{\int_0^s \abs{\epsilon(s,z)}dz-\frac{1}{s}\int_0^s dz}  &\leq \int_0^s \left|\abs{\epsilon(s,z)}-\abs{-\frac{1}{s}}\right|dz\\
    &\leq\frac{1}{s}\int_0^s e^{-z}|2+(1+C_\nu)z-2z \log (s)| dz+C(\frac{\log s}{s})^2\int_0^s dz \\
    &\lesssim \frac{(\log s)^2}{s} \\
    \int_0^s \abs{\epsilon(s,z)}dz&=1+O(\frac{(\log s)^2}{s}),\\
\end{split}
\end{equation}
and therefore, we can estimate
\begin{equation}\label{integral of phi plus epsilon times tilde v}
    \begin{split}
        \abs{\int_0^s (\phi+\epsilon(s,\cdot)) \tilde v(s,\cdot)dz}&\leq \norm{\tilde v(s,\cdot)}_{L^\infty([0,s])}(1-e^{-s}+1+C \frac{(\log s)^2}{s} )\\
     & \lesssim_\theta (2+O(\frac{(\log s)^2}{s}))\norm{\frac{\tilde v(s,\cdot)}{\bar X_{\theta}}}_{L^\infty([0,s])}\lesssim_{\theta} \norm{\frac{\tilde v(s,\cdot)}{\bar X_{\theta}}}_{L^\infty([0,s])} 
    \end{split}
\end{equation}
For all the other terms, we need to obtain a weighted $L^\infty$ estimate
\begin{equation}
\begin{split}
     \abs{\frac{F(s,z)}{\bar X_{\theta}}}&\leq \bar X_{\theta}^{-1}\left( 4\abs{(\epsilon-\epsilon_0) \phi} +\abs{\partial_z^{-1}\epsilon+s^{-1}z}\phi+3\abs{\epsilon_z}\partial_z^{-1}\phi\right)\\
     &\leq (4\norm{\epsilon_z}_{L^\infty([0,s])} +\norm{\epsilon}_{L^\infty([0,s])}+s^{-1})\norm{\frac{z\phi}{\bar X_{\theta}}}_{L^\infty}+3\norm{\epsilon_z}_{L^\infty}\norm{\frac{\partial_z^{-1}\phi}{\bar X_{\theta}}}_{L^\infty}\\
     &\lesssim_\theta \frac{\log s}{s},
\end{split}
\end{equation}
where we have used Theorem \ref{asymptotic expansion of epsilon}.
Using the linear evolution estimate on $\J$ Lemma \ref{semigroup estimate third order}, we estimate
\begin{equation}\label{multiple estimates new}
    \begin{split}
        \norm{\frac{\S_\mathcal{J}(s_1,s)\tilde v(s_1,\cdot)}{\bar X_\theta}}_{L^\infty([0,s])}&\leq \norm{\frac{\tilde v(s_1,\cdot)}{\bar X_\theta}}_{L^\infty([0,s_1])}\exp\left({-\frac{1}{2}(s-s_1)}\right)\\
         \norm{\frac{\S_\mathcal{J}(\tau,s)F(\tau,\cdot)}{\bar X_\theta}}_{L^\infty([0,s])}&\leq \norm{\frac{F(\tau,\cdot)}{\bar X_\theta}}_{L^\infty([0,\tau])}\exp\left({-\frac{1}{2}(s-\tau)}\right)\\
         &\lesssim_\theta  \frac{\log \tau}{\tau}\exp\left({-\frac{1}{2}(s-\tau)}\right)\\
          \norm{\frac{\S_\mathcal{J}(\tau,s)\partial_z^{-1}\phi}{\bar X_\theta}}_{L^\infty([0,s])}&\leq \norm{\frac{\partial_z^{-1}\phi}{\bar X_\theta}}_{L^\infty}\exp\left({-\frac{1}{2}(s-\tau)}\right)\\
          &\lesssim_\theta \exp\left({-\frac{1}{2}(s-\tau)}\right)
    \end{split}
\end{equation}
Finally, plug in all the above estimates into the solution formula \eqref{the equations of d and tilde v after gronwall new}, we obtain that for all  $s_1\leq t$ and $s_1$ sufficiently large (may depend on $\theta$),
\begin{equation}\label{preliminary estimate new}
    \begin{split}
        \abs{f(t)}&\leq 2f(s_1)\mathcal{R_\infty}(s_1)+C_\theta \mathcal{R_\infty}(s_1)^2\int_{s_1}^t\tau^{-1} \norm{\frac{\tilde v(\tau,\cdot)}{\bar X_{\theta}}}_{L^\infty([0,\tau])} d\tau \\
        \norm{\frac{\tilde v(t,\cdot)}{\bar X_\theta}}_{L^\infty[0,t]}&\leq \norm{\frac{\tilde v(s_1,\cdot)}{\bar X_{\theta}}}_{L^\infty([0,s_1])} e^{-\frac{1}{2}(t-s_1)}+C_\theta\sup_{s_1\leq \sigma\leq t}\abs{f(\sigma)}\int_{s_1}^t (\frac{\log\tau}{\tau}+\frac{1}{\tau}) e^{-\frac{1}{2}(t-\tau)}d\tau\\
        &\quad+ C_\theta \int_{s_1}^t \tau^{-1}e^{-\frac{1}{2}(t-\tau)} \norm{\frac{\tilde v(\tau,\cdot)}{\bar X_\theta}}_{L^\infty[0,\tau]}d\tau
    \end{split}
\end{equation}

\subsubsection{Boundedness of energy functional}
For any fixed $s\geq s_1$ sufficiently large, we define
\begin{equation}\label{eq:definition of energy kernel}
    H(s)=\sup_{s_1\leq t\leq s}\abs{f(t)}+\sup_{s_1\leq t\leq s} \frac{t}{\log t} \norm{\frac{\tilde v(t,\cdot)}{\bar X_\theta}}_{L^\infty[0,t]}.
\end{equation}
Our goal is to show, as long as $s_1$ is chosen large, then $H(s)$ can be bounded uniformly by initial conditions and initial time $s_1$.
\begin{lemma}\label{continuity argument new}
    For any fixed $\theta\geq\theta^*$ (defined in Corollary \ref{semigroup estimate third order}), there exists $\bar s^*(\theta)$ such that whenever $s\geq s_1\geq \bar s^*(\theta)$, the following holds. If initially,
    \begin{equation}\label{initial criterion first iteration}
        \abs{f(s_1)},\norm{\frac{\tilde v(s_1,\cdot)}{\bar X_{\theta}}}_{L^\infty([0,s_1])}<\infty,
    \end{equation}
    then for any $s\geq s_1$,
    \begin{equation}\label{eq:intermediate energy estimate}
         H(s)\lesssim_\theta \abs{f(s_1)}\mathcal R_\infty(s_1)+\frac{s_1}{\log s_1}\norm{\frac{\tilde v(s_1,\cdot)}{\bar X_{\theta}}}_{L^\infty([0,s_1])}.
    \end{equation}
    In particular, for any $s\geq s_1$,
    \begin{equation}
        \abs{f(s)}\lesssim_{s_1,\theta}1,\,\norm{\frac{\tilde v(s,\cdot)}{\bar X_{\theta}}}_{L^\infty([0,s])}\lesssim_{s_1,\theta} \frac{\log s}{s}.
    \end{equation}
\end{lemma}
\begin{proof}
By definition of $H(s)$, we have
\begin{equation}
    \abs{f(t)}\leq H(s), \norm{\frac{\tilde v(t,\cdot)}{\bar X_\theta}}_{L^\infty[0,t]}\leq H(s)\frac{\log t}{t},\text{ for all } s_1\leq t\leq s.
\end{equation}
Using \eqref{preliminary estimate new}, since the $\tau\in[s_1,t]\subseteq[s_1,s]$ we can further bound
\begin{equation}\label{f ineq}
    \begin{split}
         \abs{f(t)}&\leq 2\abs{f(s_1)}\mathcal{R_\infty}(s_1)+C_\theta H(s)\mathcal{R_\infty}(s_1)^2\int_{s_1}^t\tau^{-1} \frac{\log \tau}{\tau} d\tau \\
         &\leq  2\abs{f(s_1)}\mathcal{R_\infty}(s_1)+C_\theta H(s)\mathcal{R_\infty}(s_1)^2 (\frac{\log(s_1)}{s_1}+\frac{1}{s_1}) \\
    \end{split}
\end{equation}
Similarly, we have
\begin{equation}
    \begin{split}
         \frac{t}{\log t}\norm{\frac{\tilde v(t,\cdot)}{\bar X_\theta}}_{L^\infty[0,t]}
        &\leq \norm{\frac{\tilde v(s_1,\cdot)}{\bar X_{\theta}}}_{L^\infty([0,s_1])} e^{-\frac{1}{2}(t-s_1)}\frac{t}{\log t}\\
         &+\quad 2C_\theta \frac{t}{\log t}\sup_{s_1\leq \sigma\leq s}\abs{f(\sigma)}\int_{s_1}^t \frac{\log\tau}{\tau} e^{-\frac{1}{2}(t-\tau)}d\tau\\
        &\quad+ C_\theta H(s)\frac{t}{\log t} \int_{s_1}^t \tau^{-1}e^{-\frac{1}{2}(t-\tau)} \frac{\log\tau}{\tau}d\tau.\\
    \end{split}
\end{equation}
Using the algebraic inequalities for sufficiently large $s_1$ ($\bar C$ below is an absolute constant),
\begin{equation*}
    \begin{split}
        \int_{s_1}^t \frac{\log\tau}{\tau} e^{-\frac{1}{2}(t-\tau)}d\tau& \leq \bar C \frac{\log t}{t}\\
        \int_{s_1}^t \tau^{-1}e^{-\frac{1}{2}(t-\tau)} \frac{\log\tau}{\tau}d\tau&\leq \bar C s_1^{-1} \frac{\log t}{t}\\
        \frac{t}{\log (t)} e^{-\frac{1}{2}t}&\leq  \frac{s_1}{\log (s_1)} e^{-\frac{1}{2}s_1} \text{ for }t\geq s_1\geq 3,
    \end{split}
\end{equation*}
the weighted norm is further bounded by
\begin{equation}\label{tilde v ineq}
\begin{split}
       \frac{t}{\log t}\norm{\frac{\tilde v(t,\cdot)}{\bar X_\theta}}_{L^\infty[0,t]}&\leq \frac{s_1}{\log s_1}\norm{\frac{\tilde v(s_1,\cdot)}{\bar X_{\theta}}}_{L^\infty([0,s_1])} +10C_\theta\sup_{s_1\leq \sigma\leq s}\abs{f(\sigma)}+\bar C C_\theta H(s)s_1^{-1}.
\end{split}
\end{equation}
Taking $\displaystyle\sup_{s_1\leq t\leq s}$ on \eqref{f ineq} then plug into the equation of \eqref{tilde v ineq}, we conclude 
\begin{equation*}
    \begin{split}
        H(s)&\leq 2(1+10 C_\theta)\abs{f(s_1)}\mathcal{R_\infty}(s_1)+\frac{s_1}{\log s_1}\norm{\frac{\tilde v(s_1,\cdot)}{\bar X_{\theta}}}_{L^\infty([0,s_1])} \\
        &\quad+ H(s)C_\theta\left(\mathcal R_\infty(s_1)^2 (\frac{\log(s_1)}{s_1}+\frac{1}{s_1})(1+10 C_\theta)+\bar C s_1^{-1}\right)
    \end{split}
\end{equation*}
Since for sufficiently large $s_1$, $\mathcal R_\infty (s_1)^2\leq 4$ and thus 
\begin{equation}
    C_\theta\left(\mathcal R_\infty(s_1)^2 (\frac{\log(s_1)}{s_1}+\frac{1}{s_1})(1+10 C_\theta)+\bar C s_1^{-1}\right)\leq \frac{1}{2},
\end{equation}
which implies the desired result.
\end{proof}
\subsubsection{Convergence of modulation parameter}
\begin{corollary}\label{second iteration convergence of modulation parameter}
    Under the assumption of Lemma \ref{continuity argument new}, there exists an $f_\infty(s_1)$ such that $\lim_{s\rightarrow \infty} f(s)=f_\infty(s_1)$ and the convergence speed satisfies
    \begin{equation}
        \abs{f(s)-f_\infty(s_1)}\lesssim_{s_1,\theta} \frac{\log s}{s},
    \end{equation}
    for $s\geq s_1\geq  \bar s^*(\theta)$.
\end{corollary}
\begin{proof}
{
Using the computation of \eqref{integral of phi plus epsilon times tilde v} and Lemma \ref{continuity argument new}, we find that $
        \abs{\int_0^\tau (\phi+\epsilon(\tau,\cdot)) \tilde v(\tau,\cdot)dz}
     \lesssim_{\theta,s_1} \frac{\log\tau}{\tau},$
and by observing the solution formula \eqref{the equations of d and tilde v after gronwall new}, we claim that
\begin{equation}\label{eq:formula of f_infty}
    f_\infty(s_1)=\mathcal R_\infty(s_1)\left(f(s_1)-8\int_{s_1}^\infty \left(\tau^{-1} \int_0^\tau (\phi+\epsilon(\tau,\cdot)) \tilde v(\tau,\cdot)dz\right)\mathcal R(\tau)^{-1}d\tau\right)
\end{equation}
is the limit for $f(s)$. Indeed, by applying \eqref{continuity argument new}, the ODE of $f$ satisfies the form $f_s\lesssim |\log s|/s^2$ and the desired result follows by time integration.}
\end{proof}

\subsection{The second iteration}
In this subsection, we provide a more precise asymptotic expansion of $F_{1,0}$, by using the previously derived result in the last subsection.
\subsubsection{The decomposition and profile}
From the previous subsection, we know that $f(s)=f_\infty(s_1)+\tilde f$, $\abs{\tilde f}\lesssim_{s_1,\theta} \frac{\log s}{s}$ and $\norm{\frac{\tilde v}{\bar X_\theta}}\lesssim_{s_1,\theta} \frac{\log s}{s}$, where $s_1$ is the initial time.
In this subsection, we attempt to find an asymptotic expansion of $\tilde v$. We begin by writing
\begin{equation}\label{eq;third tangential decomp on remainder}
    \tilde v(s,z)=\frac{\log s}{s}V_{1,\log}(z)+\frac{1}{s}V_1(z)+\tilde r(s,z),
\end{equation}
where the profiles are defined as
\begin{equation}
    \begin{split}
        V_{1,\log}(z)&=-2 f_\infty z e^{-z}\\
        V_1(z)&=f_\infty(-1+C_\nu ze^{-z}-\frac{1}{2} e^{-z}+e^{-2z}+\frac{1}{2}e^{-3z}),
    \end{split}
\end{equation}
and they satisfy
\begin{equation}\label{ODE of profiles for third order}
    \begin{cases}
        &\mathcal{J}_0 V_{1,\log}+f_\infty(4 U_{1,\log }\phi+(\partial_z^{-1} U_{1,\log})\phi-3U'_{1,\log}\partial_z^{-1}\phi)=0\\
        &\mathcal{J}_0 V_{1}+f_\infty(4(U_1-1)\phi+(\partial_z^{-1}U_1+z)\phi-3U_1'\partial_z^{-1}\phi-4\partial_z^{-1}\phi)=0\\
        &V_{1,\log}(0)=V_{1}(0)=0,
    \end{cases}
\end{equation}
where $U_{1,\log}$ and $U_1$ are defined to be the profiles of $\epsilon$, as in Theorem \ref{asymptotic expansion of epsilon}: 
\begin{equation}
        \begin{split}
            \epsilon(s,z)&=\frac{\log s}{s}U_{1,\log}(z)+\frac{1}{s}U_1(z)+\eta(s,z),\, \norm{\eta}_{C^1[0,s]}\lesssim (\frac{\log s}{s})^2 \\
            U_{1,\log}(z)&=-2z e^{-z}\\
            U_1(z)&=(2+(1+C_\nu)z) e^{-z}-1.
        \end{split}
    \end{equation}
Using the ODEs of \eqref{ODE of profiles for third order} as well as \eqref{modular decomp third order}, we see $\tilde r$ satisfies
\begin{equation}
\begin{split}
\tilde r_s
={}&\,\mathcal J\tilde r
-8s^{-1}\partial_z^{-1}\phi\int_0^s(\phi+\epsilon)\tilde r\,dz + \tilde F
+(f-f_\infty)\left(\frac{\log s}{s}\mathcal F_1+\frac1s\mathcal F_2\right) \\
&\,+f F_\eta-8s^{-1}g\partial_z^{-1}\phi +\frac{\log s-1}{s^2}V_{1,\log}+\frac1{s^2}V_1,
\end{split}
\end{equation}
where we have defined

\begin{equation}\label{source definition second iteration}
    \begin{split}
        F_\eta(s,z)&=\Bigl(
4\bigl(\eta-\eta\rvert_{z=0}\bigr)\phi
+\bigl(\partial_z^{-1}\eta\bigr)\,\phi
-3\eta_z\,\partial_z^{-1}\phi
\Bigr) \\
\tilde F(s,z)&=(\mathcal J-\mathcal J_0)\left(\frac{\log s}{s}V_{1,\log}+\frac1sV_1\right)\\
g(s)&=
f\xi-\frac{f_\infty}{2}
+\int_0^s\phi\left(\frac{\log s}{s}V_{1,\log}+\frac1sV_1\right)\,dz 
+\int_0^s\epsilon\left(\frac{\log s}{s}V_{1,\log}+\frac1sV_1\right)\,dz\\
 \xi(s)&=\frac{1}{2}(1-e^{-2s})+\int_0^s \epsilon\phi dz\\
\mathcal F_1(z)
&=\,4U_{1,\log}\phi
+\bigl(\partial_z^{-1}U_{1,\log}\bigr)\,\phi-3U_{1,\log}'\,\partial_z^{-1}\phi\\
\mathcal F_2(z)
&=\,4\bigl(U_1-1\bigr)\phi
+\bigl((\partial_z^{-1}U_1)+z\bigr)\phi-3U_1'\,\partial_z^{-1}\phi.
    \end{split}
\end{equation}
and $\mathcal J_0$ is the leading order parts of $\mathcal J$:
\begin{equation}
    \begin{split}
    &\mathcal J_0=-4\partial_z^{-1}\phi \mathds 1-\partial_z^{-1}\phi \partial_z +3\phi \partial_z^{-1}\\
    &\mathcal J-\mathcal J_0=4\epsilon \mathds 1-(\partial_z^{-1}\epsilon+s^{-1}z)\partial_z-3\epsilon_z \partial_z^{-1}\\
    \end{split}
\end{equation}
Notice that we can further write
\begin{equation}\label{tilde F}
    \begin{split}
         \tilde F(s,z)&=(\mathcal J-\mathcal J_0)\left(\frac{\log s}{s}V_{1,\log}+\frac1sV_1\right) \\
&=\,4\epsilon\left(\frac{\log s}{s}V_{1,\log}+\frac1sV_1\right) -\left(\partial_z^{-1}\epsilon+s^{-1}z\right)
\left(\frac{\log s}{s}\partial_zV_{1,\log}+\frac1s\partial_zV_1\right) \\
&\quad-3(\frac{\log s}{s}U'_{1,\log}+\frac{1}{s} U'_{1})
\left(\frac{\log s}{s}\partial_z^{-1}V_{1,\log}+\frac1s\partial_z^{-1}V_1\right)\\
&\quad-3 \eta_z\left(\frac{\log s}{s}\partial_z^{-1}V_{1,\log}+\frac1s\partial_z^{-1}V_1\right)
    \end{split}
\end{equation}
and that because of \eqref{rho and xi expansion} and the convergence of $f$ Corollary \ref{second iteration convergence of modulation parameter}:
\begin{equation}\label{convergence of s dependent}
\begin{split}
    \abs{f-f_\infty}&\lesssim_{s_1,\theta} \frac{\log s}{s}\\
     \abs{f\xi-\frac{f_\infty}{2}}&\leq \abs{f-f_\infty}\abs{\xi}+\abs{f_\infty}\abs{\xi-\frac{1}{2}}\lesssim_{s_1,\theta} \frac{\log s}{s}.
\end{split}
\end{equation}

\subsubsection{Preliminary estimate}
By Duhamel's principle, we write
\begin{equation}
    \begin{split}
        \tilde r&=\S_\mathcal{J}(s_1,s)\tilde r(s_1,\cdot)-8\int_{s_1}^s \tau^{-1}\left(\int_0^\tau (\phi+\epsilon(\tau,\cdot))\tilde r(\tau,\cdot)dz\right)\S_\mathcal{J}(\tau,s)\partial_z^{-1}\phi d\tau\\
        &\quad+\int_{s_1}^s \S_\mathcal{J}(\tau,s) \tilde F(\tau,\cdot) d\tau+\int_{s_1}^s (f(\tau)-f_\infty)\left(\frac{\log \tau}{\tau} \S_\mathcal{J}(\tau,s) \mathcal F_1+\frac{1}{\tau}\S_\mathcal{J}(\tau,s)\mathcal F_2\right)d\tau\\
        &\quad +\int_{s_1}^s f(\tau) \S_\mathcal{J}(\tau,s) F_\eta(\tau,\cdot)d\tau-8\int_{s_1}^s \tau^{-1} g(\tau) \S_\mathcal{J}(\tau,s) \partial_z^{-1}\phi d\tau\\
        &\quad+ \int_{s_1}^s \frac{\log\tau-1}{\tau^2}\S_\mathcal{J}(\tau,s) V_{1,\log} d\tau+\int_{s_1}^s \frac{1}{\tau^2}\S_\mathcal{J}(\tau,s) V_1 d\tau.
    \end{split}
\end{equation}
Combining with the linear evolution estimate Lemma \ref{semigroup estimate third order}, we estimate term by term the source terms:
\begin{equation}
    \begin{split}
        \norm{\frac{e^{\J(s-s_1)} \tilde r(s_1,\cdot)}{\bar X_\theta}}_{L^\infty([0,s])}&\leq \norm{\frac{\tilde r(s_1,\cdot)}{\bar X_\theta}}_{L^\infty[0,s_1]} \exp\left({-\frac{1}{2}(s-s_1)}\right)\\
        \abs{\int_0^\tau (\phi+\epsilon(\tau,\cdot)) \tilde r(\tau,\cdot)dz}&\lesssim_{\theta} \norm{\frac{\tilde r(\tau,\cdot)}{\bar X_{\theta}}}_{L^\infty([0,\tau])} \\
        \norm{\frac{e^{\J(s-\tau)} \partial_z^{-1}\phi}{\bar X_\theta}}_{L^\infty([0,s])}&\leq \norm{\frac{\partial_z^{-1}\phi}{\bar X_\theta}}_{L^\infty}\exp\left({-\frac{1}{2}(s-\tau)}\right) \lesssim_\theta \exp\left({-\frac{1}{2}(s-\tau)}\right).
    \end{split}
\end{equation}
where we have used \eqref{integral of phi plus epsilon times tilde v}. We now estimate each term in $\tilde F$. Using the asymptotic expansion on $\epsilon$ and that $V_{1,\log},V_1$ is $O(z)$ at $z=0$ and uniformly bounded, we have
\begin{equation}
    \begin{split}
      &\abs{4\epsilon(\tau,\cdot)\left(\frac{\log \tau}{\tau}V_{1,\log}+\frac1\tau V_1\right)\bar X_\theta^{-1}}\leq C \frac{\log \tau}{\tau} (C_{s_1,\theta} \frac{\log \tau}{\tau}+C_{s_1,\theta} \frac{1}{\tau})\lesssim_{s_1,\theta} \frac{(\log \tau)^2}{\tau^2}.\\
    \end{split}
\end{equation}
Using the asymptotic expansion on $\epsilon$ and that $z V_{1,\log}', zV_1'$ is $O(z)$ at 0 and decays exponentially at infinity, we have
\begin{equation}
    \begin{split}
&\abs{\left(\partial_z^{-1}\epsilon(\tau,\cdot)+\tau^{-1}z\right)
\left(\frac{\log \tau}{\tau}\partial_zV_{1,\log}+\frac1\tau\partial_zV_1\right) \bar X_\theta^{-1}}\\&\quad\leq (\norm{\epsilon(\tau,\cdot)}_{L^\infty[0,\tau]}+\tau^{-1})\left(\frac{\log \tau}{\tau}\norm{\frac{z\partial_zV_{1,\log}}{\bar X_\theta}}_{L^\infty}+\frac1\tau\norm{\frac{z\partial_zV_1}{\bar X_\theta}}_{L^\infty}\right) \\
&\quad \leq (C\frac{\log \tau}{\tau}+\frac{1}{\tau})(C_{s_1,\theta}\frac{\log \tau}{\tau}+C_{s_1,\theta}\frac{1}{\tau})\lesssim_{s_1,\theta} \frac{(\log \tau)^2}{\tau^2}
    \end{split}
\end{equation}
Using $z U_{1,\log}', zU_1'$ is $O(z)$ at 0 and decays exponentially at infinity, as well as $V_{1,\log}$ and $V_1$ are uniformly bounded, we have
\begin{equation}
    \begin{split}
        &\abs{3(\frac{\log \tau}{\tau}U'_{1,\log}+\frac{1}{\tau} U'_{1})
\left(\frac{\log \tau}{\tau}\partial_z^{-1}V_{1,\log}+\frac1\tau\partial_z^{-1}V_1\right) \bar X_\theta^{-1}}\\
&\quad\leq \abs{3(\frac{\log \tau}{\tau}\norm{\frac{zU'_{1,\log}}{\bar X_\theta}}_{L^\infty}+\frac{1}{\tau} \norm{\frac{zU'_{1}}{\bar X_\theta}}_{L^\infty})
\left(\frac{\log \tau}{\tau}\norm{V_{1,\log}}_{L^\infty}+\frac1\tau\norm{V_1}_{L^\infty}\right) }\\
&\quad \leq 3(C_\theta \frac{\log \tau}{\tau}+C_\theta \frac{1}{\tau})(C_{s_1,\theta} \frac{\log \tau}{\tau}+C_{s_1,\theta} \frac{1}{\tau})\lesssim_{s_1,\theta} \frac{(\log \tau)^2}{\tau^2}
    \end{split}
\end{equation}
Using the bound on $\eta$, we have
\begin{equation}
    \begin{split}
        &\abs{3 \eta_z\left(\frac{\log \tau}{\tau}\partial_z^{-1}V_{1,\log}+\frac1\tau\partial_z^{-1}V_1\right)\bar X_\theta^{-1}}\\
        &\quad\leq 3\norm{\eta_z}_{L^\infty[0,\tau]}(\frac{\log \tau}{\tau}\norm{V_{1,\log}}_{L^\infty} +\frac{1}{\tau}\norm{V_1}_{L^\infty})\norm{\frac{\partial_z^{-1}\phi}{\bar X_\theta}}_{L^\infty}\norm{\frac{z}{\partial_z^{-1}\phi}}_{L^\infty[0,\tau]}\\
        &\quad\leq 3C_{s_1,\theta} \frac{(\log \tau)^2}{\tau^2}(\frac{\log\tau}{\tau}+\frac{1}{\tau})C_\theta 2\tau\lesssim_{s_1,\theta} \frac{(\log \tau)^3}{\tau^2}.
    \end{split}
\end{equation}
Thus, we have
\begin{equation}
    \begin{split}
        \norm{\frac{\tilde F(\tau,\cdot)}{\bar X_\theta}}_{L^\infty[0,\tau]}\lesssim_{s_1,\theta} \frac{(\log \tau)^3}{\tau^2}, \norm{\frac{\S_\mathcal{J}(\tau,s)\tilde F(\tau,\cdot)}{\bar X_\theta}}_{L^\infty[0,s]}\lesssim_{s_1,\theta}\frac{(\log \tau)^3}{\tau^2} e^{-\frac{1}{2}(s-\tau)}
    \end{split}
\end{equation}
Using \eqref{convergence of s dependent} and \eqref{convergence of s dependent}, we have 
\begin{equation}
    \begin{split}
        \abs{f(\tau)-f_\infty}&\lesssim_{s_1,\theta} \frac{\log \tau}{\tau}\\
        \norm{\frac{\S_\mathcal{J}(\tau,s)\mathcal F_1}{\bar X_\theta}}_{L^\infty[0,s]}&\leq \norm{\frac{\mathcal F_1}{\bar X_\theta}}_{L^\infty[0,\tau]} e^{-\frac{1}{2}(s-\tau)}\lesssim_\theta e^{-\frac{1}{2}(s-\tau)}\\
       \norm{\frac{\S_\mathcal{J}(\tau,s) \mathcal F_2}{\bar X_\theta}}_{L^\infty[0,s]}&\leq \norm{\frac{\mathcal F_2}{\bar X_\theta}}_{L^\infty[0,\tau]} e^{-\frac{1}{2}(s-\tau)}\lesssim_\theta e^{-\frac{1}{2}(s-\tau)}\\
        \norm{\frac{\frac{\log \tau }{\tau}\S_\mathcal{J}(\tau,s) \mathcal F_1+\frac{1}{\tau}\S_\mathcal{J}(\tau,s) \mathcal F_2}{\bar X_\theta}}_{L^\infty[0,s]}&\lesssim_\theta \frac{\log\tau}{\tau}e^{-\frac{1}{2}(s-\tau)},
    \end{split}
\end{equation}
where we have used the definition of $\mathcal F_1,\mathcal F_2$ and that they are $O(z)$ at $z=0$ and decays exponentially at infinity. Next, using the estimate on $\eta$, we have
\begin{equation}
    \begin{split}
    \abs{f(\tau)}&=\abs{f(\tau)-f_\infty+f_\infty} \leq  \abs{f_\infty}+C_{s_1,\theta} \frac{\log \tau}{\tau} \lesssim_{s_1,\theta} 1\\
        \abs{\frac{F_\eta(\tau,\cdot)}{\bar X_{\theta}}}&\leq \bar X_{\theta}^{-1}\left( 4\abs{(\eta-\eta\rvert_{z=0}) \phi} +\abs{\partial_z^{-1}\eta}\phi+3\abs{\eta_z}\partial_z^{-1}\phi\right)\\
     &\leq (4\norm{\eta_z}_{L^\infty([0,\tau])} +\norm{\eta}_{L^\infty([0,\tau])})\norm{\frac{z\phi}{\bar X_{\theta}}}_{L^\infty}+3\norm{\eta_z}_{L^\infty}\norm{\frac{\partial_z^{-1}\phi}{\bar X_{\theta}}}_{L^\infty}\\
     &\lesssim_\theta (\frac{\log \tau}{\tau})^2\\
     \norm{\frac{\S_\mathcal{J}(\tau,s) F_\eta(,\cdot)}{\bar X_\theta}}_{L^\infty[0,s]}&\lesssim_\theta e^{-\frac{1}{2}(s-\tau)} (\frac{\log \tau}{\tau})^2.
    \end{split}
\end{equation}
Using the integrability of $\phi$ and $\epsilon$ \eqref{the estimate of int phi+epsilon tilde v}, as well as the uniform boundedness of two profiles, we estimate
\begin{equation}
    \begin{split}
        \abs{g(\tau)}&\leq \abs{f(\tau)\xi(\tau)-\frac{f_\infty}{2}}+ (\frac{\log \tau}{\tau}\norm{V_{1,\log}}_{L^\infty}+\frac{1}{\tau} \norm{V_1}_{L^\infty})(\int_0^\tau \phi dz+\int_0^\tau \abs{\epsilon(\tau,\cdot)}dz)\\
        &\leq C_{s_1,\theta} \frac{\log \tau}{\tau}+C_{s_1,\theta}(\frac{\log \tau}{\tau}+\frac{1}{\tau})(2+C\frac{(\log \tau)^2}{\tau})\lesssim_{s_1,\theta} \frac{\log \tau}{\tau}\\
        &\norm{\frac{\S_\mathcal{J}(\tau,s) \partial_z^{-1}\phi}{\bar X_\theta}}_{L^\infty([0,s])}\lesssim_\theta \exp\left({-\frac{1}{2}(s-\tau)}\right).
    \end{split}
\end{equation}
Using $V_{1,\log},V_1$ is $O(z)$ at $z=0$ and uniformly bounded, we have
\begin{equation}
    \norm{\frac{\S_\mathcal{J}(\tau,s) V_{1,\log}}{\bar X_\theta}}_{L^\infty([0,s])}, \norm{\frac{\S_\mathcal{J}(\tau,s) V_{1}}{\bar X_\theta}}_{L^\infty([0,s])}\lesssim_{s_1,\theta}\exp\left({-\frac{1}{2}(s-\tau)}\right)
\end{equation}
Finally, the equation of $\tilde r$ satisfies for $s\geq s_1$ and $s_1$ sufficiently large (may depend on $\theta$)
\begin{equation}\label{second iteration preliminary estimate}
\begin{split}
     \norm{\frac{\tilde r(s,\cdot)}{\bar X_\theta}}_{L^\infty[0,s]}&\leq \norm{\frac{\tilde r(s_1,\cdot)}{\bar X_\theta}}_{L^\infty[0,s_1]} e^{-\frac{1}{2}(s-s_1)}+ C_\theta \int_{s_1}^s \norm{\frac{\tilde r(\tau,\cdot)}{\bar X_{\theta}}}_{L^\infty([0,\tau])}  \tau^{-1} e^{-\frac{1}{2}(s-\tau)}d\tau\\
     &\quad+C_{s_1,\theta}\int_{s_1}^s \frac{(\log \tau)^3}{\tau^2} e^{-\frac{1}{2}(s-\tau)}d\tau\\
     &\leq\norm{\frac{\tilde r(s_1,\cdot)}{\bar X_\theta}}_{L^\infty[0,s_1]} e^{-\frac{1}{2}(s-s_1)}+ \frac{1}{4}\int_{s_1}^s \norm{\frac{\tilde r(\tau,\cdot)}{\bar X_{\theta}}}_{L^\infty([0,\tau])}   e^{-\frac{1}{2}(s-\tau)}d\tau\\
     &\quad+C_{s_1,\theta}\int_{s_1}^s \frac{(\log \tau)^3}{\tau^2} e^{-\frac{1}{2}(s-\tau)}d\tau
\end{split}
\end{equation}

\subsubsection{The decay of remainder}
\begin{corollary}\label{cor:third tangential refined estimate}
 Under the assumption of Lemma \ref{continuity argument new}, for $s_1$ sufficiently large,
 \begin{equation}
      \norm{\frac{\tilde r(s,\cdot)}{\bar X_\theta}}_{L^\infty[0,s]}\lesssim_{s_1,\theta}\frac{(\log s)^3}{s^2} \text{ for all }s\geq s_1.
 \end{equation}
\end{corollary}
\begin{proof}
    Applying Gronwall's lemma to \eqref{second iteration preliminary estimate}, we find that
    \begin{equation}
       \begin{split}
           \norm{\frac{\tilde r(s,\cdot)}{\bar X_\theta}}_{L^\infty[0,s]}&\leq \norm{\frac{\tilde r(s_1,\cdot)}{\bar X_\theta}}_{L^\infty[0,s_1]} e^{-\frac{1}{4}(s-s_1)}+C_{s_1,\theta}\int_{s_1}^s \frac{(\log \tau)^3}{\tau^2} e^{-\frac{1}{4}(s-\tau)}d\tau \lesssim_{\theta,s_1} \frac{(\log s)^3}{s^2}
       \end{split}
    \end{equation}
    where for the first term we used that $\tilde r(s_1,\cdot)=\tilde v(s_1,\cdot)-\frac{\log s_1}{s_1}V_{1,\log}-\frac{1}{s_1}V_1$ with $\tilde v(s_1),V_{1,\log},V_1\in \bar X_\theta L^\infty([0,s_1])$.
    \end{proof}

\begin{remark}\label{rmk:third tangential independence of limit}
    The value of \(f_\infty\) is in fact independent of
\(s_1\). Indeed, changing \(s_1\) only splits the same Duhamel integral
at a different time, and the integrating factor identity gives the same limit.
Thus \(f_\infty\) is fixed by the solution, equivalently by the initial
data. We suppress the dependence of $f_\infty$ from now on.
\end{remark}
{\begin{remark}\label{rmk:nonzero limit}
    For any fixed initial time $s_1$ sufficiently large, there exists an mean-free initial condition $G^*(\cdot)=\phi(\cdot)-\frac{1-e^{-s_1}}{s_1}$ at $s_1$ such that $f_\infty[G^*]>0$, where $f_\infty[u_0]=\lim_{s\rightarrow \infty} F_{1,0}(s,z=0)$ with $F_{1,0}(s_1,\cdot)=u_0$. Furthermore, letting $\mathcal A_{s_1}$ be the vector space of
mean-free admissible initial profiles $G$ on $[0,s_1]$ for which the
norm $\|G\|_{\mathcal A_{s_1}}
 :=|G\rvert_{z=0}|+
 \left\|({G-G\rvert_{z=0}\phi}){\bar X_\theta^{-1}}
 \right\|_{L^\infty((0,s_1])}<\infty,$
then the kernel set $\ker (f_\infty)=\{G\in\mathcal A_{s_1}:\;f_\infty[G]=0\}$ is a closed codimension-one subspace of $\mathcal A_{s_1}$, and its complement is open and dense.
\end{remark}
\begin{proof}
   In view of the decomposition \eqref{eq: third tangential crude decomp}, at $s=s_1$, for this mean-free initial condition
   \begin{align*}
   F_{1,0}(s_1,z)&=e^{-z}-\frac{1-e^{-s_1}}{s_1},\quad
 f(s_1)=1-\frac{1-e^{-s_1}}{s_1},\quad
 \tilde v(s_1,z)
 =-\frac{1-e^{-s_1}}{s_1}(1-e^{-z}).
\end{align*}
In light of the energy estimate of $H(s)$ \eqref{eq:intermediate energy estimate}, the bound on $\mathcal R_\infty(s_1)$ \eqref{eq:bound on R_inf} and the definition of $\bar X_\theta$ \eqref{weights for tilde epsilon new derivative}, combining with the definition of energy \eqref{eq:definition of energy kernel}, for this initial condition, we have for all $s\geq s_1$
\begin{equation}\label{equation:modulationandv}
 |f(s)|\lesssim_\theta1,
 \qquad
 \|\tilde v(s,\cdot)\bar X_\theta^{-1}\|_{L^\infty([0,s])}
 \lesssim_\theta\frac{\log s}{s}.
\end{equation}
In view of the formula for $f_\infty$ \eqref{eq:formula of f_infty}, 
\begin{equation}
\begin{split}
      f_\infty[G^*]&=R_\infty(s_1)f(s_1)+O_\theta(\int_{s_1}^\infty\frac{\log\tau}{\tau^2}\,d\tau)=1+O_\theta\left(\frac{\log s_1}{s_1}\right).
\end{split}
\end{equation}
where we have used \eqref{integral of phi plus epsilon times tilde v} and simplifies \eqref{eq: asymptote of R_inf} to $ \mathcal R_\infty(s_1)=1+O\left(\frac{\log s_1}{s_1}\right)$. Thanks to the linearity of the operator $f_\infty: u_0\mapsto f_\infty[u_0]$, using the same energy estimates \eqref{eq:intermediate energy estimate}, we see for any $G\in\mathcal A_{s_1}$, its corresponding modulation parameters satisfy $\abs{f(s)}\lesssim_{s_1,\theta} \norm{G}_{A_{s_1}}, \norm{\frac{\tilde v(s,\cdot)}{\bar X_\theta}}_{L^\infty[0,t]}\lesssim_{s_1,\theta}\frac{\log s}{s}  \norm{G}_{A_{s_1}}\,\text{for }s_1\leq s.$ In light of the formula of $f_\infty$ \eqref{eq:formula of f_infty}, $|f_\infty[G]|\leq C_{s_1,\theta}\|G\|_{\mathcal A_{s_1}}$ and thus $f_\infty$ is a continuous operator. The desired result follows as we have already shown that $f_\infty[G^*]\neq 0$.
\end{proof}
}

Finally, in view of Corollary \ref{second iteration convergence of modulation parameter}, Corollary \eqref{cor:third tangential refined estimate}, Remark \ref{rmk:third tangential independence of limit}, as well as the decompositions \eqref{eq;third tangential decomp on remainder} and \eqref{eq: third tangential crude decomp}, Theorem \ref{asymptotic expansion of F_1,0 recall} is proven.

\section{Analytical Boundedness and Extension}\label{sec: analytical extension}
This section corresponds to the third objective mentioned in Section \ref{sec:equations and goal}. Recall that in the last two sections, we have extracted precise quantitative information on the asymptotic behavior of each mode for $F_{i,k}$'s. However, this fixed-order result does not control the entire infinite hierarchy. For example, at a given sufficiently large $s$, independent of $i,k$, the higher order derivatives can be uncontrollably large. This section answers whether this situation is possible or not, by providing a qualitative result on the boundedness of weighted analytic norm.\par

{In addition, we show that the radius of convergence of the Taylor
series reconstructed from $(F_{i,k})$ remains bounded below by a
positive constant in the renormalized variables. Using the change
of variables \eqref{setting-renormalization}, we reconstruct, relying on the information on the axis,
encoded by the hierarchy $(F_{i,k})$, a local solution
of \eqref{2d primitive intro} that is analytic in a spatial
neighborhood of the symmetry axis for all times before $T$.
The horizontal analytic radius of this local solution is bounded below by a constant
multiple of $(T-t)^{3/2}$.}\par
{Any other local analytic solution with the same initial data near
the axis agrees with the reconstructed solution wherever both
solutions are defined. In particular, if the initial data
agree near the axis, the reconstructed solution coincides near
the axis with the full solution constructed in
\cite{Collot2023} for as long as the full solution exists. However, the full solution may still
become singular away from the axis before time $T$.}


{Our approach to spatial analytic extension is based on a priori estimates for the hierarchy of derivatives along the symmetry axis. Using the established asymptotics of the background and first-shell modes as inputs, we derive weighted bounds for the higher modes under a smallness condition on their initial data and their source. We then show that the required initial-data conditions are satisfied by a nonempty class of analytic data.}

Except for the order of induction, this section is methodologically fundamentally different than the earlier sections. In particular, we shall treat the information of $F_{1,0}$ and $F_{0,1}$ as background information, and the induction starts at the second shell $i+k=2$. The other difference is that as the expected result needs to be uniform in $i,k$, we shall therefore also be proving various combinatorial estimates. 

Under the assumption of the smallness of initial analytic norm at $s_0$, we shall first prove a form of comparison principles that argues the solution size is controlled whenever the source $N_{i,k}$ is. In the next step, we shall prove that the source can indeed be controlled whenever the induction hypothesis is satisfied.

\subsection{Known information and objective}
\subsubsection{Asymptotic information on first order derivatives}
We begin by summarizing what we have shown so far. From the preceding sections, we know that there exists $s^{**}\geq s^*$ sufficiently large such that the following asymptotic relations hold.
\begin{enumerate}
    \item For $s\geq s^*$, the correctional term $\epsilon$
\begin{equation}\label{eq:eps-expansion}
\eps(s,z)
=
\frac{\log s}{s}\eps_{1,\log}(z)+\frac{1}s\eps_1(z)+R_\eps(s,z),
\qquad
\norm{R_{\eps}(s,\cdot)}_{C^1(0,s)}\lesssim_{s^*} \frac{\log^2s}{s^2},
\end{equation}
where
\begin{equation}\label{eq:eps-profiles}
  \eps_{1,\log}(z)=-2ze^{-z},
  \qquad
  \eps_1(z)=\bigl(2+(1+C_\nu)z\bigr)e^{-z}-1.
\end{equation}
For simplicity, we define
\begin{equation}\label{eq:epsrem-def}
  \eps_{\rm rem}(s,z):=\eps(s,z)+\frac{1}s,
\end{equation}
so that $\eps_{\rm rem}$ vanishes at infinity and
\begin{equation}\label{eq:bg-model}
  \phi+\eps=e^{-z}-s^{-1}+\eps_{\rm rem}.
\end{equation}
\item The $(1,0)$-mode has the expansion
\begin{equation}\label{eq:F10-asymptotic-input}
\begin{aligned}
F_{1,0}(s,z)
={}&F_{1,0}|_{z=0}(s)\phi(z)
+\frac{\log s}{s}V_{1,\log}(z)
+\frac{1}sV_1(z)+\tilde r(s,z),\\
V_{1,\log}(z)={}&-2f_\infty z e^{-z},\\
V_1(z)={}&f_\infty\left(-1+C_\nu ze^{-z}-\frac{1}{2}e^{-z}+e^{-2z}+\frac{1}{2}e^{-3z}\right)
\end{aligned}
\end{equation}
for all \(s\ge s_1\geq s^{**}\). Furthermore,
\begin{equation}\label{eq:F10-asymptotic-remainder}
  |F_{1,0}|_{z=0}(s)-f_\infty|\le C(s_1)\frac{\log s}{s},
  \qquad
  \left\|\frac{\tilde r(s,\cdot)}{\partial_z^{-1}\phi}\right\|_{L^\infty[0,s]}
  \le C(s_1)\frac{(\log s)^3}{s^2}.
\end{equation}
 Equivalently, the bound used below is the pointwise estimate
\begin{equation}\label{eq:F10-remainder-pointwise}
  |\tilde r(s,z)|
  \le C(s_1)\frac{(\log s)^3}{s^2}(1-e^{-z}),
  \qquad 0\le z\le s.
\end{equation}
\item The $(0,1)$-mode has the expansion
\begin{equation}\label{eq:F01-asymptotic-input}
\begin{aligned}
F_{0,1}(s,z)
={}&(1+\epsilon_0(s))\phi(z)
+\frac{\log s}{s}\Phi_{0,1}(z)
+\frac{1}s\Psi_{0,1}(z)+r_{0,1}(s,z),\\
\Phi_{0,1}(z)
={}&-\partial_zU_{1,\log}(z)=2e^{-z}-2ze^{-z},\\
\Psi_{0,1}(z)
={}&-\phi(z)-\partial_zU_1(z)=\bigl((1+C_\nu)z-C_\nu\bigr)e^{-z},
\end{aligned}
\end{equation}
for $s\geq s^*$, where $\epsilon_0=\eps(s,z=0)$ satisfying
\begin{equation}\label{eq:epsilon0-input}
  |\epsilon_0(s)|\simeq \frac{1}{s}\lesssim_{s^*}\frac{\log s}{s},
\end{equation}
Here the constant \(-1\) in \(U_1\) disappears after one \(z\)-derivative.  The
remainder satisfies
\begin{equation}\label{eq:F01-remainder-input}
  |r_{0,1}(s,z)|\lesssim_{s^*} \frac{(\log s)^2}{s^2},
  \qquad 0\le z\le s
\end{equation}
This is an immediate consequence of the \(C_z^1\) estimate in
\eqref{eq:eps-expansion}.
\end{enumerate}
In the following sections, we treat $s^*,s^{**}$ as universal constants.

\subsubsection{Tentative weights, analytic Radius and objective}
 In general, we expect that, based on Theorem \ref{refined C1 expansion} and Theorem \ref{first shell mode expansions}, the higher order profiles behave similarly to the first-shell modes. In particular, at the scale of $z\approx 1$, we expect that for all $i,k$'s, $F_{i,k}\approx \pm\phi$, while at the scale $z\approx s$, the decay for $F_{i,k}$ ($k\geq 1$) should be strictly faster than $\frac{1}{s}$, while the decay for $F_{i,0}$ ($k=0$) should be of order $\frac{1}{s}$. This motivates the introduction of the weight 
\begin{equation}\label{eq:hk-wk}
  \hh_k:=
  \begin{cases}
  \hh_0, & k=0,\\
  \hh_+, & k\ge1,
  \end{cases}
  \qquad
  \wh_k:=\hh_k^{-1},
\end{equation}
\begin{equation}\label{eq:mixed-h}
  \hh_0(s,z):=e^{-z}+s^{-1},
  \qquad
  \hh_+(s,z):=e^{-z}+\Lambda^2(\log s)^2/s^2.
\end{equation}
Here, we expect the positive modes $F_{i,k}$ ($k\geq 1$) to be of order $(\log s)^2/s^2$ at spatial scale $z\approx s$. This weight is accompanied by an analytical radius
\begin{equation}\label{eq:rho}
  -\frac{\rho'}{\rho}=C_0\Lambda^2(\log s)^2/s^2,
  \qquad \rho(s_0)=\rho_0,
\end{equation}
{where $C_0,\Lambda,\rho_0>0$ are parameters we shall choose as the analysis progresses.}
Equivalently,
\begin{equation}\label{eq:rho-explicit}
  \rho(s)=\rho_0\exp\left(-C_0\Lambda^2\int_{s_0}^s
  \frac{(\log \sigma)^2}{\sigma^2}\dd \sigma\right).
\end{equation}

This radius also comes with additional parameters. In particular, at $s_0 \gg 1$, the conditions on $F_{i,k}$'s are imposed and we shall show that for $s\geq s_0$, these bounds remain satisfied. However, we shall consider the equations of a normalized version of $F_{i,k}$'s. In that setting, the parameter $\Lambda, C_0\gg 1$ serves as a margin for positive modes: it ensures that the influence from the detrended primitive of $\eps$: $ |\pzi\eps_{\rm rem}(s,z)|$ term, measured relative to \(h_+\), is absorbed respectively by the original damping terms and the damping terms resulting from the radius, so that the damping coefficient is uniformly bounded below by some preselected scale. 

Notice that although the radius is specifically chosen to decay, when one considers a normalized version of $F_{i,k}$, this decay introduces terms favoring damping. On the other hand, since the integrand \eqref{eq:rho-explicit} is integrable, $\rho(s)$ therefore remains of constant order. The parameter $0<\rho_0\ll 1$ therefore serves as a scale for analytic radius. Indeed, if we choose sufficiently large $s_0$, then by explicit computation on \eqref{eq:rho-explicit}
\begin{equation}\label{eq:bound on analytical radius}
 s_0\geq s_{\text{rad}}:=\max\left(1,\left(\frac{4C_0 \Lambda^2}{\log 2}\right)^2\right)\implies  C_0\Lambda^2\frac{(\log s_0)^2+2\log s_0+2}{s_0}\le \log2 \implies \frac{\rho_0}{2}\leq \rho(s) \leq \rho_0.
\end{equation} 
This parameter also allows us to absorb the influence of $F_{1,0}$ and $F_{0,1}$ resulting from the non-linear source term. We shall also introduce a quantity 
\begin{equation*}
     B_1(s_1,\Lambda):=
  \sup_{s\ge s_1}
  \left(
  \norm{\wh_0F_{1,0}(s)}_{L^\infty_z}
  +\norm{\wh_+F_{0,1}(s)}_{L^\infty_z}
  \right)
\end{equation*}
measuring the influence of shell-one background modes. Here, after $\Lambda$ is fixed, $s_1\leq s_0$ is then fixed as an auxiliary entrance time where the background asymptotes and the comparison between two weights ($h_0,h_+$) are simultaneously available. The comparability of the two weights allows us to better control the non-linear source.

Our goal is to show that there exists a family of parameters, as well as $s_0$ independent of $i,k$, such that when $i+k\geq 2$, as long as
\begin{equation}
    |F_{i,k}(s_0,z)|\le 
  M\frac{(2i+k+1)!}{(1+i+k)^{a+1}\rho_0^{2i+k-1}}\,\hh_k(s_0,z),\qquad s=s_0
\end{equation}
is satisfied, then
\begin{equation}\label{eq:raw-step2-target}
  |F_{i,k}(s,z)|
  \le
  M\frac{(2i+k+1)!}{(1+i+k)^{a+1}\rho(s)^{2i+k-1}}\,\hh_k(s,z),
  \qquad s\ge s_0,
\end{equation}
where \(\hh_k\) is the zero/positive tail scale defined in \eqref{eq:hk-wk}. Here, $a\geq 1$ serves as a Sobolev corrector term, allowing us to estimate decay of certain combinatorial identities, which helps control the nonlinear source, ensuring the induction can be closed. {It will soon be fixed it to $3$, after some combinatorial inequalities are derived.} The parameter $M$ serves as a scale for solution size. Effectively, $\eps,F_{1,0}$ and $F_{0,1}$ are treated as background information for $i+k=1$, and the induction is performed on all the active shells, that is, when $i+k\geq 2$.

\subsubsection{Normalized equations}
For active modes \(i+k\ge2\), define
\begin{equation}\label{eq:Aik-def}
  A_{i,k}(s):=
  \frac{(1+i+k)^{a+1}\rho(s)^{2i+k-1}}{(2i+k+1)!},
  \qquad
  U_{i,k}:=A_{i,k}\wh_kF_{i,k}.
\end{equation}
Then the goal \eqref{eq:raw-step2-target} is exactly equivalent to the
normalized target
\begin{equation}\label{eq:active-target}
  \norm{U_{i,k}(s)}_{L^\infty_z}\le M,
  \qquad i+k\ge2.
\end{equation}
This is the same as the raw pointwise form
\begin{equation}\label{eq:active-raw-bound-form}
  |F_{i,k}(s,z)|\le M A_{i,k}(s)^{-1}\hh_k(s,z)\qquad i+k\ge2.
  \quad s\ge s_0,
  \quad 0\le z\le s,
\end{equation}
or equivalently
\begin{equation}\label{eq:active-A-bound-form}
  A_{i,k}(s)\wh_k(s,z)|F_{i,k}(s,z)|\le M,
  \qquad i+k\ge2,
  \quad s\ge s_0,
  \quad 0\le z\le s.
\end{equation}

For every active pair \(i+k\ge2\), the normalized unknown
\(U_{i,k}=A_{i,k}\wh_kF_{i,k}\) satisfies
\begin{equation}\label{eq:normalized-pde}
\begin{aligned}
\partial_sU_{i,k}
&+\left(\pzi(\phi+\eps)+\frac{z}s\right)\partial_zU_{i,k}
+G_{i,k}U_{i,k} \\
&=\one_{k=0}\mathcal K_i[U_{i,0}]+S_{i,k}.
\end{aligned}
\end{equation}
The terms in \eqref{eq:normalized-pde} are defined as follows.  The new source is the
normalized nonlinear source of the original hierarchy
\begin{equation}\label{eq:normalized-source-def}
  S_{i,k}:=A_{i,k}\wh_kN_{i,k},
\end{equation}
where \(N_{i,k}\) is defined in \eqref{eq:Nik}.  The effective
zeroth-order damping coefficient is
\begin{equation}\label{eq:Geff-def}
\begin{aligned}
G_{i,k}:={}&
\frac{k}{s}+3i+1-(2i+2-k)(\phi+\eps)
-(2i+k-1)\frac{\rho'}{\rho}
+\frac{\partial_s\hh_k}{\hh_k} \\
&-\left(\pzi(\phi+\eps)+\frac{z}s\right)\frac{e^{-z}}{\hh_k}.
\end{aligned}
 \end{equation}
Here \(\rho\), \(A_{i,k}\), \(\hh_k\), and \(\wh_k=\hh_k^{-1}\) are the quantities
introduced in the last subsection.
For zero modes \(k=0\), \(i\ge2\), the normalized linear nonlocal operator is
\begin{equation}\label{eq:normalized-zero-operator}
  \begin{aligned}
\mathcal K_i[V](s,z)
={}&(2i+1)\wh_0(s,z)(\phi-\eps_z)(s,z)
      \int_0^z \hh_0(s,y)V(s,y)\dd y  \\
&-4(i+1)\wh_0(s,z)s^{-1}
      \int_0^s (\phi+\eps)(s,y)\hh_0(s,y)V(s,y)\dd y.
\end{aligned} 
\end{equation}
The operator \(\mathcal K_i\) appears only when \(k=0\).

\subsubsection{Nested induction}\label{subsec: nested induction}
For notational simplicity, we set
\[
  \mathcal S_m:=\{(j,\ell):j+\ell=m\}.
\]
As the induction starts at the second shell, for \(m\ge2\), define
\begin{equation}\label{eq:Pm}
  \mathsf P_m:
  \qquad
  \norm{U_{j,\ell}(s)}_{L^\infty_z}\le M
  \quad\text{for all }(j,\ell)\in\mathcal S_m,
  \quad s\ge s_0.
\end{equation}
The outer induction is performed on the shell \(m=j+\ell\), and inside each fixed shell we introduce another level of induction, attempting to prove the target bound in the order of
\begin{equation}\label{eq:shell-order}
  (0,m),\ (1,m-1),\ (2,m-2),\ldots,\ (m,0).
\end{equation}
Since the $F_{0,m}$ mode does not contain same-shell term, when the current pair is \((i,k)\in\mathcal S_m\), all active shells
already satisfy the induction hypothesis, and all the same-shell modes \((j,m-j)\) with \(j<i\) have already been proven.
\begin{center}
\begin{minipage}{\linewidth}
\centering
\begin{tikzpicture}[
    scale=0.78,
    every node/.style={font=\scriptsize},
    knownlower/.style={circle, fill=blue!35, draw=blue!70!black, inner sep=2pt},
    same/.style={circle, fill=green!40, draw=green!50!black, inner sep=2pt},
    current/.style={diamond, aspect=1, fill=white, draw=red!80!black, line width=1.1pt, inner sep=2.5pt},
    future/.style={rectangle, minimum size=6pt, fill=white, draw=gray!65, line width=0.7pt, inner sep=0pt},
    base/.style={regular polygon, regular polygon sides=3, minimum size=7pt, inner sep=0pt, fill=violet!50, draw=violet!80!black},
    basecase/.style={circle, fill=orange!45, draw=orange!80!black, inner sep=2pt}
]

\begin{scope}
\node[font=\small] at (7.0,7.85) {Mode availability when estimating $F_{i,k}$ ($m=i+k$)};

\draw[->] (-0.45,0) -- (7.35,0) node[right] {$i$};
\draw[->] (0,-0.45) -- (0,7.05) node[above] {$k$};

\foreach \m in {2,3,5}{
  \draw[dashed, gray!25] (-0.08,\m+0.08) -- (\m+0.08,-0.08);
}
\node[gray!55, rotate=-35] at (2.0,2.0) {$\cdots$};
\node[gray!55, rotate=-35] at (2.55,2.55) {$\cdots$};

\foreach \m in {2,3}{
  \foreach \ii in {0,...,\m}{
    \pgfmathtruncatemacro{\kk}{\m-\ii}
    \node[knownlower] at (\ii,\kk) {};
  }
}
\foreach \ii/\kk in {0/5,1/4,4/1,5/0}{
  \node[knownlower] at (\ii,\kk) {};
}

\foreach \ii in {0,1,2}{
  \pgfmathtruncatemacro{\kk}{2-\ii}
  \node[basecase] at (\ii,\kk) {};
}

\node[base,label=below:{$F_{1,0}$}] at (1,0) {};
\node[base,label=left:{$F_{0,1}$}] at (0,1) {};

\draw[dashed, red!65] (-0.1,6.1) -- (6.1,-0.1);
\node[same] at (0,6) {};
\node[same] at (1,5) {};
\foreach \ii/\kk in {1.82/4.18,2/4,2.18/3.82}{
  \node[same, inner sep=1pt] at (\ii,\kk) {};
}
\node[same,label=above left:{$F_{i-1,k+1}$}] at (3,3) {};
\node[current,label=above right:{$F_{i,k}$}] (curL) at (4,2) {};
\node[future] at (5,1) {};
\foreach \ii/\kk in {5.52/0.48,5.7/0.3,5.88/0.12}{
  \node[future, minimum size=3pt] at (\ii,\kk) {};
}

\node at (6.65,0.15) {$\cdots$};
\node at (0.15,6.68) {$\vdots$};
\end{scope}

\begin{scope}[xshift=9.0cm]
\node[anchor=west, font=\scriptsize\bfseries] at (-0.15,6.4) {Filled markers: estimates available};
\node[knownlower] at (0,5.85) {};
\node[right] at (0.25,5.85) {modes in shells \(\le m-1\)};

\node[basecase] at (0,5.35) {};
\node[right] at (0.25,5.35) {second shell \(i+k=2\), base case within lower shells};

\node[same] at (0,4.85) {};
\node[right] at (0.25,4.85) {earlier modes in the \(m\)-th shell};

\node[base] at (0,4.35) {};
\node[right] at (0.25,4.35) {shell-one background information};

\node[anchor=west, font=\scriptsize\bfseries] at (-0.15,3.7) {Outlined diamond: estimate in progress};
\node[current] at (0,3.15) {};
\node[right] at (0.25,3.15) {current mode $F_{i,k}$};

\node[anchor=west, font=\scriptsize\bfseries] at (-0.15,2.55) {Outlined squares: estimates not yet available};
\node[future] at (0,2.05) {};
\node[right] at (0.25,2.05) {later modes in the \(m\)-th shell};

\draw[dashed, red!65] (-0.15,1.35) -- (0.45,1.35);
\node[align=left, anchor=west] at (0.65,1.35)
{current shell \(i+k=m\)};

\draw[dashed, gray!60] (-0.15,0.8) -- (0.45,0.8);
\node[align=left, anchor=west] at (0.65,0.8)
{lower shell \(i+k=n,\ n\le m-1\)};
\end{scope}

\end{tikzpicture}

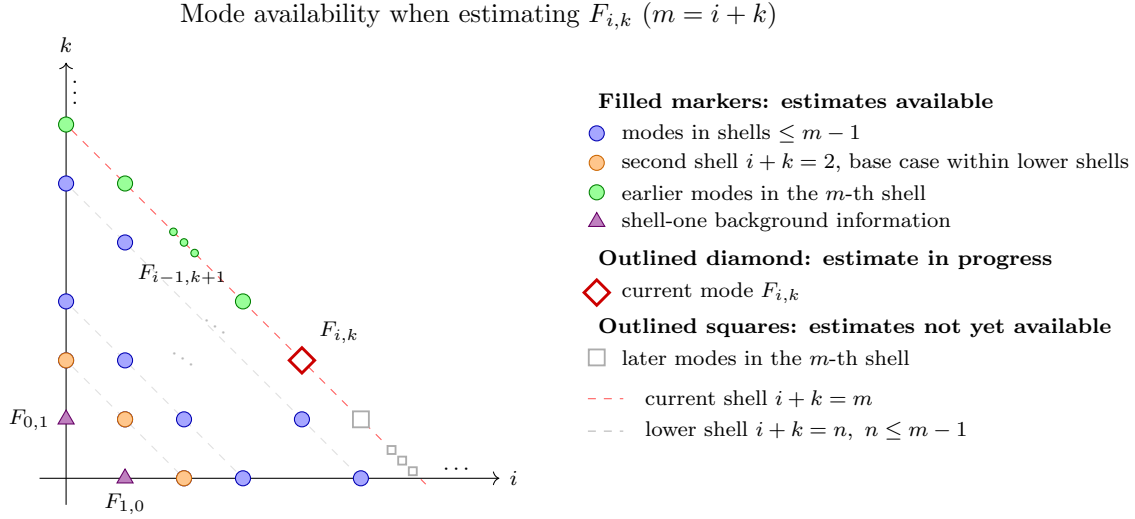
\captionof{figure}{Mode availability when estimating \(F_{i,k}\)
in the nested induction, where \(m=i+k\).}
\label{fig:mode-availability}
\end{minipage}
\end{center}
In other words, thanks to the lower-triangular dependence, if one attempts to show $F_{i,k}$ with $i+k=m\geq 3$ satisfies this bound, then in this nested order, by \eqref{eq:active-raw-bound-form}, the factor already satisfies
\begin{equation}\label{eq:available-active-pointwise}
\begin{split}
     |F_{j,\ell}(s,z)|\le M A_{j,\ell}(s)^{-1}\hh_\ell(s,z) \text{ for }(2\leq j+l\leq m-1) \text{ or }(j+l=m, 0\leq j\leq i-1)
\end{split}
\end{equation}
The base shell then consists of
\[
  \mathcal S_2=\{(0,2),(1,1),(2,0)\},
\]
and it is also proven in the order
\[
  (0,2)\longrightarrow(1,1)\longrightarrow(2,0).
\]
We will soon define $B_1(\Lambda,s_1)$ and show that the shell-one input factors satisfy
\begin{equation}\label{eq:available-base-pointwise}
  |F_{1,0}(s,z)|\le B_1\hh_0(s,z),
  \qquad
  |F_{0,1}(s,z)|\le B_1\hh_+(s,z).
\end{equation}
{For simplicity, we refer the modes $F_{j,l}$ with $j+l\geq 2$ as ``active'', whereas ``positive'' or ``zero'' modes stand for the mode $F_{j,l}$ with $l\geq 1$ or $l=0$, respectively.}

\subsection{Immediate consequences from background information}
For fixed $\Lambda\geq 1$, whose range will be given later, we start by fixing $s_1\geq s_{\text{ent}}(\Lambda)$ such that
\begin{equation}\label{eq:s1-requirements}
\begin{aligned}
s_{\text{ent}}(\Lambda)=\max(s^*,s^{**},e,4\Lambda^2 (\log(4\Lambda^2))^2),
\end{aligned}
\end{equation}
then $\Lambda^2\frac{(\log s)^2}{s^2}\le \frac1s$ for all $s\geq s_1$, and the asymptotic expansions and remainder estimates on $\eps$ \eqref{eq:eps-expansion}, on $F_{1,0}$ \eqref{eq:F10-asymptotic-input},\ \eqref{eq:F10-asymptotic-remainder}, and on $F_{0,1}$ \eqref{eq:F01-asymptotic-input},\ \eqref{eq:F01-remainder-input} 
simultaneously hold for $s\geq s_1$. The following lemma translates the asymptotic expansions into spatially uniform or pointwise estimates that are inserted later to obtain a lower bound on the damping coefficient.

\subsubsection{Background information related to the weights}
\begin{corollary}[Estimates for the detrended $\eps$: $\eps_{\rm rem}$]\label{lem:epsrem-estimates}
{ For every sufficiently large lower threshold \(S\), the remainder
\(\eps_{\rm rem}\) from \eqref{eq:epsrem-def} satisfies, uniformly for
\(s\ge S\geq {S^*(\Lambda)}\),}
\begin{equation}\label{eq:epsrem-pointwise-split}
  |\eps_{\rm rem}(s,z)|
  \le o_S(1)e^{-z}+C\frac{(\log s)^2}{s^2}
  \le o_S(1)e^{-z}+C\Lambda^2\frac{(\log s)^2}{s^2},
  \qquad 0\le z\le s.
\end{equation}
Here, $C$ is universal constant, and \(o_S(1)\to0\) denotes a quantity depending only on the lower threshold $S$, that can be made arbitrarily small by
increasing the temporary threshold \(S\). 
Moreover, globally its primitive satisfies
\begin{equation}\label{eq:epsrem-primitive-small}
  \sup_{0\le z\le s}|\pzi\eps_{\rm rem}(s,z)|
  \le C\frac{\log^2s}{s}=o_S(1),
\end{equation}
and point-wise it can be bounded:
\begin{equation}\label{eq:epsrem-localized-small}
  |\pzi\eps_{\rm rem}(s,z)|\frac{e^{-z}}{e^{-z}+\Lambda^2(\log s)^2/s^2}
  \le
  \frac{C}{\Lambda}\left(e^{-z}+\frac{\Lambda^2(\log s)^2/s^2}{e^{-z}+\Lambda^2(\log s)^2/s^2}\right),
  \qquad 0\le z\le s.
\end{equation}
Finally, the derivative satisfies
\begin{equation}\label{eq:Ez-weighted-bound}
  \sup_{0\le z\le s}
  \frac{|\eps_z(s,z)|}{e^{-z}+s^{-1}}
  \le C\frac{\log^2s}{s}.
\end{equation}
\end{corollary}

\begin{proof}
Using the formula for \(\eps_{\rm rem}\) above and the remainder bound
in \eqref{eq:eps-expansion}, we see that for \(s\ge \max(e,s^*)\),
\[
\begin{aligned}
  |\eps_{\rm rem}(s,z)|
  &\le
  \frac{e^{-z}}s\Bigl(2+|1+C_\nu|z+2z\log s\Bigr)
  +C\frac{(\log s)^2}{s^2} \\
  &\le
  C\frac{(1+z\log s)e^{-z}}s
  +C\frac{(\log s)^2}{s^2}.
\end{aligned}
\]
We claim that \((z\log s/s)e^{-z}\leq o_S(1)e^{-z}+C\frac{(\log s)^2}{s^2}\). Indeed, splitting the interval \(0\le z\le s\) into $\Omega_1:=\left\{0\le z\le \frac{\sqrt{s}}{\log s}\right\}$ and $\Omega_2:=\left\{\frac{\sqrt{s}}{\log s}\le z\le s\right\}$, we see that
\begin{equation}(z\log s/s)e^{-z}\leq
    \begin{cases}
         \frac{z\log s}{s}e^{-z}\le s^{-1/2}e^{-z}
  \le S^{-1/2}e^{-z} &\text{ on }\Omega_1\\
  \frac{z\log s}{s}e^{-z}
  \le \log s\,e^{-\sqrt{s}/\log s}\leq Cs^{-2} &\text{ on }\Omega_2,
    \end{cases}
\end{equation}
which proves \eqref{eq:epsrem-pointwise-split}.
For the primitive estimate, integrate the formula for \(\eps_{\rm rem}\) and use the bound on \(R_\epsilon\), we see that for every \(0\le z\le s\),
\begin{equation}\label{eq:epsrem-primitive-decomp}
\begin{aligned}
  |\pzi\eps_{\rm rem}(s,z)|
  &\le
  \frac{C}{s}\int_0^z (1+y\log s)e^{-y}\,dy
  +C\frac{(\log s)^2}{s^2}z  \\
  &\leq C\frac{\log s}{s}+C\frac{(\log s)^2}{s}\lesssim \frac{(\log s)^2}{s},
\end{aligned}
\end{equation}
which proves \eqref{eq:epsrem-primitive-small}. It remains to prove the localized bound for the primitive. For simplicity, let us introduce $A(s):=\Lambda^2\frac{(\log s)^2}{s^2}$. Multiplying the second last inequality in decomposition \eqref{eq:epsrem-primitive-decomp} by
\(e^{-z}/(e^{-z}+A(s))\) gives
\begin{equation}\label{eq:epsrem-localized-decomp}
  \begin{aligned}
  |\pzi\eps_{\rm rem}(s,z)|\frac{e^{-z}}{e^{-z}+A(s)}
  &\le C\frac{\log s}{s}
  +C\frac{(\log s)^2}{s^2}
    z\frac{e^{-z}}{e^{-z}+A(s)}.
\end{aligned} 
\end{equation}
We claim that:
\begin{equation}\label{eq:z-exp-over-floor}
  \sup_{0\le z\le s}
  (1+z)\frac{e^{-z}}{e^{-z}+A(s)}
  \le C\log s.
\end{equation}
Indeed, this can be derived by noticing $\Lambda\geq 1$ and
\begin{equation*}
z\frac{e^{-z}}{e^{-z}+A(s)}\le
    \begin{cases}
         z\le  3\log s &\text{ if }z\le3\log s\\
         z\frac{e^{-z}}{A(s)}
  \le s\frac{s^{-3}}{\Lambda^2(\log s)^2/s^2}
  =\frac{1}{\Lambda^2(\log s)^2} &\text{ if }3\log s<z\le s,
    \end{cases}
\end{equation*}
Substituting
\eqref{eq:z-exp-over-floor} into \eqref{eq:epsrem-localized-decomp} gives
\begin{equation}\label{eq:epsrem-localized-left-bound}
  |\pzi\eps_{\rm rem}(s,z)|\frac{e^{-z}}{e^{-z}+A(s)}
  \le C\frac{\log s}{s}+C\frac{(\log s)^3}{s^2}
  \le C\frac{\log s}{s},
\end{equation}
upon increasing \(S\).
Next, observing that by increasing \(S\) so that \(0<A(s)\le1\), one obtains
\begin{equation}\label{eq:localized-weight-lower}
\begin{split}
     e^{-z}+\frac{A(s)}{e^{-z}+A(s)}&= e^{-z}+A(s)+\frac{A(s)}{e^{-z}+A(s)}-A(s)\\
  &\ge 2\sqrt{A(s)}-A(s)
  \ge \sqrt{A(s)}
  =\Lambda\frac{\log s}{s}.
\end{split}
\end{equation}
Combining \eqref{eq:epsrem-localized-left-bound} with
\eqref{eq:localized-weight-lower} proves \eqref{eq:epsrem-localized-small},
after enlarging the uniform constant.
Finally, for the derivative bound, using the explicit profiles in \eqref{eq:eps-expansion} and the estimate on the remainder, one can show
\begin{equation}\label{eq:epsz-raw-bound}
  |\eps_z(s,z)|
  \le C\frac{\log s}{s}(1+z)e^{-z}
  +C\frac{1}{s}(1+z)e^{-z}
  +C\frac{(\log s)^2}{s^2}.
\end{equation}
Dividing \eqref{eq:epsz-raw-bound} by \(e^{-z}+s^{-1}\), and using
\eqref{eq:z-exp-over-floor}, gives
\[
\begin{aligned}
  \frac{|\eps_z(s,z)|}{e^{-z}+s^{-1}}
  &\le
  C(\frac{\log s}{s}+\frac1s)
  \frac{(1+z)e^{-z}}{e^{-z}+s^{-1}}
  +C(\log s)^2\frac{s^{-2}}{e^{-z}+s^{-1}} \\
  &\lesssim
  \frac{(\log s)^2}{s}
  +\frac{\log s}{s}
  +\frac{(\log s)^2}{s}\lesssim \frac{(\log s)^2}{s}
\end{aligned}
\]
which in turn proves \eqref{eq:Ez-weighted-bound}.
\end{proof}
The two background modes $F_{1,0}$ and $F_{0,1}$ are already known through their asymptotic expansions. The following corollary defines their combined weighed $L^\infty$ energies and is applied whenever a nonlinear term in the source contains one of those modes later.
\begin{corollary}[Well-definedness of the energies of shell-one background modes in weighted norm]
\label{lem:shell-one-asymptotic-bounds}
After fixing \(\Lambda\ge1\), choose \(s_1\) so that \eqref{eq:s1-requirements} hold. That is, $s_1\geq s_{\text{ent}}(\Lambda)$.  Then for all \(s\ge s_1\) and
\(0\le z\le s\),
\begin{equation}\label{eq:F10-zero-tail}
  |F_{1,0}(s,z)|\lesssim_{s_1,\Lambda}\hh_0(s,z),
\end{equation}
\begin{equation}\label{eq:F01-positive-tail}
  |F_{0,1}(s,z)|
 \lesssim_{s_1,\Lambda}\hh_+(s,z).
\end{equation}
Consequently, the weighted norm 
\begin{equation}\label{eq:Bone-def}
  B_1(s_1,\Lambda):=
  \sup_{s\ge s_1}
  \left(
  \norm{\wh_0F_{1,0}(s)}_{L^\infty_z}
  +\norm{\wh_+F_{0,1}(s)}_{L^\infty_z}
  \right)
\end{equation}
for the shell-one modes is well-defined for fixed $\Lambda\geq 1$ and $s_1$ satisfying \eqref{eq:s1-requirements}. 
\end{corollary}

\begin{proof}
We first prove the bound for \(F_{1,0}\).  From \eqref{eq:F10-asymptotic-input},
\begin{equation}
  \begin{aligned}
|F_{1,0}(s,z)|
&\le |F_{1,0}|_{z=0}(s)|e^{-z}
 + C\frac{\log s}{s} z e^{-z} 
 + C\frac{1}s\bigl(1+ze^{-z}+e^{-z}+e^{-2z}+e^{-3z}\bigr)
 + |\tilde r(s,z)|.
\end{aligned} 
\end{equation}
Thanks to \eqref{eq:F10-asymptotic-remainder}, the coefficient of the leading term $|F_{1,0}|_{z=0}(s)|e^{-z}$ is bounded by \(C(s_1)e^{-z}\le C(s_1)\hh_0\) for \(s\ge s_1\). Next, we observe that
\begin{equation}
\label{eq:ze-tail-h0}
  \frac{\log s}{s}ze^{-z}\lesssim (e^{-z}+s^{-1})=h_0,
\end{equation}
where the constant is uniform, for all \(s\ge e\). Indeed, splitting the domain into $[0,\log s ]$ and $[\log s,s]$
\begin{equation*}
\frac{\log s}{s}ze^{-z}\leq
    \begin{cases}
          \frac{(\log s)^2}{s}e^{-z}
  \le C e^{-z} &\text{ if }0\le z\le \log s\\
  
\frac{(\log s)}{s} (\log s e^{-\log s})=\frac{(\log s)^2}{s^2}&\text{ if }\log s\le z\le s,
    \end{cases}
\end{equation*}
where we have used that for
\(\log s<z\le s\), the function \(ze^{-z}\) is decreasing on \([1,\infty)\), and
\(\log s\ge1\).
Finally, the pointwise
remainder bound \eqref{eq:F10-remainder-pointwise} gives
\begin{equation}
  |\tilde r(s,z)|
  \le C(s_1)\frac{(\log s)^3}{s^2}(1-e^{-z})
  \le C(s_1)s^{-1}
  \le C(s_1)\hh_0(s,z),
\end{equation}
since \((\log s)^3/s\) is bounded above by $\frac{1}{s}$ for \(s\ge s_1\).  Combining these bounds proves
\eqref{eq:F10-zero-tail}.

We now show the bound for \(F_{0,1}\). Recall that 
\(
  A(s):=\Lambda^2\frac{(\log s)^2}{s^2} \text{ and } \hh_+(s,z)=e^{-z}+A(s).\)
The leading term in \eqref{eq:F01-asymptotic-input} is controlled because
\eqref{eq:epsilon0-input} gives \(|1+\epsilon_0(s)|\le C\) and
\(e^{-z}\le\hh_+\). Since the two profiles satisfy
\begin{equation}
  |\Phi_{0,1}(z)|+|\Psi_{0,1}(z)|\le C(1+z)e^{-z},
\end{equation}
and $\frac{1}{s}\leq \frac{\log s}{s}$ for $s\geq e$, it remains to prove
\begin{equation}
\label{eq:poly-exp-hplus}
  \frac{\log s}{s}(1+z)e^{-z}
  \le C(s_1,\Lambda)\bigl(e^{-z}+A(s)\bigr).
\end{equation}
Again, we split the interval based on whether \(e^{-z}\ge A(s)\).
\begin{equation*}
     \frac{\log s}{s}(1+z)e^{-z}\leq\begin{cases}
         \frac{\log s}{s}(1+2\log s)e^{-z}\leq C e^{-z} &\text{ if }e^{-z}\geq A(s)\\
         \frac{\log s}{s}(1+\log(A(s)^{-1}))A(s)\leq  C\frac{\log s}{s}  \log (s)A(s) &\text{ if }e^{-z}<A(s)
     \end{cases}
\end{equation*}
where we have used that $e^{-z}\geq A(s) \implies z\leq 2 \log s$ in the first region, \(z\mapsto(1+z)e^{-z}\) is decreasing on
\([0,\infty)\) in the second region, and $\Lambda \geq 1,s\geq e$ in both regions.
Finally, the desired result is obtained by using the remainder estimate \eqref{eq:F01-remainder-input} and $\Lambda\geq 1$.
\end{proof}
\subsubsection{Properties of the weights}
The nonlinear source term includes the interactions between two modes in the form of products, primitives, and averages. The following lemma shows how these interactions are controlled under those two weights.
\begin{corollary}[Primitive, product, and average estimates]\label{lem:mixed-weight-estimates}
For {every \(s\ge1\) } and \(0\le z\le s\), the following estimates hold.
\begin{enumerate}[label=\textup{(\roman*)}, leftmargin=2.4em]
\item \emph{Zero-mode primitives.} If \(\norm{\wh_0 f}_{L^\infty}\le A\), then
\begin{equation}\label{eq:pzi-zero-bound}
  |\pzi f(z)|\le2A.
\end{equation}

\item \emph{Ordinary products.}  {Let \(p,q,k\in\mathbb N_0\) with \(k=p+q\).
} If \(|f|\le A\hh_p\) {and } \(|g|\le B\hh_q\) then
\begin{equation}\label{eq:mixed-product-basic}
  \wh_k|fg|
  \le 2 AB
  \begin{cases}
  1, & k=0,\\
  \hh_+, & p\ge1,\ q\ge1,\\
  \hh_0, & k\ge1\text{ and exactly one of }p,q\text{ is }0.
  \end{cases}
\end{equation}

\item \emph{Positive-primitive product.} If \(|f|\le A\hh_0\)
and \(|g|\le B\hh_+\), then
\begin{equation}\label{eq:pzi-zero-positive-product}
  {\wh_+|(\pzi f)\,g|}\le 2 AB.
\end{equation}
\item \emph{Zero-mode product average.} If \(|f|\le A\hh_0\) and \(|g|\le B\hh_0\), then
\begin{equation}\label{eq:zero-avg-product}
  \wh_0 s^{-1}\int_0^s |fg|\dd z\le \frac{7}{2} AB.
\end{equation}
\end{enumerate}
\end{corollary}

\begin{proof}
This proof follows by definitions of weight functions \eqref{eq:mixed-h}.
\end{proof}

\subsection{Uniform lower bound on the effective damping coefficient}

For all active modes ($i+k\geq 2$) we define the {reference damping rate} by
\begin{equation}\label{eq:Dik}
  \mathcal D_{i,k}(s,z)=
  \begin{cases}
  1+i, &  i\ge2,\,k=0,\\[0.35em]
  (1+i)+(2i+k)\bigl(e^{-z}+\Lambda^2(\log s)^2/s^2\bigr), & i\ge1,\ k\ge1,\\[0.35em]
  k\bigl(e^{-z}+\Lambda^2(\log s)^2/s^2\bigr)
  +\dfrac{\Lambda^2(\log s)^2/s^2}{e^{-z}+\Lambda^2(\log s)^2/s^2}, & i=0,\ k\ge2.
  \end{cases}
\end{equation}
Recall that $G_{i,k}$ is the coefficient
\begin{equation}
\begin{split}
        G_{i,k}:={}&
\frac{k}{s}+3i+1-(2i+2-k)(\phi+\eps)
-(2i+k-1)\frac{\rho'}{\rho}
+\frac{\partial_s\hh_k}{\hh_k} \\
&-\left(\pzi(\phi+\eps)+\frac{z}s\right)\frac{e^{-z}}{\hh_k}.
\end{split}
\end{equation}
of the zeroth order term in the normalized PDE \eqref{eq:normalized-pde}. We now divide into two cases, depending on either $k\geq 1$ or $k=0$. We shall show $G_{i,k}$ can be bounded below by this scale, up to a universal constant, taking into consideration of the influence of non-local terms in the linear part, if they do exist. These lower bounds are later used to establish comparison principles.
\subsubsection{Positive modes}
We restrict to the case where the modes are positive and active ($i+k\geq 2, k\geq 1$). Notice that positive modes do not contain nonlocal terms in the linear part of the PDE. The following lemma shows that, after normalization and for sufficiently large time, their linear part has enough damping, so that later the non-linear source can be absorbed, which establishes a comparison principle.

\begin{lemma}[Uniform positive-mode gap]\label{lem:positive-gap}
 There exist universal constants
\[
  \Lambda_+\ge1,\qquad C_{0,+}>0,\qquad c_+>0,
\]
such that for every
\(\Lambda\ge\Lambda_+\), every \(C_0\ge C_{0,+}\), the following holds. {There exists $s_+(\Lambda)$}, such that for every active positive mode
\(i+k\ge2\), \(k\ge1\), every \(s\ge s_+\), and every \(0\le z\le s\), 
\begin{equation}\label{eq:positive-gap}
  G_{i,k}(s,z)\ge c_+\mathcal D_{i,k}(s,z).
\end{equation}
\end{lemma}
\begin{proof}
{In this proof, \(S\) denotes a temporary lower-time threshold. At the end, $s_+\geq S$ is chosen after we obtain all the restriction on \(S\).}
As usual, we recall:
\[
  A(s):=\Lambda^2\frac{(\log s)^2}{s^2},
  \qquad \hh_+(s,z)=e^{-z}+A(s).
\]
For positive modes, \(k\ge1\), we have \(\hh_k=\hh_+\). Combining the definition of $G_{i,k}$ \eqref{eq:Geff-def}, the definition of $\eps_{\rm rem}$ \eqref{eq:epsrem-def}, and the definition of $\rho$ \eqref{eq:rho}, we decompose $G_{i,k}$ as
\begin{equation}\label{eq:positive-G-decomp}
  G_{i,k}=d_{i,k}^+ + R_{i,k}+E_{i,k},
\end{equation}
\begin{equation}\label{eq:model-positive}
\begin{aligned}
  d_{i,k}^+
  :={}&3i+1+(2i+2)s^{-1}-(2i+2-k)e^{-z}
  -\frac{e^{-z}(1-e^{-z})}{\hh_+}
  +\frac{\partial_s\hh_+}{\hh_+},
\end{aligned}
\end{equation}
\begin{equation}\label{eq:positive-radius-piece}
  R_{i,k}:=(2i+k-1)C_0A(s),
\end{equation}
\begin{equation}\label{eq:positive-error-piece}
  E_{i,k}:=-(2i+2-k)\eps_{\rm rem}
  -\pzi\eps_{\rm rem}\frac{e^{-z}}{\hh_+}.
\end{equation}
By direct computation and use of the definition of $A$ and $h_+$ we have
\[ \left|\frac{\partial_s\hh_+}{\hh_+}\right|
  =\left|\frac{A'(s)}{\hh_+}\right| \le \left|\frac{C\frac{A(s)}{s}}{A(s)}\right| 
  = Cs^{-1} \]
uniformly for \(s\ge S\).  This is the only loss inside the model coefficient
\(d_{i,k}^+\).
Let us first assume \(i\ge1\) and \(k\ge1\).  Using $h_+\geq e^{-z}$, \eqref{eq:model-positive}
gives
\begin{equation}\label{eq:i-positive-model}
\begin{aligned}
  d_{i,k}^+
  &\ge 3i+1-(2i+2-k)e^{-z}-(1-e^{-z})-o_S(1) \\
  &=3i(1-e^{-z})+(i-1+k)e^{-z}-o_S(1).
\end{aligned}
\end{equation}
We now compare the last line with the corresponding scale in \eqref{eq:Dik}. Using the fact that the function $\frac{3i(1-x)+(i-1+k)x}{(1+i)+(2i+k)x}$ is fractionally linear in $x$, and thus the minimum on $x\in[0,1]$ is attained at the end point, we conclude that based on $i,k\geq1$
{
\begin{equation}\label{eq:positive-model-controls-local-scale}
\begin{split}
       \frac{3i(1-e^{-z})+(i-1+k)e^{-z}}{\bigl((1+i)+(2i+k)e^{-z}\bigr)}&\geq \min(\frac{3i}{1+i},\frac{i+k-1}{3i+k+1})\geq \frac 15
\end{split}
\end{equation}
}
Therefore, for $S$ large enough,
\begin{equation}\label{eq:i-positive-local-gap}
    \begin{split}
        d_{i,k}^+
  \ge \frac{1}{5}\bigl((1+i)+(2i+k)e^{-z}\bigr)-\frac{1}{10}\cdot 2
  \ge\frac{1}{10}\bigl((1+i)+(2i+k)e^{-z}\bigr)\qquad i\ge1,\ k\ge1
    \end{split}
\end{equation}
For $i,k\geq 1$, using the definition of $E_{i,k}$ \eqref{eq:positive-error-piece}, the bound on $\eps_{\rm rem}$ \eqref{eq:epsrem-pointwise-split} and the primitive bound \eqref{eq:epsrem-localized-small}
\begin{equation}\label{eq:positive-error-lower}
\begin{split}
      E_{i,k}&\ge -\abs{(2i+2-k)\eps_{\rm rem}}
  -\abs{\pzi\eps_{\rm rem}}\frac{e^{-z}}{\hh_+}\\
  &\ge
  -\abs{o_S(1)}(2i+k+1)e^{-z}
  -C(2i+k+1)A(s)
  -\frac{C}{\Lambda}\left(e^{-z}+\frac{A(s)}{\hh_+}\right)\\
  &\ge -C\abs{o_S(1)}(2i+k)e^{-z}-\frac{C}{\Lambda}\left((2i+k)e^{-z}+1\right)-C(2i+k+1)A(s)\\
         &\geq -C\left(\abs{o_S(1)}+\frac{1}{\Lambda}\right)(1+i+(2i+k)e^{-z})-C(2i+k+1)A(s)\\
         &\geq -\frac{1}{20}(1+i+(2i+k)e^{-z})-C(2i+k+1)A(s)\qquad i\ge1,\ k\ge1
\end{split}
\end{equation}
for $\Lambda^+$ and then $S$ large enough. Finally, by definition \eqref{eq:positive-radius-piece}, for $i,k\geq 1$, the last term in \eqref{eq:positive-error-lower} is absorbed into $R_{i,k}$:
\begin{equation}\label{eq:R-i,k}
\begin{split}
        R_{i,k}-C(2i+k+1)A(s)&\geq (2i+k-1)C_0A(s)- C(2i+k+1)A(s)\\
        &=(\frac{C_0}{3}-C)(2i+k-1)A(s)\geq \frac{1}{20}(2i+k)A(s)\qquad i\ge1,\ k\ge1
\end{split}
\end{equation}
where we have chosen $C_0$ large enough. Consequently, combining \eqref{eq:i-positive-local-gap}, \eqref{eq:positive-error-lower} and \eqref{eq:R-i,k}, we have
\[
  G_{i,k}
  \ge c\bigl((1+i)+(2i+k)e^{-z}+(2i+k)A(s)\bigr)
  =c\mathcal D_{i,k},
  \qquad i\ge1,\ k\ge1.
\]
It remains to {prove the case for} \(i=0\), {in which case}
\(k\ge2\), as $i+k\geq 2$.
The borderline case is \(k=2\).  In that case, by definition of $h_+$ and \(|A'(s)|\le C A(s)/s\)
\begin{equation}\label{eq:i0-k2}
\begin{aligned}
  d_{0,2}^+
  &=\frac{e^{-2z}+A(s)}{\hh_+}
    +2s^{-1}+\frac{\partial_s\hh_+}{\hh_+}\geq \frac{1}{2} \frac{\bigl(e^{-z}+A(s)\bigr)}{h_+}e^{-z}+\frac{1}{2} \frac{A(s)}{h_+}+\frac{A'(s)}{h_+}\geq \frac{1}{2}e^{-z}+(\frac{1}{2}-\frac{C}{s})\frac{A(s)}{h_+}
\end{aligned}
\end{equation}
After increasing \(S\) so that \(C/s\le1/4\) for all \(s\ge S\), this yields
\begin{equation}\label{eq:i0-k2-lower}
  d_{0,2}^+
  \ge
  \frac{1}{2} e^{-z}+\frac{1}{4}\frac{A(s)}{\hh_+}
  = \frac{1}{4}\left(2e^{-z}+\frac{A(s)}{\hh_+}\right).
\end{equation}
Next, for arbitrary \(k\ge2\), observe that
\begin{equation}\label{eq:i0-model-control}
  d_{0,k}^+=d_{0,2}^+ +(k-2)e^{-z}
  \ge \frac{1}{4}\left(k e^{-z}+\frac{A(s)}{\hh_+}\right),
  \qquad i=0,\ k\ge2.
\end{equation}
For \(i=0\), the error estimate \eqref{eq:positive-error-lower} is still valid up until the second line,
\[
  E_{0,k}
  \ge
  -\abs{o_S(1)}(k+1)e^{-z}
  -C(k+1)A(s)
  -\frac{C}{\Lambda}\left(e^{-z}+\frac{A(s)}{\hh_+}\right).
\]
The terms
\(
  \abs{o_S(1)}(k+1)e^{-z}
  \quad\hbox{and}\quad
  \frac{C}{\Lambda}\left(e^{-z}+\frac{A(s)}{\hh_+}\right)
\)
are absorbed by \eqref{eq:i0-model-control}, after enlarging \(\Lambda_+\)
and increasing \(S\). Indeed, for \(k\ge2\),
\((k+1)e^{-z}\le Ck e^{-z}\), while
\(e^{-z}+A(s)/\hh_+\le k e^{-z}+A(s)/\hh_+\).  Thus, after these choices are made, the
only remaining negative term from \(E_{0,k}\) is the loss
\(C(k+1)A(s)\), which is again controlled through $ R_{0,k}$. Indeed,
since \(k+1\le3(k-1)\) for \(k\ge2\), choosing \(C_0\) large enough gives
\[
  R_{0,k}-C(k+1)A(s)=((k-1)C_0-C(k+1))A(s)\gtrsim  k A(s).
\]
Therefore, by definition of $G_{0,k}$,
\begin{equation}\label{eq:i0-D-lower}
\begin{aligned}
  G_{0,k}
  &\ge c\left(k e^{-z}+\frac{A(s)}{\hh_+}\right)+c kA(s)
   =c\left(k\hh_+ +\frac{A(s)}{\hh_+}\right)
   =c\mathcal D_{0,k}\qquad k\geq 2,
\end{aligned}
\end{equation}
where the terms \(A(s)/\hh_+\) and $ke^{-z}$ above comes from the estimate
\eqref{eq:i0-model-control}, and the term \(kA(s)\) comes from the previous equation on $R_{0,k}$.
\end{proof}

\subsubsection{Zero modes}
Here, we deal with $i\geq 2, k=0$. Zero modes contain a nonlocal term that can potentially weaken the local damping.  This lemma shows that the local damping still dominates the largest possible contribution of the nonlocal term by a uniform constant.
\begin{lemma}[Effective zero-mode gap]\label{lem:zero-mode-gap}
{There exist {a universal} $c_{\rm zm}>0$ and {$s_{\rm zm}(\Lambda)$}, such that,
for every active zero mode $(i,0)$ with $i\ge2$, every $s\ge s_{\rm zm}$, and every
$0\le z\le s$,}
\begin{equation}\label{eq:zero-gap}
\begin{aligned}
\Gamma_i(s,z):={}&G_{i,0}(s,z)
-(2i+1)\wh_0(s,z)|\phi-\eps_z|(s,z)
       \int_0^z\hh_0(s,y)\dd y \\
&-4(i+1)\wh_0(s,z)s^{-1}
       \int_0^s |\phi+\eps|(s,y)\hh_0(s,y)\dd y \\
&\ge c_{\rm zm}(1+i)=c_{\rm zm} \mathcal D_{i,0}.
\end{aligned}
\end{equation}
\end{lemma}
\begin{proof}
{As usual, let $S$ denote a temporary lower-time threshold, by the end of the proof, we choose
$s_{\rm zm}$ no smaller than the final required lower bound on $S$.}
Set
\[
  \theta(s,z):=\frac{e^{-z}}{\hh_0(s,z)}
  =\frac{e^{-z}}{e^{-z}+s^{-1}},
\]
which implies $1-\theta=s^{-1}/\hh_0=s^{-1}w_0$. Using the formula of coefficient
\eqref{eq:Geff-def}, we decompose
\begin{equation}\label{eq:Gi0-expanded}
\begin{aligned}
G_{i,0}={}&3i+1-(2i+2)e^{-z}+(2i+2)s^{-1}
-(2i+2)\eps_{\rm rem} \\
&+(2i-1)C_0A(s)-\frac{s^{-2}}{\hh_0}
-\theta(1-e^{-z})-\pzi\eps_{\rm rem}\frac{e^{-z}}{\hh_0}.
\end{aligned}
\end{equation}
The two terms
$(2i+2)s^{-1}$ and $(2i-1)C_0A(s)$ are nonnegative, so we may drop them for a
lower bound.
The remaining errors are controlled by Lemma~\ref{lem:epsrem-estimates}:
from \eqref{eq:epsrem-pointwise-split},
\[
  |(2i+2)\eps_{\rm rem}|
  \le (2i+2)\bigl(|o_S(1)|e^{-z}+C(\log s)^2/s^2\bigr)
  \le |o_S(1)|(1+i),
\]
after increasing $S$, while
\eqref{eq:epsrem-primitive-small} gives
\[
  |\pzi\eps_{\rm rem}|\frac{e^{-z}}{\hh_0}
  \le |\pzi\eps_{\rm rem}|
  \le |o_S(1)|
  \le |o_S(1)|(1+i),
\]
and by definition $ \frac{s^{-2}}{\hh_0}\le \frac{s^{-2}}{s^{-1}}=s^{-1}
  \le |o_S(1)|(1+i)$.
Consequently,
\begin{equation}\label{eq:zero-local-model-lower}
  G_{i,0}
  \ge
  3i+1-(2i+2)e^{-z}-\theta(1-e^{-z})
  -|o_S(1)|(1+i).
\end{equation}
Next, by definition of $\theta$, $\phi=e^{-z}$ and \eqref{eq:Ez-weighted-bound} 
\begin{equation}\label{eq:zero-first-volterra-model}
\begin{aligned}
\wh_0|\phi-\eps_z|\int_0^z\hh_0
&\le \left(\theta+\frac{|\eps_z|}{\hh_0}\right)(1-e^{-z}+zs^{-1}) \\
&\le \bigl(\theta+|o_S(1)|\bigr)(1-e^{-z}+zs^{-1}) \leq \theta(1-e^{-z})+|o_S(1)|
\end{aligned}
\end{equation}
The last inequality holds since $\frac{e^{-z}}{e^{-z}+s^{-1}}z s^{-1} \lesssim \frac{\log s}{s}$, which can be proved by splitting into $z\le2\log s$ and
$z>2\log s$: in the first region it is at most $2\log s/s$, and in the
second region $e^{-z}\le s^{-2}$ gives
$ze^{-z}/(s(e^{-z}+s^{-1}))\le z/s^2\le s^{-1}$.  Multiplying
\eqref{eq:zero-first-volterra-model} by the coefficient $2i+1$ appearing in
\(\Gamma_i\) gives:
\begin{equation}\label{eq:zero-first-volterra-term}
  (2i+1)\wh_0|\phi-\eps_z|\int_0^z\hh_0
  \le (2i+1)\theta(1-e^{-z})+|o_S(1)|(1+i).
\end{equation}
Finally, by \eqref{eq:epsrem-pointwise-split}, $z\leq s$ and the definition of $h_0$,
\begin{equation}\label{eq:zero-average-model}
\begin{aligned}
\int_0^s |\phi+\eps|\hh_0\dd y
&\le \int_0^s |e^{-y}-s^{-1}|(e^{-y}+s^{-1})\dd y
   +\int_0^s |\eps_{\rm rem}|\hh_0\dd y \\
 &\leq  \int_0^s |e^{-2y}-s^{-2}|\dd y
   + |o_S(1)|\int_0^s e^{-y}\hh_0\dd y
 +C\frac{(\log s)^2}{s^2}\int_0^s\hh_0\dd y\\
 &\le \frac{1}{2}+|o_S(1)|.
\end{aligned}
\end{equation}
 Therefore, by the relation between $\theta$ and $w_0$,
\begin{equation}\label{eq:zero-average-term}
  4(i+1)\wh_0s^{-1}\int_0^s |\phi+\eps|\hh_0
  \le \bigl(2+|o_S(1)|\bigr)(i+1)(1-\theta).
\end{equation}
Combining \eqref{eq:zero-local-model-lower},
\eqref{eq:zero-first-volterra-term},
\eqref{eq:zero-average-term}, $0<\theta<1$ and $e^{-z}(1-\theta)\leq s^{-1}$ we have
\begin{align*}
\Gamma_i
&\ge
3i+1-(2i+2)e^{-z}-(2i+2)\theta(1-e^{-z})-2(i+1)(1-\theta)-|o_S(1)|(1+i) \\
&= i-1-(2i+2)e^{-z}(1-\theta)-|o_S(1)|(1+i)\\
&\geq \frac{1}{3}(i+1)-2(i+1)s^{-1}-|o_S(1)|(1+i)=(\frac{1}{3}-|o_S(1)|)(i+1)\geq \frac{1}{6}(1+i),
\end{align*}
for $S$ large enough.
\end{proof}

\subsection{The comparison principles}
{We now prove a modified comparison principle, in an abstract form. This comparison principle incorporates the effect of the non-local operator (if it exists), which may potentially diminish the damping. It indicates that the solution remains bounded in the whole domain, whenever it is bounded initially and the source is small, relative to some scale, under the assumption that the linear coefficient of the PDE dominates the potential influence of the nonlocal operator and that {the boundary is a characteristic trajectory of the transport operator}. This abstract result is then readily applied to obtain the maximum-principle result for the equation of $U_{i,k}$.}
{
\begin{lemma}[Modified comparison principle]\label{lem:moving-strip-comparison}
Let \(s_{\mathrm{in}}>0\), \(R>0\), \(\mu>0\) be positive constants and let \(b,G,S\),
and \(q>0,\,P\ge0\) be continuous on the moving strip \( \{(s,z):s\ge s_{\mathrm{in}},\ 0\le z\le s\}\), consider the PDE in the form of:
\begin{equation}\label{eq:abstract-comparison-equation}
  \partial_sV+b\partial_zV+GV-\mathcal L_s[V]=S
\end{equation}
suppose that
\begin{enumerate}
    \item The transport field $b$ satisfies $ b(s,0)=0,\,b(s,s)=1$.
    \item The linear operator satisfies the point-wise bound 
    \begin{equation}\label{eq:abstract-operator-bound}
  \abs{\mathcal L_s[f](z)}
  \le \norm{f}_{L^\infty(0,s)}P(s,z)
\end{equation}
\item The difference of zeroth order coefficient and the upper bound of the norm of the linear operator, $P$, is bounded below,
\begin{equation}\label{eq:abstract-gap}
  G(s,z)-P(s,z)\ge \mu q(s,z)
\end{equation}
Then the following implication holds on the moving strip:
If initially, the solution $V$ is bounded
\begin{equation}\label{eq:abstract-init-bounded}
     |V(s_{\mathrm{in}},z)|\le R,\qquad 0\le z\le s_{\mathrm{in}},
\end{equation}
and the source is controlled
\begin{equation}\label{eq:abstract-source-small}
  |S(s,z)|\le \frac{\mu}{2}R q(s,z),\qquad s\ge s_{\mathrm{in}},\quad 0\le z\le s,
\end{equation}
then the solution can be bounded via
\[
  |V(s,z)|\le R,\qquad s\ge s_{\mathrm{in}},\quad 0\le z\le s.
\]
\end{enumerate}
\end{lemma}
\begin{proof}
Let \(\eta>0\), defining the super solution to be
\[
  B(s)=R+\eta(1+s-s   {_{\mathrm{in}}} ),
\]
and suppose  that $\abs{V}\leq R$ does not hold.
Introduce the two signed differences
\[
  W^\pm(s,z):=\pm V(s,z)-B(s).
\]
At the first contact, one of \(W^+\) or \(W^-\) vanishes at \((s_*,z_*)\), we thus set $W^\pm$ to be the branch that such contact occurs. For earlier times,
\[
  W^\pm(s,z)\le |V(s,z)|-B(s)<0\quad(s<s_*),
\]
and \(W^\pm(s_*,z)\le0\) for \(0\le z\le s_*\). Notice that necessarily $s_*>s_\text{in}$, by the condition on $s_\text{in}$ \eqref{eq:abstract-init-bounded}.
We claim that
\begin{equation}\label{eq:zero-boundary-contact-transport-derivative}
  (\partial_s+b\partial_z)W^\pm(s_*,z_*)\ge0.
\end{equation}
There are three cases.
If \(0<z_*<s_*\), then \(z\mapsto W^\pm(s_*,z)\) has a spatial maximum at \(z_*\),
so \(\partial_zW^\pm(s_*,z_*)=0\).  Also, for \(s<s_*\) close to \(s_*\), the fixed
point \(z_*\) still lies inside the slice \(0\le z\le s\), and
\(W^\pm(s,z_*)<0=W^\pm(s_*,z_*)\).  Therefore, by definition,
\( \partial_sW^\pm(s_*,z_*)\ge0,\) which proves \eqref{eq:zero-boundary-contact-transport-derivative} in the
interior case.
If \(z_*=0\), we restrict \(W^\pm\) to the lower boundary and set
\(w_0(s):=W^\pm(s,0)\).  For every    {\(s_{\mathrm{in}}\le s<s_*\)} , the first-contact definition gives
\(w_0(s)<0\), while \(w_0(s_*)=0\).  Hence, in this case
\( w_0'(s_*)\ge0.\)
Since \(b(s,0)=0\), we have
\( w_0'(s_*)=\partial_sW^\pm(s_*,0)=(\partial_s+b\partial_z)W^\pm(s_*,0)\geq 0.\)
If \(z_*=s_*\), we restrict \(W^\pm\) to the moving upper boundary and set
\(w_1(s):=W^\pm(s,s)\).  For every    {\(s_{\mathrm{in}}\le s<s_*\)}, \(w_1(s)<0\), while \(w_1(s_*)=0\).  Thus
\(w_1'(s_*)\ge0.\)
By chain rule and \(b(s,s)=1\),
\(  w_1'(s_*)=\partial_sW^\pm(s_*,s_*)+\partial_z W^\pm(s_*,s_*)
  =(\partial_s+b\partial_z)W^\pm(s_*,s_*)\geq 0.\)
On the other hand, using the definition
\eqref{eq:abstract-comparison-equation}, \(W^\pm(s_*,z_*)=0\), \(W^\pm=\pm V-B\) and \(B'(s)=\eta\), we see at \((s_*,z_*)\)
\begin{align*}
(\partial_s+b\partial_z)W^\pm(s_*,z_*)
&=\pm S(s_*,z_*)-G(s_*,z_*)(\pm V(s_*,z_*))+
  \mathcal L_{s_*}[\pm V(s_*,\cdot)](z_*)-\eta \\
&=\pm S(s_*,z_*)-G(s_*,z_*)B(s_*)
  \pm\mathcal L_{s_*}[V(s_*,\cdot)](z_*)-\eta \\
\end{align*}
Next, combining with the condition on the source \eqref{eq:abstract-source-small}, the operator bound \eqref{eq:abstract-operator-bound}, \(|\pm V(s_*,y)|\le B(s_*)\), $B\geq R$, and the gap on $G-P$ \eqref{eq:abstract-gap}, we have
\begin{equation*}
    \begin{split}
        (\partial_s+b\partial_z)W^\pm(s_*,z_*)&\leq \frac{\mu}{2}R q(s_*,z_*)-G(s_*,z_*)B(s_*)+ B(s_*)P(s_*,z_*)-\eta\\
        &\leq \frac{\mu}{2}R q(s_*,z_*)-\mu q(s_*,z_*) B(s_*)-\eta\\
        &\leq  (-\mu R+\frac{\mu}{2}R) q(s_*,z_*)-\eta\leq -\eta,
    \end{split}
\end{equation*}
which contradicts \eqref{eq:zero-boundary-contact-transport-derivative}.  Hence
\(|V(s,z)|\le B(s)\) for all    {\(s\ge s_{\mathrm{in}}\)} , \(0\le z\le s\). Finally, for each fixed
\(s\), letting \(\eta\downarrow0\) gives \(|V(s,z)|\le R\).
\end{proof}
}

{The following lemma identifies the constants, functions and linear operator in the abstract comparison principle, and turns it into a maximum principle ready to be applied for the normalized variable $U_{i,k}$.}
\begin{lemma}[Comparison principle for all active modes]\label{lem:linear-comparison}
Let \(\Lambda_+\) and \(C_{0,+}\) be the universal
constants supplied by Lemma~\ref{lem:positive-gap}.  There exist universal
constant
\(\mu_0>0,\)
such that, for every
\[
  \Lambda\ge\Lambda_+,
  \qquad
  C_0\ge C_{0,+},
\]
there exists {$
  s_{\mathrm{lin}}(\Lambda)$} and for $s_0\geq s_{\mathrm{lin}}$, the following holds for every active pair \(i+k\ge2\). Let \(U_{i,k}\)
solve, on the moving strip
\(s\ge s_0\), \(0\le z\le s\), the normalized linear equation
\[
\begin{aligned}
\partial_sU_{i,k}
&+\left(\pzi(\phi+\eps)+\frac{z}s\right)\partial_zU_{i,k}
+G_{i,k}U_{i,k} =\one_{k=0}\mathcal K_i[U_{i,0}]+S_{i,k}.
\end{aligned}
\]
Assume that
\[
  |U_{i,k}(s_0,z)|\le R,
  \qquad 0\le z\le s_0.
\]
If, for every \(s\ge s_0\), \(0\le z\le s\),
\begin{equation}\label{eq:linear-comparison-source}
  |S_{i,k}(s,z)|\le \frac{\mu_0}{2}R\mathcal D_{i,k}(s,z),
\end{equation}
then
\begin{equation}\label{eq:linear-comparison-conclusion}
  |U_{i,k}(s,z)|\le R
  \qquad (s\ge s_0,\ 0\le z\le s).
\end{equation}
\end{lemma}
\begin{proof}
{Let us consider the case $k=0$, the transport coefficient equality is given by \eqref{eq:transport-tangent-boundary}. In light of Lemma \ref{lem:moving-strip-comparison}, we identify
\(c_0:=c_{\mathrm{zm}},\,
  s_{\mathrm{cmp}}:=s_{\mathrm{zm}},\)
where \(c_{\mathrm{zm}}\) and \(s_{\mathrm{zm}}\) are supplied by
Lemma~\ref{lem:zero-mode-gap}. We identify $q(s,z):=1+i,\,\mathcal L_s=\mathcal K_i, S=S_{i,0},\mu=\mu_0$ and \[
\begin{aligned}
P_i(s,z):={}&
(2i+1)\wh_0(s,z)|\phi-\eps_z|(s,z)
       \int_0^z\hh_0(s,y)\dd y \\
&+4(i+1)\wh_0(s,z)s^{-1}
       \int_0^s |\phi+\eps|(s,y)\hh_0(s,y)\dd y .
\end{aligned}
\] By definition \eqref{eq:normalized-zero-operator}, this immediately gives
\[
  |\mathcal K_i[f](s,z)|
  \le \norm{f}_{L^\infty(0,s)}P_i(s,z).
\]
The gap condition
\eqref{eq:abstract-gap} is precisely Lemma~\ref{lem:zero-mode-gap}, namely
\[
  G_{i,0}(s,z)-P_i(s,z)\ge c_{\mathrm{zm}}(1+i),
  \qquad s\ge s_{\mathrm{cmp}}.
\]
In light of Lemma \ref{lem:positive-gap}, the case of $k\geq 1$ is similar. In particular, the nonlocal operator is simply zero.
}
\end{proof}

\subsection{Non-linear source estimate}
{Recall that the comparison principles established earlier require the nonlinear source to be small relative to the {reference damping rate}. We shall first prove combinatorial estimates, corresponding to different types of inter-modal interactions, which are contained inside the nonlinear source. Once they are established, we shall how that the nonlinear source can indeed be controlled, whenever we assume that all the active modes inside the source satisfy the target bound, as well as the background information on shell-one modes.}
\subsubsection{Combinatorial properties}\label{subsubsec:combinatorial}
\begin{lemma}[Convolutional estimates]\label{lem:endpoint-convolution}
Let \(b\ge3\).  There is a uniform constant \(C\) such that, for every
\(m\ge2\),
\begin{equation}\label{eq:endpoint-conv-one}
(1+   {m} )^b\sum_{r=1}^{m-1}
\frac{1}{(1+r)^b(1+m-r)^b}
\le C\left(\frac{3}{4}\right)^b,
\end{equation}
and consequently
\begin{equation}\label{eq:endpoint-conv-count}
(1+   {m} )^b\sum_{r=1}^{m-1}
\frac{1+r}{(1+r)^b(1+m-r)^b}
\le C(1+   {m} )\left(\frac{3}{4}\right)^b.
\end{equation}
\end{lemma}
\begin{proof}
To prove \eqref{eq:endpoint-conv-one}, we write
\[
  T_m(r):=(1+   {m} )^b\frac{1}{(1+r)^b(1+m-r)^b},
  \,\qquad 1\le r\le m-1.
\]
The summand is symmetric under \(r\mapsto m-r\), since
\(T_m(m-r)=T_m(r)\).  Hence, for every \(m\ge2\),
\begin{equation}\label{eq:endpoint-left-half-proof}
  \sum_{r=1}^{m-1}T_m(r)
  \le 2\sum_{1\le r\le \lfloor m/2\rfloor}T_m(r).
\end{equation}
The endpoint of the summation on the left hand side is \(r=1\) and
\( T_m(1)=\left(\frac{1+m}{2m}\right)^b
  \le \left(\frac{3}{4}\right)^b,
 \text{ for }m\ge2.\)
It remains to estimate the sum in with index
\(2\le r\le \lfloor m/2\rfloor\).  This range may be empty, for instance when
\(m=2\) or \(m=3\).  If it is nonempty, then in particular, \(r\le m/2\), so
   {\(1+m-r\ge (1+m)/2\).  Therefore
} \[
  T_m(r)
  =\left(\frac{1+m}{(1+r)(1+m-r)}\right)^b
  \le \left(\frac{2}{1+r}\right)^b.
\]
Consequently,
\[
\begin{aligned}
  \sum_{2\le r\le \lfloor m/2\rfloor}T_m(r)
  \le \sum_{r=2}^{\infty}\left(\frac{2}{1+r}\right)^b  & = \left(\frac{2}{3}\right)^b
     \left(1+\sum_{n=4}^{\infty}\left(\frac{3}n\right)^b\right)\\
     &\leq \left(\frac{2}{3}\right)^b
     \left(1+\sum_{n=4}^{\infty}\left(\frac{3}n\right)^3\right)\lesssim \left(\frac{2}{3}\right)^b
   \leq \left(\frac{3}{4}\right)^b,
\end{aligned}
\]
where, in the third line, we used \(b\ge3\).  Combining
this with \eqref{eq:endpoint-left-half-proof} proves
\eqref{eq:endpoint-conv-one}. Finally, to obtain \eqref{eq:endpoint-conv-count}, we simply use \(1+r\le1+m\) inside the sum.
\end{proof}

\begin{lemma}[Factorial and shell-sum estimates]\label{lem:active-active-preliminaries}
Assume \(a\ge3\) and define \(\Theta_a=\left(\frac{3}{4}\right)^a\).
Then the following estimates hold uniformly in all admissible indices.\par\smallskip
\noindent\emph{Factorial-binomial estimates.}
If {\(A_1,C_1,P,K\in\mathbb N\)}, then
\begin{equation}\label{eq:factorial-block}
\binom{P+K}{P}\binom{A_1+C_1}{A_1}
\frac{(A_1+P)!(C_1+K+1)!}{(A_1+C_1+P+K)!}
\le C_1+K+1.
\end{equation}
In particular, if \(0\le p\le k-1\), \(\ell=k-p\), and \(r=i-q\), then
\begin{equation}\label{eq:shifted-block}
\binom{k}{p+1}\binom{2i+1}{2q}
\frac{(2q+p+1)!(2r+\ell+1)!}{(2i+k+1)!}
\le 2(1+r+\ell).
\end{equation}
\smallskip
For \(q\ge2\), \(r=i-q\), and \(k\ge1\),
\begin{equation}\label{eq:zero-positive-shifted-factor}
 k\binom{2i+1}{2q}
\frac{(2q+1)!(2r+k+1)!}{(2i+k+1)!}
\lesssim (1+i).
\end{equation}
{\noindent\emph{Endpoint shell sums.}  For every integer \(m\ge2\),
\begin{equation}\label{eq:shell-sum-small}
\begin{gathered}
\sum_{\substack{m_1,m_2\ge2\\ m_1+m_2=m}}
\frac{(1+m)^{a+1}}{(1+m_1)^{a+1}(1+m_2)^{a+1}}
\le C\Theta_a, \\
\sum_{\substack{m_1,m_2\ge2\\ m_1+m_2=m+1}}
\frac{(1+m)^{a+1}}{(1+m_1)^{a+1}(1+m_2)^{a+1}}
\le C\Theta_a.
\end{gathered}
\end{equation}
}
Let \(k\ge2\), then
\begin{equation}\label{eq:shell-sum-counted}
\sum_{\substack{0\le q\le i,\ 1\le p\le k-1\\ q+p\ge2,\ i-q+k-p\ge2}}
\frac{(1+i+k)^{a+1}(1+i-q+k-p)}{(1+q+p)^{a+1}(1+i-q+k-p)^{a+1}}
\le C\Theta_a(1+i+k).
\end{equation}
All the constants $C$ in
these estimates are independent of \((i,k)\) and of \(a\).
\end{lemma}

\begin{proof}
For \eqref{eq:factorial-block}, we use the following consequence
from Vandermonde's formula.  With
\(n_1=A_1+C_1\), \(j_1=A_1\), \(n_2=P+K\), and \(j_2=P\), Vandermonde equality
\[
\binom{n_1+n_2}{j_1+j_2}
=\sum_j \binom{n_1}{j}\binom{n_2}{j_1+j_2-j}.
\]
with \(j=j_1\) yields
\[
\frac{(A_1+C_1+P+K)!}{(A_1+P)!(C_1+K)!}=\binom{A_1+C_1+P+K}{A_1+P}
\ge \binom{A_1+C_1}{A_1}\binom{P+K}{P}.
\]
Multiplying by
\((A_1+P)!(C_1+K+1)!/(A_1+C_1+P+K)!\) gives
\eqref{eq:factorial-block}.  The shifted factorial-binomial estimate
\eqref{eq:shifted-block} is \eqref{eq:factorial-block} under
\((A_1,P,C_1,K)=(2q,p+1,2r+1,\ell-1)\) with \(\ell=k-p\ge1\) and $r=i-q$.
For \eqref{eq:zero-positive-shifted-factor}, rewrite the finite product as:
\[
\begin{aligned}
& k\binom{2i+1}{2q}
\frac{(2q+1)!(2r+k+1)!}{(2i+k+1)!}  
= k(2q+1)
  \frac{(2i+1)!}{(2r+1)!}
  \frac{(2r+k+1)!}{(2i+k+1)!} 
= k(2q+1)
 \prod_{h=1}^{2q}\frac{2r+1+h}{2r+k+1+h}.
\end{aligned}
\]
Each factor in the last product is at most
\((2i+1)/(2i+1+k)\), since it is a increasing function in $h$, $h\leq 2q$ and $r+q=i$.  Therefore, with \(x=k/(2i+1)\), the left-hand side is bounded by
\[
(2i+1)x(2q+1)(1+x)^{-2q}.
\]
For \(q\ge2\), the function \(x(2q+1)(1+x)^{-2q}\) is uniformly bounded in
\(x\ge0\).  Indeed, its maximum occurs at \(x=(2q-1)^{-1}\), and the maximum value
is
\[
\frac{2q+1}{2q-1}\left(1+\frac{1}{2q-1}\right)^{-2q}\lesssim 1,
\]
which proves \eqref{eq:zero-positive-shifted-factor}.
{For the first estimate in
\eqref{eq:shell-sum-small}, apply
\eqref{eq:endpoint-conv-one}.
For the second estimate, notice that
\[
\begin{aligned}
&\sum_{\substack{m_1,m_2\ge2\\ m_1+m_2=m+1}}
\frac{(1+m)^{a+1}}{(1+m_1)^{a+1}(1+m_2)^{a+1}}\le
\sum_{\substack{m_1,m_2\ge2\\ m_1+m_2=m+1}}
\frac{(2+m)^{a+1}}{(1+m_1)^{a+1}(1+m_2)^{a+1}}
\le C\Theta_a,
\end{aligned}
\]
where the last inequality follows from
\eqref{eq:endpoint-conv-one} with $m\leftarrow 1+m$.
This proves both estimates in \eqref{eq:shell-sum-small}.
}
Finally, to prove \eqref{eq:shell-sum-counted}, we write
\(n=q+p,\,m=i+k\), thus \(i-q+k-p=m-n\), and 
restrictions in the sum imply
\(2\le n\le m-2.\)
For each fixed \(n\), the condition \(q+p=n\) determines \(p\) once \(q\) is chosen,
so there are at most \(1+n\) admissible pairs \((q,p)\).  Hence the left-hand side
of \eqref{eq:shell-sum-counted} is bounded by
\[
\begin{aligned}
&(1+i+k)^{a+1}
\sum_{n=2}^{m-2}
\frac{(1+n)(1+m-n)}{(1+n)^{a+1}(1+m-n)^{a+1}}=
(1+m)^{a+1}
\sum_{n=2}^{m-2}
\frac{1}{(1+n)^a(1+m-n)^a}.
\end{aligned}
\]
If \(m<4\), the range is empty.  Otherwise, 
\eqref{eq:endpoint-conv-one} gives the desired inequality.
\end{proof}
\subsubsection{Decomposition of the nonlinear source}\label{subsec:source-decomposition}
{The smallness of source can either come from induction assumption on previous modes or background information. Since each inter-modal interaction only appears as a product of pair, it is therefore convenient to isolate the interactions involving at least one shell-one mode, from the interactions among active modes, and estimate them separately. This is achieved by the following decomposition of nonlinear source.}\par
For integers \(i,k\ge0\) with \(i+k\ge2\), split
\begin{equation}\label{eq:full-source-family-decomposition}
\begin{aligned}
N_{i,k}&=N_{i,k}^{\sharp}+N_{i,k}^{\flat},\\
N_{i,k}^{\sharp}
&:=\mathrm I_{i,k}^{\sharp}+\mathrm {II}_{i,k}^{\sharp}+\mathrm {III}_{i,k}^{\sharp}
  +\mathrm {IV}_{i,k}^{\sharp}+\mathrm {V}_{i,k}^{\sharp},\\
N_{i,k}^{\flat}
&:=\mathrm I^{\flat}_{i,k}+\mathrm {II}^{\flat}_{i,k}+\mathrm {III}^{\flat}_{i,k}
  +\mathrm {IV}^{\flat}_{i,k}+\mathrm V^{\flat}_{i,k}.
\end{aligned}
\end{equation}
The terms in \(N_{i,k}^{\sharp}\) are those for which both nonlinear factors
have index sum at least two:
\begin{align}
\mathrm I_{i,k}^{\sharp}
&:=\one_{i\ge2}\one_{k\ge1}(2i+1)(\pzi F_{i,0})F_{0,k+1},
\label{eq:family-I-aa}\\
\mathrm {II}_{i,k}^{\sharp}
&:=\one_{i\ge3}\sum_{q=2}^{i-1}
\binom{2i+1}{2q}(\pzi F_{q,0})F_{i-q,k+1},
\label{eq:family-II-aa}\\
\mathrm {III}_{i,k}^{\sharp}
&:=-\sum_{\substack{0\le p\le k,\ 0\le q\le i\\
q+p\ge2,\ i-q+k-p\ge2}}
\binom{k}{p}\binom{2i+1}{2q+1}F_{q,p}F_{i-q,k-p},
\label{eq:family-III-aa}\\
\mathrm {IV}_{i,k}^{\sharp}
&:=\one_{k\ge1}
\sum_{\substack{0\le p\le k-1,\ 0\le q\le i\\
q+p\ge2,\ i-q+k-p\ge2}}
\binom{k}{p+1}\binom{2i+1}{2q}F_{q,p}F_{i-q,k-p},
\label{eq:family-IV-aa}\\
\mathrm {V}_{i,k}^{\sharp}
&:=\one_{k=0}\one_{i\ge4}
\sum_{q=2}^{i-2}\binom{2i+2}{2q+1}s^{-1}
\int_0^sF_{q,0}F_{i-q,0}\,\dd z.
\label{eq:family-V-aa}
\end{align}
The remaining terms contain at least one of the two index-sum-one factors
\(F_{1,0}\) or \(F_{0,1}\):
\begin{align}
\mathrm I^{\flat}_{i,k}
&:=\one_{i=1}\one_{k\ge1}3(\pzi F_{1,0})F_{0,k+1},
\label{eq:family-I-flat}\\
\mathrm {II}^{\flat}_{i,k}
&:=\one_{i\ge2}\binom{2i+1}{2}(\pzi F_{1,0})F_{i-1,k+1},
\label{eq:family-II-flat}\\
\mathrm {III}^{\flat}_{i,k}
&:=-\sum_{\substack{0\le p\le k,\ 0\le q\le i\\
(q,p)\ne(0,0),(i,k)\\
q+p=1\ \text{or}\ i-q+k-p=1}}
\binom{k}{p}\binom{2i+1}{2q+1}F_{q,p}F_{i-q,k-p},
\label{eq:family-III-flat}\\
\mathrm {IV}^{\flat}_{i,k}
&:=\one_{k\ge1}
\sum_{\substack{0\le p\le k-1,\ 0\le q\le i\\
(q,p)\ne(0,0)\\
q+p=1\ \text{or}\ i-q+k-p=1}}
\binom{k}{p+1}\binom{2i+1}{2q}F_{q,p}F_{i-q,k-p},
\label{eq:family-IV-flat}\\
\mathrm V^{\flat}_{i,k}
&:=\one_{k=0}\one_{i\ge2}
\sum_{\substack{1\le q\le i-1\\q=1\ \text{or}\ i-q=1}}
\binom{2i+2}{2q+1}s^{-1}\int_0^sF_{q,0}F_{i-q,0}\,\dd z.
\label{eq:family-V-flat}
\end{align}
The restrictions above partition every summand of \eqref{eq:Nik}, and thus
\eqref{eq:full-source-family-decomposition} is disjoint and exhaustive. Recall that the induction hypothesis has the form of 
\begin{equation}\label{eq:available-active-normalized}
  A_{j,\ell}(s)\wh_\ell(s,z)|F_{j,\ell}(s,z)|\le M\quad j+l\geq 2
\end{equation}
We shall obtain a bound on \(N_{i,k}^{\sharp},N_{i,k}^{\flat}\) when \eqref{eq:available-active-normalized} is satisfied for each $F_{j,\ell}(s,z)$ appearing in them.
\subsubsection{Estimate on interactions among non-shell-one modes}
{Thanks to the combinatorial inequalities we developed in sub-subsection \ref{subsubsec:combinatorial}, for computational simplicity, let us now fix $$a:=3,$$ once for all. Consequently, in view of Lemma \ref{lem:active-active-preliminaries}$, \Theta_a=\Theta_3\lesssim 1$ can be absorbed as a universal constant, whenever it applies.}\par
The following proposition estimates \(N_{i,k}^{\sharp}\), in which both nonlinear factors have index sum at least two. 
\begin{lemma}[{Uniform estimates} for \(N_{i,k}^{\sharp}\)]\label{lem:mixed-combinatorics}
{Let \(\Lambda\ge1\), \(a=3\), \(0<\rho_0\le1\), and
\(s_0\ge s_1\geq s_{\text{ent}}(\Lambda)\) (\(s_{\text{ent}}(\Lambda)\) defined in \eqref{eq:s1-requirements}).  Let \(i,k\ge0\) be integers with \(i+k\ge2\), and suppose
that every factor occurring in \(N_{i,k}^{\sharp}\) is available in the sense
of \eqref{eq:available-active-normalized}.  Then, on
\(s\ge s_0\), \(0\le z\le s\),
\begin{equation}\label{eq:active-active-sum-bound}
  A_{i,k}\wh_k|N_{i,k}^{\sharp}|
  \le CM^2\mathcal D_{i,k},
\end{equation}
with a universal constant \(C\).}
\end{lemma}
\begin{proof}
{Notice that, by the choice of $s_1\geq s_{\text{ent}}(\Lambda)$ (defined in \eqref{eq:s1-requirements}) for every \(s\ge s_0\geq s_1\) and \(0\le z\le s\),
}\begin{equation}{\label{eq:hplus-hzero-large-s}
  \hh_+(s,z)\le\hh_0(s,z),
  \qquad
  \wh_0(s,z)\le \wh_+(s,z),
  \qquad
  \wh_0(s,z)\hh_+(s,z)\le1.
}\end{equation}
Furthermore, since $\rho(s)$ is decreasing and $\rho(s_0)=\rho_0$, we see
\begin{equation}
    \rho(s)\leq \rho_0 \leq 1 \text{ for all }s\geq s_0.
\end{equation}
These two facts will be used in this proof extensively. We shall show that
\begin{equation}\label{eq:active-active-family-bound}
A_{i,k}\wh_k|\mathrm I_{i,k}^{\sharp}|,
\quad\ldots,\quad
A_{i,k}\wh_k|\mathrm {V}_{i,k}^{\sharp}|
\le CM^2\mathcal D_{i,k},
\end{equation}
which, upon summing these estimates and using the definition of
\(N_{i,k}^{\sharp}\) in \eqref{eq:full-source-family-decomposition}, proves
\eqref{eq:active-active-sum-bound}.

\medskip
\noindent\textbf{Step 1: Three normalized templates.}
We pre-compute the three normalized reductions used in the five families. The symbol \(\mathfrak c\) denotes the concrete numerical coefficient
of the current summand. In the family estimates below, \(\mathfrak c\) is 
replaced by the actual coefficient of in the summation.
Recall from \eqref{eq:Aik-def} that, for every \(j,\ell\ge0\) with \(j+\ell\ge2\),
\begin{equation}\label{eq:Ajell-recall}
  A_{j,\ell}(s)
  =
  \frac{(1+j+\ell)^{a+1}\rho(s)^{2j+\ell-1}}{(2j+\ell+1)!}.
\end{equation}
\medskip
\noindent\emph{Type A: ordinary products.}
For a product \(\mathfrak c F_{q,p}F_{i-q,k-p}\) from Family III or IV,
both factors have index sum at least two, so the assumption in the
statement gives
\begin{equation}\label{eq:type-A-Astep}
  A_{i,k}\wh_k |\mathfrak c F_{q,p}F_{i-q,k-p}|
  \le
  M^2|\mathfrak c|\frac{A_{i,k}}{A_{q,p}A_{i-q,k-p}}
  \wh_k\hh_p\hh_{k-p}.
\end{equation}
Plugging the definition of \(A\) into the quotient gives
\begin{equation}\label{eq:rho-power-product}
  \begin{aligned}
 \frac{A_{i,k}}{A_{q,p}A_{i-q,k-p}}
 & =
 \rho(s)
 \frac{(1+i+k)^{a+1}}{(1+q+p)^{a+1}(1+i-q+k-p)^{a+1}}\times
 \frac{(2q+p+1)!(2(i-q)+k-p+1)!}{(2i+k+1)!}.
\end{aligned}
\end{equation}
Combining \eqref{eq:type-A-Astep} and \eqref{eq:rho-power-product},
every Type A summand satisfies
\begin{equation}\label{eq:type-A-ready}
  \begin{aligned}
  A_{i,k}\wh_k |\mathfrak c F_{q,p}F_{i-q,k-p}|
  &\le
  M^2 |\mathfrak c|\rho(s)
  \frac{(1+i+k)^{a+1}}{(1+q+p)^{a+1}(1+i-q+k-p)^{a+1}} \\
  &\quad\times
  \frac{(2q+p+1)!(2(i-q)+k-p+1)!}{(2i+k+1)!}
  \wh_k\hh_p\hh_{k-p}.
\end{aligned}
\end{equation}
\medskip
\noindent\emph{Type B: primitive products.}
For a primitive product \(\mathfrak c(\pzi F_{q,0})F_{i-q,k+1}\) from
Family I or II, both factors have index sum at least two. Combining with the primitive estimate in Lemma~\ref{lem:mixed-weight-estimates}, when
\(k\geq 0\), we compute (notice when $k=0$, we use \eqref{eq:hplus-hzero-large-s} to bound $w_0$ by $w_+$)
\begin{equation}\label{eq:type-B-Astep}
  A_{i,k}\wh_k |\mathfrak c(\pzi F_{q,0})F_{i-q,k+1}|
  \le
  C M^2 |\mathfrak c|\frac{A_{i,k}}{A_{q,0}A_{i-q,k+1}}.
\end{equation}
Plugging the definition of \(A_{i,k}\) into the quotient gives
\begin{equation}\label{eq:primitive-A-quotient}
  \begin{aligned}
\frac{A_{i,k}}{A_{q,0}A_{i-q,k+1}}
&=
\frac{(1+i+k)^{a+1}}{(1+q)^{a+1}(2+i-q+k)^{a+1}}\times
\frac{(2q+1)!(2(i-q)+k+2)!}{(2i+k+1)!}.
\end{aligned}
\end{equation}
Therefore, each Type B summand
satisfies
\begin{equation}\label{eq:type-B-ready}
  \begin{aligned}
  A_{i,k}\wh_k |\mathfrak c(\pzi F_{q,0})F_{i-q,k+1}|
  &\le
  C M^2 |\mathfrak c|
  \frac{(1+i+k)^{a+1}}{(1+q)^{a+1}(2+i-q+k)^{a+1}}\times
  \frac{(2q+1)!(2(i-q)+k+2)!}{(2i+k+1)!}.
\end{aligned}
\end{equation}
\medskip
\noindent\emph{Type C: averaged products.}
For the averaged term \(\mathfrak c s^{-1}\int_0^sF_{q,0}F_{i-q,0}\,\dd z\)
from Family V, the zero-mode average estimate
\eqref{eq:zero-avg-product} implies
\begin{equation}\label{eq:type-C-Astep}
  A_{i,0}\wh_0
  \left|\mathfrak c s^{-1}\int_0^sF_{q,0}F_{i-q,0}\,\dd z\right|
  \le
  C M^2 |\mathfrak c|\frac{A_{i,0}}{A_{q,0}A_{i-q,0}}.
\end{equation}
Plugging the definition of \(A\) into the quotient gives
\begin{equation}\label{eq:average-A-quotient}
  \begin{aligned}
\frac{A_{i,0}}{A_{q,0}A_{i-q,0}}
&=
\rho(s)
\frac{(1+i)^{a+1}}{(1+q)^{a+1}(1+i-q)^{a+1}}
\frac{(2q+1)!(2(i-q)+1)!}{(2i+1)!}.
\end{aligned}
\end{equation}
Therefore, every Type C summand satisfies
\begin{equation}\label{eq:type-C-ready}
  \begin{aligned}
  A_{i,0}\wh_0
  \left|\mathfrak c s^{-1}\int_0^sF_{q,0}F_{i-q,0}\,\dd z\right|
  &\le
  C M^2 |\mathfrak c|\rho(s)
  \frac{(1+i)^{a+1}}{(1+q)^{a+1}(1+i-q)^{a+1}} \times
  \frac{(2q+1)!(2(i-q)+1)!}{(2i+1)!}.
\end{aligned}
\end{equation}
For the rest of this proof, any term involving \(F_{1,0}\) or \(F_{0,1}\) is not in the families below;
These terms are estimated separately in Lemma~\ref{lem:shell-one-smallness}. We now attempt to bound
\eqref{eq:family-I-aa}--\eqref{eq:family-V-aa}, using the computations we did above.
\medskip

\noindent\textbf{Step 2: The primitive families.}
The two primitive families are of Type B; Family I is a single endpoint term and Family II is an interior sum.

\smallskip
\noindent\emph{Family I: primitive endpoint term \eqref{eq:family-I-aa}.}
This family is nonzero only for \(i\ge2\), \(k\ge1\). It is of Type B term with
\(q=i\) and \(\mathfrak c=2i+1\). The Type B computation reads
\[
A_{i,k}\wh_+ (2i+1)|\pzi F_{i,0}F_{0,k+1}|
\le
C M^2(2i+1)\frac{A_{i,k}}{A_{i,0}A_{0,k+1}}.
\]
The non-polynomial part of the fraction is uniformly bounded in this range:
\[
 E_{i,k}:=(2i+1)\frac{(2i+1)!(k+2)!}{(2i+k+1)!}\le C,
 \qquad i\ge1,\ k\ge1.
\]
Indeed, for $i\ge 1$, we can expand
\(E_{i,k}
=(2i+1)\frac{(k+2)!}{(2i+2)(2i+3)\cdots(2i+k+1)},\)
and observe
\(\frac{E_{i,k+1}}{E_{i,k}}=\frac{k+3}{2i+k+2}<1.\)
Thus, the worst case for \(k\ge1\) is \(k=1\), where
\(E_{i,1}=6\frac{2i+1}{2i+2}\le6.\)
Therefore,
\begin{equation}\label{eq:fam1-aa}
(2i+1)\frac{A_{i,k}}{A_{i,0}A_{0,k+1}}
\le
C
\frac{(1+i+k)^{a+1}}{(1+i)^{a+1}(2+k)^{a+1}}.
\end{equation}
For the polynomial part, we apply the second estimate in \eqref{eq:shell-sum-small} with
\(m=i+k\): Since the active shells are \(i\geq 2\) and \(k+1\geq 2\), whose sum is \(m+1\), the polynomial part is thus bounded by the summation, which, in turn, is 
bounded by \(C\). Using the restriction \(i\ge2\) and \(k\ge1\), the
definition \eqref{eq:Dik} gives \(\mathcal D_{i,k}\ge 1+i\ge1\). The
Type B estimate gives the Family I bound.
\medskip

\smallskip
\noindent\emph{Family II: interior primitive sum \eqref{eq:family-II-aa}.}
In this case, the summation ranges from
\( 2\le q\le i-1.\)
which implies
\(  i-q+k+1\ge2.\)
Each summand is of Type B term with
\(\mathfrak c=\binom{2i+1}{2q}\). The factorial-binomial coefficient ratio is
therefore
\begin{equation}\label{eq:fam2-factorial-cancel}
\begin{aligned}
&\binom{2i+1}{2q}
\frac{(2q+1)!(2(i-q)+k+2)!}{(2i+k+1)!}
=(2q+1)(2(i-q)+2)
\prod_{h=1}^{2q-1}
\frac{2(i-q)+2+h}{2(i-q)+k+2+h}.
\end{aligned}
\end{equation}
For \(k=0\), the product in \eqref{eq:fam2-factorial-cancel} is exactly one.
For \(k\ge1\), each factor in the product is at most one.
We claim that
\begin{equation}\label{eq:fam2-summed}
\sum_{q=2}^{i-1}
\binom{2i+1}{2q}
\frac{A_{i,k}}{A_{q,0}A_{i-q,k+1}}
\le C(1+i).
\end{equation}
If the sum is empty, there is nothing to prove. For \(k=0\),
\eqref{eq:fam2-factorial-cancel} gives
\[
\begin{aligned}
\text{left side of \eqref{eq:fam2-summed}}
&\le
C\sum_{q=2}^{i-1}
(1+q)(1+i-q)
\frac{(1+i)^{a+1}}{(1+q)^{a+1}(2+i-q)^{a+1}} \\
&\le
C(1+i)(1+i)^a
\sum_{q=2}^{i-1}\frac{1}{(1+q)^a(2+i-q)^a} \\
&\le C(1+i),
\end{aligned}
\]
where in the last line we used \((1+i)^a\le(2+i)^a\) and the endpoint convolution estimate
\eqref{eq:endpoint-conv-one} with exponent \(b=a\) and convolution parameter
\(m=i+1\).
Now assume \(k\ge1\). Since the product in
\eqref{eq:fam2-factorial-cancel} is at most one, the \(q\)-summand on the left
side of \eqref{eq:fam2-summed} satisfies
\begin{equation}\label{eq:fam2-positive-summand}
\begin{aligned}
&\binom{2i+1}{2q}
\frac{A_{i,k}}{A_{q,0}A_{i-q,k+1}} \\
&\qquad\le
C\,X_{q,k}\,B_{q,k}^{a},
\end{aligned}
\qquad
X_{q,k}:=(1+i-q)\frac{1+i+k}{2+i-q+k},
\quad
B_{q,k}:=\frac{1+i+k}{(1+q)(2+i-q+k)}.
\end{equation}
We simplify
\[
  X_{q,k}
  =(1+i-q)+\frac{(1+i-q)(q-1)}{2+i-q+k}
  \le (1+i-q)+(q-1)
  \le 1+i, \text{ uniformly in }k.
\]
Now, we split the sum into
\(q\le i-q\) and \(q>i-q\). If \(q\le i-q\), then
\[
B_{q,k}
=\frac{1+i+k}{(1+q)(2+i-q+k)}
\le \frac{2}{1+q},
\]
and hence the coefficient sum on this half satisfies (Since $a\geq 3$, the constant below is uniform.)
\[
\sum_{\substack{2\le q\le i-1\\ q\le i-q}}
X_{q,k}B_{q,k}^a
\le (1+i)\sum_{q=2}^{\infty}\left(\frac{2}{1+q}\right)^a
\le C(1+i)\left(\frac{2}{3}\right)^a.
\]
If \(q>i-q\), then for fixed \(q\) and \(i-q\) the quantity
\((1+i+k)/(2+i-q+k)\) is decreasing in \(k\). Combining with \(k\ge1\) and the fact that the restriction \(q>i-q\) gives \(i+2=q+(i-q)+2\le2(q+1)\),
\[
B_{q,k}
\le \frac{i+2}{(q+1)(i-q+3)}\le \frac{2}{i-q+3}
\le \frac{2}{i-q+2}.
\]
Therefore, the
coefficient sum on this half satisfies
\[
\sum_{\substack{2\le q\le i-1\\ q>i-q}}
X_{q,k}B_{q,k}^a
\le (1+i)\sum_{j=1}^{\infty}\left(\frac{2}{j+2}\right)^a
\le C(1+i)\left(\frac{2}{3}\right)^a.
\]
Since \(a=3\), this proves \eqref{eq:fam2-summed}. The range in Family II is
\(2\le q\le i-1\), which is empty unless \(i\ge3\). Hence, if \(k\ge1\),
then \(1+i\le\mathcal D_{i,k}\) by \eqref{eq:Dik}; if \(k=0\), then
\(1+i=\mathcal D_{i,0}\). Applying the Type B computation
\eqref{eq:type-B-ready} gives the Family II bound.
\medskip

\noindent\textbf{Step 3: The ordinary product families.}
Families III and IV are of Type A. For Families III, The cases $k=0$ and $k=1$ are treated differently. For $k=0$ the index range of Families III is then further divided into three subranges, depending on the index $q$, where different combinatorial estimates and weight products are used. 

\smallskip
\noindent\emph{Family III: ordinary product sum \eqref{eq:family-III-aa}.}
This family is of Type A. The range in \eqref{eq:family-III-aa} is
\[
0\le p\le k,\qquad 0\le q\le i,\qquad
q+p\ge2,\qquad i-q+k-p\ge2.
\]
For each summand in
this range, Type A and the factorial-binomial estimate
\eqref{eq:factorial-block}, with
\[
A_1=2q+1,\qquad C_1=2(i-q),\qquad P=p,
\qquad K=k-p,
\]
give the coefficient simplification
\begin{equation}\label{eq:ordinary-aa-coeff}
  \begin{aligned}
&\binom{k}{p}\binom{2i+1}{2q+1}
\frac{A_{i,k}}{A_{q,p}A_{i-q,k-p}}\le
C\rho(s)
\frac{(1+i+k)^{a+1}\bigl(1+i-q+k-p\bigr)}{(1+q+p)^{a+1}\bigl(1+i-q+k-p\bigr)^{a+1}}.
\end{aligned}
\end{equation}
The remaining weight product is the one
from \eqref{eq:type-A-ready}, namely
\(\wh_k\hh_p\hh_{k-p}.\)
We first restrict to the case where $k\geq 1$, in this case, the range \(0\le p\le k\) is the disjoint union of
\[
1\le p\le k-1,\qquad p=0,\qquad p=k,
\]
whereas the case \(k=0\) is treated separately at the end. In the estimates below, we shall repeatedly use the fact that
\(\rho(s)\le1\). 
\smallskip

\emph{Positive--interior subrange.} Assume \(k\ge1\) and restrict to the subrange
\[
1\le p\le k-1,
\qquad 0\le q\le i,
\qquad q+p\ge2,
\qquad i-q+k-p\ge2.
\]
This subrange is empty unless \(k\ge2\).
Using the reduction \eqref{eq:ordinary-aa-coeff}, the quantity to be
bounded is the following coefficient sum:
\begin{equation}\label{eq:ordinary-pp-pure-sum}
\sum_{\substack{1\le p\le k-1,\ 0\le q\le i\\ q+p\ge2,\ i-q+k-p\ge2}}
\frac{(1+i+k)^{a+1}\bigl(1+i-q+k-p\bigr)}{(1+q+p)^{a+1}\bigl(1+i-q+k-p\bigr)^{a+1}}
\le C(1+i+k),
\end{equation}
where \eqref{eq:shell-sum-counted} applies directly. The
weight product becomes
\(  \wh_k\hh_p\hh_{k-p}=\wh_+\hh_+^2=\hh_+.\)
Since this subrange has \(k\ge2\),
\((1+i+k)\hh_+\le C(2i+k)\hh_+\le C\mathcal D_{i,k}.\)
Thus \eqref{eq:type-A-ready}, \eqref{eq:ordinary-aa-coeff}, and
\eqref{eq:ordinary-pp-pure-sum} give the desired bound on the
Positive--interior part of Family III.
\smallskip

\emph{Positive--left-endpoint subrange.} Assume \(k\ge1\) and restrict to
\[
p=0,
\qquad 0\le q\le i,
\qquad q\ge2,
\qquad i-q+k\ge2.
\]
Directly computation on this subrange gives a sharper bound:
\[
\binom{2i+1}{2q+1}
\frac{(2q+1)!(2(i-q)+k+1)!}{(2i+k+1)!}
=(2(i-q)+1)\prod_{h=1}^{2q}
\frac{2(i-q)+1+h}{2(i-q)+k+1+h}
\le C(1+i).
\]
Therefore the quantity to be bounded after the Type A reduction is
\begin{equation}\label{eq:ordinary-zero-positive-first-pure}
(1+i)
\sum_{\substack{0\le q\le i\\ q\ge2,\ i-q+k\ge2}}
\frac{(1+i+k)^{a+1}}{(1+q)^{a+1}(1+i-q+k)^{a+1}}
\le C(1+i),
\end{equation}
{where the first estimate in \eqref{eq:shell-sum-small} is applied. }
The weight product in
\eqref{eq:type-A-ready} is
\( \wh_k\hh_0\hh_k=\wh_+\hh_0\hh_+=\hh_0\le C,\)
since this subrange has
\(q\ge2\), it is empty unless \(i\ge2\). Hence, for \(k\ge1\),
\(\mathcal D_{i,k}\ge 1+i\) by \eqref{eq:Dik}. Combining this with
\eqref{eq:type-A-ready} and
\eqref{eq:ordinary-zero-positive-first-pure} gives the desired bound on
the Positive--left-endpoint part of Family III.
\smallskip

\emph{Positive--right-endpoint subrange.} Assume \(k\ge1\) and restrict to
\[
p=k,
\qquad 0\le q\le i,
\qquad q+k\ge2,
\qquad i-q\ge2.
\]
Again, upon applying \eqref{eq:ordinary-aa-coeff} with
\(p=k\) and
\(C(1+i-q)\le C(1+i)\), the quantity to be bounded after the Type A
reduction is
\begin{equation}\label{eq:ordinary-positive-zero-second-pure}
(1+i)
\sum_{\substack{0\le q\le i\\ i-q\ge2,\ q+k\ge2}}
\frac{(1+i+k)^{a+1}}{(1+q+k)^{a+1}(1+i-q)^{a+1}}
\le C(1+i),
\end{equation}
{where the first estimate in \eqref{eq:shell-sum-small} is applied.}
The weight product in \eqref{eq:type-A-ready} is
\(\wh_k\hh_k\hh_0=\wh_+\hh_+\hh_0=\hh_0\le C.\)
Since this subrange has \(i-q\ge2\), it is empty unless \(i\ge2\). Hence
\(\mathcal D_{i,k}\ge1+i\) by \eqref{eq:Dik}. Combining this with
\eqref{eq:type-A-ready} and
\eqref{eq:ordinary-positive-zero-second-pure} gives the desired bound on
the Positive--right-endpoint of Family III.
\smallskip

\emph{Special zero case.} Finally, assume \(k=0\), then necessarily \(p=0\), and
the Family III range is
\[
2\le q\le i-2.
\]
The weight product in \eqref{eq:type-A-ready} is
\(  \wh_0\hh_0^2=\hh_0\le C.\)
Similarly, upon applying the first estimate in \eqref{eq:shell-sum-small},  the coefficient sum is bounded via 
\begin{equation}\label{eq:ordinary-zero-zero-pure}
(1+i)
\sum_{q=2}^{i-2}
\frac{(1+i)^{a+1}}{(1+q)^{a+1}(1+i-q)^{a+1}}
\le C(1+i).
\end{equation}
Since \(\mathcal D_{i,0}=1+i\), \eqref{eq:type-A-ready} gives the required bound.\par
These three subranges and the two cases above cover all the terms of Family
III, thus Family III is bounded by \(CM^2\mathcal D_{i,k}\).
\medskip
\smallskip

\noindent\emph{Family IV: shifted  product sum \eqref{eq:family-IV-aa}.}
This family is also of Type A . Each summand has the nonlinear product
\[
  F_{q,p}F_{i-q,k-p}.
\]
The Family IV has range \(k\ge1\) and
\(0\le p\le k-1\), so the second factor above is always a positive mode. Applying the shifted factorial-binomial estimate \eqref{eq:shifted-block}
to the Type A coefficient gives
\begin{equation}\label{eq:shifted-aa-coeff}
  \begin{aligned}
&\binom{k}{p+1}\binom{2i+1}{2q}
\frac{A_{i,k}}{A_{q,p}A_{i-q,k-p}}\le
C\rho(s)
\frac{(1+i+k)^{a+1}\bigl(1+i-q+k-p\bigr)}{(1+q+p)^{a+1}\bigl(1+i-q+k-p\bigr)^{a+1}}.
\end{aligned}
\end{equation}
Let us first consider the subrange \(k\ge1\), \(1\le p\le k-1\),
\(0\le q\le i\), with \(q+p\ge2\) and \(i-q+k-p\ge2\). Applying
\eqref{eq:shell-sum-counted} gives
\begin{equation}\label{eq:shifted-pp-pure-sum}
\sum_{\substack{1\le p\le k-1,\ 0\le q\le i\\ q+p\ge2,\ i-q+k-p\ge2}}
\frac{(1+i+k)^{a+1}\bigl(1+i-q+k-p\bigr)}{(1+q+p)^{a+1}\bigl(1+i-q+k-p\bigr)^{a+1}}
\le C(1+i+k).
\end{equation}
Since \(p\ge1\) and \(k-p\ge1\), the weight product in
\eqref{eq:type-A-ready} is
\( \wh_k\hh_p\hh_{k-p}=\wh_+\hh_+^2=\hh_+,\)
and \((1+i+k)\hh_+\le C\mathcal D_{i,k}\). Therefore, this part of Family IV is bounded by
\(CM^2\mathcal D_{i,k}\).
It remains to consider the subrange \(k\ge1\), \(p=0\), \(0\le q\le i\),
with \(q\ge2\) and \(i-q+k\ge2\). The shifted
estimate \eqref{eq:zero-positive-shifted-factor}, applied to this
zero--positive product, gives
\begin{align}
&\sum_{\substack{i\ge q\ge2\\ i-q+k\ge2}}
\binom{k}{1}\binom{2i+1}{2q}
\frac{A_{i,k}}{A_{q,0}A_{i-q,k}} \notag\le
C(1+i)
\sum_{\substack{q\ge2\\ i-q+k\ge2}}
\frac{(1+i+k)^{a+1}}{(1+q)^{a+1}(1+i-q+k)^{a+1}}.\label{eq:shifted-one-positive-reduction}
\end{align}
The right hand side is further bounded above via
\begin{equation}\label{eq:shifted-one-positive-sum}
(1+i)
\sum_{\substack{q\ge2\\ i-q+k\ge2}}
\frac{(1+i+k)^{a+1}}{(1+q)^{a+1}(1+i-q+k)^{a+1}}
\le C(1+i)
\end{equation}
{by the first estimate in \eqref{eq:shell-sum-small}. Since \(p=0\), the
weight product in}
\eqref{eq:type-A-ready} is thus
\(  \wh_k\hh_0\hh_k=\wh_+\hh_0\hh_+=\hh_0\le C.\)
Since \(q\ge2\), it is empty unless \(i\ge2\). Thus,
\eqref{eq:Dik} gives \(1+i\le\mathcal D_{i,k}\), and this part of Family IV
is thus bounded by \(CM^2\mathcal D_{i,k}\).
\medskip

\noindent\textbf{Step 4: The averaged product sum family.}
Family V is of Type C and occurs only for zero mode.

\smallskip

\noindent\emph{Family V: averaged product sum \eqref{eq:family-V-aa}.} This family corresponds to
Type C. In this case \(k=0\), and the range is \(2\le q\le i-2\). Multiplying the Type C quotient by the actual coefficient gives 
\[
\binom{2i+2}{2q+1}
\frac{(2q+1)!(2(i-q)+1)!}{(2i+1)!}=2i+2.
\]
Therefore,
\begin{equation}\label{eq:fam5-zero}
\binom{2i+2}{2q+1}
\frac{A_{i,0}}{A_{q,0}A_{i-q,0}}
\le
C\rho(s)(1+i)
\frac{(1+i)^{a+1}}{(1+q)^{a+1}(1+i-q)^{a+1}}.
\end{equation}
The coefficient sum satisfies
\begin{equation}\label{eq:fam5-zero-pure-sum}
(1+i)
\sum_{q=2}^{i-2}
\frac{(1+i)^{a+1}}{(1+q)^{a+1}(1+i-q)^{a+1}}
\le C(1+i)
\end{equation}
{by the first estimate in \eqref{eq:shell-sum-small}. Since
\(\mathcal D_{i,0}=1+i\), the Type C}
reduction gives the Family V bound.

\end{proof}

\subsubsection{Estimate on interactions with or within shell-one modes}
{The following lemma estimates the complementary part \(N_{i,k}^{\flat}\), whose component contains at least one of \(F_{1,0}\) or \(F_{0,1}\). Recall that by Corollary  \ref{lem:shell-one-asymptotic-bounds}, for each $ s_1\geq s^*_1(\Lambda)$, $B_1=B_1(s_1,\Lambda)$ is defined in \eqref{eq:Bone-def} and thus for $s\geq s_1, 0\leq z\leq s$,
\begin{equation}\label{eq:available-base-pointwise-recall}
  |F_{1,0}(s,z)|\le B_1\hh_0(s,z),
  \qquad
  |F_{0,1}(s,z)|\le B_1\hh_+(s,z).
\end{equation}}
\begin{lemma}[Uniform coefficients for \(N_{i,k}^{\flat}\)]\label{lem:shell-one-smallness}
{Let \(\Lambda\ge1\), \(a=3\), \(0<\rho_0\le1\),
\(s_0\ge s_1\geq s^*_1(\Lambda)\) (defined in \eqref{eq:s1-requirements}) and $B_1(s_1,\Lambda)$ be defined in Corollary \ref{lem:shell-one-asymptotic-bounds}.  Let \(i,k\ge0\) be integers with \(i+k\ge2\).  Assume that
every factor in \(N_{i,k}^{\flat}\) is available in
the sense of either \eqref{eq:available-active-normalized} or
\eqref{eq:available-base-pointwise-recall}. Then,
on \(s\ge s_0\), \(0\le z\le s\),
\begin{equation}\label{eq:shell-one-sum-bound}
  A_{i,k}\wh_k|N_{i,k}^{\flat}|
  \le C\rho_0(B_1M+B_1^2)\mathcal D_{i,k},
\end{equation}
with a universal constant \(C\).}
\end{lemma}
\begin{proof}
\medskip
\noindent\textbf{Step 1: Reusable weight bounds.}
As usual, we notice that
\(\wh_0\le\wh_+\), which follows from \eqref{eq:s1-requirements}, as
\(s_0\ge s_1\).
We will repeatedly use the following immediate consequences of
\eqref{eq:Dik}:
\begin{equation}\label{eq:shell-one-D-lower-bounds}
  \begin{aligned}
\mathcal D_{i,0}&=1+i &&(i\ge2),\\
\mathcal D_{i,k}&\ge1+i,
&\qquad \mathcal D_{i,k}&\ge(2i+k)\hh_+
&& (i\ge1,\ k\ge1),\\
\mathcal D_{i,k}&\ge k\hh_+ &&(i+k\ge2,\ k\ge1).
\end{aligned}
\end{equation}
In particular, for \(k\ge1\), the last inequality gives the following weight-scale
absorptions
\begin{equation}
\begin{split}
      &\hh_+\le \frac{1}k\mathcal D_{i,k}\le\mathcal D_{i,k}\quad k\geq1,\, i+k\geq2\\
     &1\leq \mathcal D_{i,k}\quad (k=0\text{ or }i\geq 1),\, i+k\geq2
\end{split}
\end{equation}
Moreover,
\(s\ge s_0\ge s_1\ge e\) gives
\begin{equation}\label{eq:hzero-bounded}
  \hh_0=e^{-z}+s^{-1}\le2.
\end{equation}
\medskip
\noindent\textbf{Step 2: Reusable reductions.}
Let \(i,k\ge0\) and \(i+k\ge2\), \(\mathfrak c\) be a concrete scalar
coefficient of the current summand and $F_{j,0},F_{j,l}$ satisfy the hypothesis \eqref{eq:available-active-normalized}. Combining \eqref{eq:pzi-zero-positive-product}, \eqref{eq:zero-avg-product}, and \(\wh_0\le\wh_+\), one can derive the following implications.

\begin{enumerate}
    \item If a primitive shell-one term contains \(\pzi F_{1,0}\) and one active
positive factor \(F_{j,\ell}\), then
\begin{equation}\label{eq:shell-pziF10-ready}
\left[
|\mathfrak c|\frac{A_{i,k}}{A_{j,\ell}}
\le C\rho_0\mathcal D_{i,k}
\right]
\Longrightarrow
A_{i,k}\wh_k|\mathfrak c (\pzi F_{1,0})F_{j,\ell}|
\le C\rho_0 B_1M\mathcal D_{i,k}.
\end{equation}
\item If a product contains \(F_{1,0}\) and one active factor ($j+l\geq 2$)
\(F_{j,\ell}\), then
\begin{equation}\label{eq:shell-F10-ready}
\left[
|\mathfrak c|\frac{A_{i,k}}{A_{j,\ell}}\wh_k\hh_0\hh_\ell
\le C\rho_0\mathcal D_{i,k}
\right]
\Longrightarrow
A_{i,k}\wh_k|\mathfrak c F_{1,0}F_{j,\ell}|
\le C\rho_0 B_1M\mathcal D_{i,k}.
\end{equation}
\item  If a  product contains \(F_{0,1}\) and one active factor
\(F_{j,\ell}\), then
\begin{equation}\label{eq:shell-F01-ready}
\left[
|\mathfrak c|\frac{A_{i,k}}{A_{j,\ell}}\wh_k\hh_+\hh_\ell
\le C\rho_0\mathcal D_{i,k}
\right]
\Longrightarrow
A_{i,k}\wh_k|\mathfrak c F_{0,1}F_{j,\ell}|
\le C\rho_0 B_1M\mathcal D_{i,k}.
\end{equation}


\item If an averaged  term contains \(F_{1,0}\) and active
zero mode \(F_{j,0}\), then
\begin{equation}\label{eq:shell-average-F10-ready}
\left[
|\mathfrak c|\frac{A_{i,0}}{A_{j,0}}
\le C\rho_0\mathcal D_{i,0}
\right]
\Longrightarrow
A_{i,0}\wh_0
\left|
\mathfrak c\,s^{-1}\int_0^s F_{1,0}F_{j,0}\,\dd z
\right|
\le C\rho_0B_1M\mathcal D_{i,0}.
\end{equation}
\end{enumerate}

\begingroup
Similarly, upon using \eqref{eq:available-base-pointwise-recall}, for any fixed
coefficient \(\mathfrak c\), the following implications hold among shell-one self-interaction:
\begin{equation}\label{eq:finite-F10F10-ready1}
\left[
|\mathfrak c|A_{i,k}\wh_k\hh_0^2
\le C\rho_0\mathcal D_{i,k}
\right]
\Longrightarrow
A_{i,k}\wh_k|\mathfrak c F_{1,0}^2|
\le C\rho_0B_1^2\mathcal D_{i,k},
\end{equation}
\begin{equation}\label{eq:finite-F10F01-ready2}
\left[
|\mathfrak c|A_{i,k}\wh_k\hh_0\hh_+
\le C\rho_0\mathcal D_{i,k}
\right]
\Longrightarrow
A_{i,k}\wh_k|\mathfrak c F_{1,0}F_{0,1}|
\le C\rho_0B_1^2\mathcal D_{i,k},
\end{equation}
\begin{equation}\label{eq:finite-F01F01-ready3}
\left[
|\mathfrak c|A_{i,k}\wh_k\hh_+^2
\le C\rho_0\mathcal D_{i,k}
\right]
\Longrightarrow
A_{i,k}\wh_k|\mathfrak c F_{0,1}^2|
\le C\rho_0B_1^2\mathcal D_{i,k}.
\end{equation}
For the averaged product (necessarily \((i,k)=(2,0)\)),
\begin{equation}\label{eq:finite-average-F10F10-ready4}
\left[
|\mathfrak c|A_{i,0}
\le C\rho_0\mathcal D_{i,0}
\right]
\Longrightarrow
A_{i,0}\wh_0
\left|
\mathfrak c\,s^{-1}\int_0^sF_{1,0}^2\,\dd z
\right|
\le C\rho_0B_1^2\mathcal D_{i,0}.
\end{equation}
\endgroup

\medskip
\noindent\textbf{Step 3: Shell-one-Non-shell-one interaction.}
 In this step, we only focus on the cases where exactly one of the modes in the product is shell-one ($F_{1,0}$ or $F_{0,1}$), and the other one is thus active. The case where both modes are of shell-one is left in Step 4.

\smallskip
\noindent\emph{Family \(\mathrm{I}^{\flat}\): primitive endpoint term \eqref{eq:family-I-flat}.}
This family occurs only for \(i=1\), \(k\ge1\):
\[
(2i+1)(\pzi F_{i,0})F_{0,k+1}
=3(\pzi F_{1,0})F_{0,k+1}.
\]
From \eqref{eq:Aik-def}, we compute
\begin{equation}\label{eq:fam1-F10-shell}
\begin{aligned}
3\frac{A_{1,k}}{A_{0,k+1}}
&=3
\frac{(k+2)^{a+1}\rho(s)^{k+1}}{(k+3)!}
\frac{(k+2)!}{(k+2)^{a+1}\rho(s)^k}=\frac{3\rho(s)}{k+3}
\le \rho(s)
\le \rho_0.
\end{aligned}
\end{equation}
As \(\mathcal D_{1,k}\ge1\), the hypothesis in \eqref{eq:shell-pziF10-ready} thus holds. Therefore,
\[
A_{1,k}\wh_k|\mathrm I^{\flat}_{1,k}|
\le C\rho_0B_1M\mathcal D_{1,k}.
\]
\medskip
\smallskip
\noindent\emph{Family \(\mathrm{II}^{\flat}\): interior primitive sum \eqref{eq:family-II-flat}.}
This case satisfies $i\geq 2$, the coefficient ratio is
\begin{equation}\label{eq:q-one-primitive}
\begin{aligned}
\binom{2i+1}{2}\frac{A_{i,k}}{A_{i-1,k+1}}
&=\frac{i(2i+1)}{2i+k+1}\rho(s)\le i\rho(s)
\le \rho_0(1+i).
\end{aligned}
\end{equation}
In view of the reduction \eqref{eq:shell-pziF10-ready}, since \(1+i\le\mathcal D_{i,k}\) for both \(k=0\) and
\(k\ge1\),
\[
A_{i,k}\wh_k|\mathrm{II}^{\flat}_{i,k}|
\le C\rho_0B_1M\mathcal D_{i,k}.
\]
\medskip
\smallskip
\noindent\emph{Family \(\mathrm{III}^{\flat}\): ordinary product sum \eqref{eq:family-III-flat}.}
   The restriction in \eqref{eq:family-III-flat} decomposes the sum as:
\[
(q,p)=(1,0),\quad(0,1),\quad(i-1,k),\quad(i,k-1).
\]
These four possibilities place \(F_{1,0}\) or \(F_{0,1}\)
in either factor. In each case, one of the factors has index sum
\((i+k)-1.\)
If \(i+k=2\), that other factor is also shell one; these terms are listed in
\eqref{eq:finite-two-shell-one-list} and estimated in the next step.\par
Assuming \(i+k\ge3\), we observe that
 \[
  \left(\frac{1+i+k}{i+k}\right)^{a+1}
  \le \left(\frac{4}{3}\right)^{a+1},
  \qquad i+k\ge3,
 \]
 and this fixed factor is absorbed into \(C\). Cases (a) and (b) below
verify the hypothesis of \eqref{eq:shell-F01-ready}, whereas cases (c) and
(d) verify the hypothesis of \eqref{eq:shell-F10-ready}.

\noindent\emph{(a) First factor is \(F_{0,1}\).}
This is the case where \((q,p)=(0,1)\), then necessarily \(k\ge1\) and the second factor is
\(F_{i,k-1}\), which has shell number \(i+k-1\). Since \(i+k\ge3\), \(F_{i,k-1}\) is active and
 \begin{equation}\label{eq:ordinary-F01-first}
\begin{aligned}
k\binom{2i+1}{1}\frac{A_{i,k}}{A_{i,k-1}}
&=k(2i+1)\rho(s)
\left(\frac{1+i+k}{i+k}\right)^{a+1}
\frac{1}{2i+k+1}\le C\rho_0 k.
\end{aligned}
\end{equation}
If \(k\ge2\), then the weight product is
\(\wh_k\hh_+\hh_{k-1}
  =\wh_+\hh_+^2
  =\hh_+.\)
Combining this identity with \eqref{eq:ordinary-F01-first} gives the factor
\(C\rho_0 k\hh_+\), which is bounded by
\(C\rho_0\mathcal D_{i,k}\) by
\eqref{eq:shell-one-D-lower-bounds}. If \(k=1\), the active condition
\(i+k\ge3\) gives \(i\ge2\), and the companion factor is the zero mode \(F_{i,0}\). In this case the weight product becomes
   \(  \wh_+\hh_+\hh_0=\hh_0\le2 \le 1+i\le D_{i,1}.\)
The left hand side of \eqref{eq:ordinary-F01-first} is then bounded by
\(C\rho_0\). Both subcases give the estimate
\[
  A_{i,k}\wh_k
  \bigl|k(2i+1)F_{0,1}F_{i,k-1}\bigr|
  \le C\rho_0B_1M\mathcal D_{i,k}.
\]
Notice that if \(i+k=2\), the possible
inputs are exactly \( (i,k)=(0,2)\) and \((1,1)\); the corresponding terms are
listed in \eqref{eq:finite-two-shell-one-list}. 

\noindent\emph{(b) Second factor is \(F_{0,1}\).}
   This is the case where \((i-q,k-p)=(0,1)\). Thus, \((q,p)=(i,k-1)\) the first factor is
\(F_{i,k-1}\) with $k\geq 1$ necessarily. As \(i+k\ge3\), the coefficient ratio is
 \begin{equation}\label{eq:ordinary-F01-second}
\begin{aligned}
k\binom{2i+1}{2i+1}\frac{A_{i,k}}{A_{i,k-1}}
&=k\rho(s)
\left(\frac{1+i+k}{i+k}\right)^{a+1}
\frac{1}{2i+k+1}\le C\rho_0.
\end{aligned}
\end{equation}
If \(k\ge2\geq 1\), the weight product is
is
\(  \wh_k\hh_{k-1}\hh_+
  =\wh_+\hh_+^2
  =\hh_+ \leq  k\hh_+\leq D_{i,k}.\)
If \(k=1\), then \(i+k\ge3\) gives
 \(i\ge2\), and
\(  \wh_+\hh_0\hh_+=\hh_0\le2 \leq 1+i\, \mathcal \le D_{i,1}.\)
Therefore, for both subcases
\[
  A_{i,k}\wh_k
  \bigl|kF_{i,k-1}F_{0,1}\bigr|
  \le C\rho_0B_1M\mathcal D_{i,k}.
\]
Again, when \(i+k=2\), the inputs are exactly \( (i,k)=(0,2)\) and \((1,1)\);
the corresponding terms are listed in
\eqref{eq:finite-two-shell-one-list} and estimated in the next step.

\noindent\emph{(c) First factor is \(F_{1,0}\).}
This is the case where \((q,p)=(1,0)\), so necessarily \(i\ge1\), and the second factor is
\(F_{i-1,k}\), with shell number \(i+k-1\). Since \(i+k\ge3\), 
\begin{equation}\label{eq:ordinary-F10-first}
\begin{aligned}
\binom{2i+1}{3}\frac{A_{i,k}}{A_{i-1,k}}
&=\binom{2i+1}{3}\rho(s)^2
\left(\frac{1+i+k}{i+k}\right)^{a+1}
\frac{1}{(2i+k)(2i+k+1)}\le C\rho_0^2(1+i)
\le C\rho_0(1+i).
\end{aligned}
\end{equation}
If \(k\ge1\), the
weight product is
\(  \wh_k\hh_0\hh_k
  =\wh_+\hh_0\hh_+
  =\hh_0\le2.\)
If \(k=0\), the weight product becomes
\(  \wh_0\hh_0^2=\hh_0\le2.\)
Since \( 1+i\le\mathcal D_{i,k}\), for case (c), it satisfies
\[
  A_{i,k}\wh_k
  \left|\binom{2i+1}{3}F_{1,0}F_{i-1,k}\right|
  \le \rho_0B_1M\mathcal D_{i,k}.
\]
When \(i+k=2\), the possible
inputs are exactly \((i,k)=(1,1)\) and \((2,0)\); the corresponding terms are
listed in \eqref{eq:finite-two-shell-one-list}.

\noindent\emph{(d) Second factor is \(F_{1,0}\).}
 This is the case where \((i-q,k-p)=(1,0)\), so \((q,p)=(i-1,k)\), $i\geq 1$, and the first factor is
\(F_{i-1,k}\). Since \(i+k\ge3\), 
 \begin{equation}\label{eq:ordinary-F10-second}
\begin{aligned}
\binom{2i+1}{2i-1}\frac{A_{i,k}}{A_{i-1,k}}
&=\binom{2i+1}{2}\rho(s)^2
\left(\frac{1+i+k}{i+k}\right)^{a+1}
\frac{1}{(2i+k)(2i+k+1)}\le \rho_0^2
\le \rho_0.
\end{aligned}
\end{equation}
If \(k\ge1\), the first factor is
positive and
\(  \wh_k\hh_k\hh_0
  =\wh_+\hh_+\hh_0
  =\hh_0\le2.\)
If \(k=0\), then
\(  \wh_0\hh_0^2=\hh_0\le2.\)
Using \(\mathcal D_{i,k}\ge1\), we see in case (d),
\[
  A_{i,k}\wh_k
  \left|\binom{2i+1}{2i-1}F_{i-1,k}F_{1,0}\right|
  \le \rho_0B_1M\mathcal D_{i,k}.
\]
When \(i+k=2\), the possible inputs are exactly
 \((i,k)=(1,1)\) and \((2,0)\); the corresponding terms are listed in
\eqref{eq:finite-two-shell-one-list} and estimated in the next step.

The terms argued in (a)--(d) therefore contribute
\(\rho_0B_1M\mathcal D_{i,k}.\)
On the other hand, the corresponding terms listed in \eqref{eq:finite-two-shell-one-list} (shell-one self-interaction terms) are argued in Step 4, where we show they contribute at most \(\rho_0B_1^2\mathcal D_{i,k}\). Therefore, altogether
\[
A_{i,k}\wh_k|\mathrm{III}^{\flat}_{i,k}|
\le \rho_0(B_1M+B_1^2)\mathcal D_{i,k}.
\]

\medskip
\smallskip
\noindent\emph{Family \(\mathrm{IV}^{\flat}\): shifted product sum \eqref{eq:family-IV-flat}.}
In this case, necessarily we have \(k\ge1\), and the second factor always has positive second index
\(k-p\ge1\). The restriction in \eqref{eq:family-IV-flat}
yields the following possibilities
\[
(q,p)=(1,0),\quad(0,1),\quad(i,k-1).
\]
These possibilities place
\(F_{0,1}\) in one of the two factors or \(F_{1,0}\) in the first factor.
In either case, the other factor has index sum \(i+k-1\). If \(i+k=2\), both factors are of shell one; the resulting
terms are listed in \eqref{eq:finite-two-shell-one-list} and estimated
in the next step. \par
We now assume \(i+k\ge3\), below, cases (a) and (b) below
verify the hypothesis of \eqref{eq:shell-F01-ready}, while case (c)
verifies the hypothesis of \eqref{eq:shell-F10-ready}.
\medskip

\noindent\emph{(a) First factor is \(F_{0,1}\).}
This case corresponds to \((q,p)=(0,1)\), and the second factor is
\(F_{i,k-1}\). Thus, necessarily \(k\ge2\). Since \(i+k\ge3\),
 \begin{equation}\label{eq:shifted-F01-first}
\begin{aligned}
\binom{k}{2}\frac{A_{i,k}}{A_{i,k-1}}
&=\binom{k}{2}\rho(s)
\left(\frac{1+i+k}{i+k}\right)^{a+1}
\frac{1}{2i+k+1}\le \rho_0 k.
\end{aligned}
\end{equation}
The weight product becomes
\(  \wh_k\hh_+\hh_{k-1}
  =\wh_+\hh_+^2
  =\hh_+.\)
Therefore, by Step 1,
\(C\rho_0 k\hh_+\le C\rho_0\mathcal D_{i,k}\), which implies
\[
  A_{i,k}\wh_k
  \left|\binom{k}{2}F_{0,1}F_{i,k-1}\right|
  \le C\rho_0B_1M\mathcal D_{i,k}.
\]
Notice that in this subcase, the condition \(i+k=2\) forces the unique input \((i,k)=(0,2)\),
whose term is listed in \eqref{eq:finite-two-shell-one-list} and estimated
in the next step. 

\noindent\emph{(b) Second factor is \(F_{0,1}\).}
In this case, \((i-q,k-p)=(0,1)\), \((q,p)=(i,k-1)\), and the first factor is thus
\(F_{i,k-1}\). Since \(i+k\ge3\), we have
 \begin{equation}\label{eq:shifted-F01-second}
\begin{aligned}
\binom{k}{k}\binom{2i+1}{2i}\frac{A_{i,k}}{A_{i,k-1}}
&=(2i+1)\rho(s)
\left(\frac{1+i+k}{i+k}\right)^{a+1}
\frac{1}{2i+k+1}\le C\rho_0.
\end{aligned}
\end{equation}
If \(k\ge2\), the first factor is positive and the weight product
is
\(  \wh_k\hh_{k-1}\hh_+
  =\wh_+\hh_+^2
  =\hh_+\leq k\hh_+ \leq \mathcal  D_{i,k}.\) On the other hand,
if \(k=1\), then \(i+k\ge3\) forces
\(i\ge2\), and the weight product becomes
   \(  \wh_+\hh_0\hh_+=\hh_0\le2 \), while \(1+i\le \mathcal D_{i,1}.\)
In either case, we have
\[
  A_{i,k}\wh_k
  \bigl|(2i+1)F_{i,k-1}F_{0,1}\bigr|
  \le C\rho_0B_1M\mathcal D_{i,k}.
\]
Again, in this subcase, if \(i+k=2\), the inputs are exactly \((i,k)=(0,2)\) and \((1,1)\); the
corresponding terms are listed in \eqref{eq:finite-two-shell-one-list}
and estimated in the next step. 

\noindent\emph{(c) First factor \(F_{1,0}\).}
In this case \((q,p)=(1,0)\) and the second factor is
\(F_{i-1,k}\). Necessarily, we have \(i\ge1\). Since
\(i+k\ge3\), 
 \begin{equation}\label{eq:shifted-F10-first}
\begin{aligned}
k\binom{2i+1}{2}\frac{A_{i,k}}{A_{i-1,k}}
&=ki(2i+1)\rho(s)^2
\left(\frac{1+i+k}{i+k}\right)^{a+1}
\frac{1}{(2i+k)(2i+k+1)}\\
&\le C\rho_0^2(1+i)
\le C\rho_0(1+i).
\end{aligned}
\end{equation}
The weight product is
\( \wh_+\hh_0\hh_+
  =\hh_0\le2\) and \(1+i\le\mathcal D_{i,k}\), the exact case (c) term satisfies
\[
  A_{i,k}\wh_k
  \left|k\binom{2i+1}{2}F_{1,0}F_{i-1,k}\right|
  \le C\rho_0B_1M\mathcal D_{i,k}.
\]
In this subcase, the condition \(i+k=2\) gives the unique input \((i,k)=(1,1)\), the corresponding term is
listed in \eqref{eq:finite-two-shell-one-list}.\par 
The terms argued in (a)--(c) therefore contribute at most
\(C\rho_0B_1M\mathcal D_{i,k}\), while we will show the terms listed in \eqref{eq:finite-two-shell-one-list} contribute at most \(C\rho_0B_1^2\mathcal D_{i,k}\) by the estimates in Step 4.
Altogether,
\[
A_{i,k}\wh_k|\mathrm{IV}^{\flat}_{i,k}|
\le C\rho_0(B_1M+B_1^2)\mathcal D_{i,k}.
\]
\medskip
\smallskip
\noindent\emph{Family \(\mathrm{V}^{\flat}\): averaged product sum \eqref{eq:family-V-flat}.}
In this case, \(k=0\). If \(i\ge3\), the only possible indices in the summation are \(q=1\) and
\(q=i-1\), corresponding to the companion factor
\(F_{i-1,0}, F_{1,0}\) respectively. Using weight reduction, we compute that for $i\geq 3$
\begin{equation}\label{eq:avg-F10-endpoint}
\begin{aligned}
\binom{2i+2}{3}\frac{A_{i,0}}{A_{i-1,0}}=\frac{2i+2}{6}\rho(s)^2
\left(\frac{i+1}{i}\right)^{a+1}\le C\rho_0^2(1+i)
\le C\rho_0\mathcal D_{i,0}.
\end{aligned}
\end{equation}
   Therefore, the reduction \eqref{eq:shell-average-F10-ready} applies, and by symmetry, each endpoint is bounded by
 \[
 C\rho_0B_1M\mathcal D_{i,0}.
\]
On the other hand, for \(i=2\), the unique summand is the averaged product:
\(F_{1,0}F_{1,0}\), which is the Family \(\mathrm V^{\flat}_{2,0}\) term
listed in \eqref{eq:finite-two-shell-one-list} and we will show it contributes at most \(C\rho_0B_1^2\mathcal D_{i,k}\) in Step 4. Altogether, 
\[
A_{i,k}\wh_k|\mathrm{V}^{\flat}_{i,k}|
\le C\rho_0(B_1M+B_1^2)\mathcal D_{i,k}.
\]

\medskip
\noindent\textbf{Step 4: {Inter-Shell-one interaction}.}
The preceding arguments leave exactly the following six 
contributions, resulting from interactions between shell-one modes:
\begin{equation}\label{eq:finite-two-shell-one-list}
\begin{aligned}
\mathrm{III}^{\flat}_{0,2}&=-2F_{0,1}^2,
&\qquad \mathrm{IV}^{\flat}_{0,2}&=F_{0,1}^2,\\
\mathrm{III}^{\flat}_{1,1}&=-4F_{1,0}F_{0,1},
&\mathrm{IV}^{\flat}_{1,1}&=3F_{1,0}F_{0,1},\\
\mathrm{III}^{\flat}_{2,0}&=-10F_{1,0}^2,
&\mathrm{V}^{\flat}_{2,0}&=20s^{-1}\int_0^sF_{1,0}^2\,\dd z.
\end{aligned}
\end{equation}
Thus the list contains every two-shell-one summand, with no overlap or
omission.
\begingroup
One can explicitly compute that
\begin{equation}\label{eq:finite-output-A-values}
A_{0,2}=\frac{3^{a+1}}{3!}\rho(s),
\qquad
A_{1,1}=\frac{3^{a+1}}{4!}\rho(s)^2,
\qquad
A_{2,0}=\frac{3^{a+1}}{5!}\rho(s)^3.
\end{equation}
The hypotheses \eqref{eq:finite-F10F10-ready1}-\eqref{eq:finite-average-F10F10-ready4} in Step 1 are therefore satisfied. Indeed, \eqref{eq:finite-two-shell-one-list} written in the same order as the six terms has the form of
\begin{equation}\label{eq:finite-two-shell-one-coefficient-checks}
\begin{aligned}
2A_{0,2}\wh_+\hh_+^2
&=2A_{0,2}\hh_+
\le C\rho_0\mathcal D_{0,2},\\
A_{0,2}\wh_+\hh_+^2
&=A_{0,2}\hh_+
\le C\rho_0\mathcal D_{0,2},\\
4A_{1,1}\wh_+\hh_0\hh_+
&=4A_{1,1}\hh_0
\le C\rho_0\mathcal D_{1,1},\\
3A_{1,1}\wh_+\hh_0\hh_+
&=3A_{1,1}\hh_0
\le C\rho_0\mathcal D_{1,1},\\
10A_{2,0}\wh_0\hh_0^2
&=10A_{2,0}\hh_0
\le C\rho_0\mathcal D_{2,0},\\
20A_{2,0}
&\le C\rho_0\mathcal D_{2,0},
\end{aligned}
\end{equation}
where we have used \eqref{eq:shell-one-D-lower-bounds} and $\rho(s)\leq \rho_0\leq 1$. Hence, each contribution resulting from shell-one self-interaction is
bounded by \(C\rho_0B_1^2\mathcal D_{i,k}\).
\endgroup

\medskip
\noindent\textbf{Step 5: summation of the family bounds.}
Summarizing the the estimates in Step 3, we obtain
\begin{equation}\label{eq:shell-one-family-bound}
A_{i,k}\wh_k|\mathrm I^{\flat}_{i,k}|,
\quad\ldots,\quad
A_{i,k}\wh_k|\mathrm V^{\flat}_{i,k}|
\le C\rho_0(B_1M+B_1^2)\mathcal D_{i,k},
\end{equation}
upon summing and using the definition of \(N_{i,k}^{\flat}\) in
\eqref{eq:full-source-family-decomposition} proves
\eqref{eq:shell-one-sum-bound}.
\end{proof}

\subsubsection{Conclusion}
\begin{corollary} [Weighted full-source bound]\label{lem:source-bound}
{Let \(\Lambda\ge1\), \(a=3\), \(0<\rho_0\le1\),
\(s_0\ge s_1\geq s_{\text{ent}}(\Lambda)\), $B_1(s_1,\Lambda)$ be defined in Corollary \ref{lem:shell-one-asymptotic-bounds}.  Let \(i,k\ge0\) be integers with \(i+k\ge2\).
Assume every factor $F_{j,l}$ inside $S_{i,k}$ satisfies either \eqref{eq:available-active-pointwise} or
\eqref{eq:available-base-pointwise}, depending whether $j+l\geq 2$ or $j+l=1$.  Then, for every \(s\ge s_0\) and
\(0\le z\le s\),

\begin{equation}\label{eq:source-bound}
  {|S_{i,k}(s,z)|
  \lesssim \Bigl[ M^2
     +\rho_0(B_1M+B_1^2)
  \Bigr]\mathcal D_{i,k}(s,z)}.
\end{equation}}
\end{corollary}
\begin{proof}
Combine \eqref{eq:full-source-family-decomposition} \eqref{eq:active-active-sum-bound} and \eqref{eq:shell-one-sum-bound}.
\end{proof}

\subsection{Closure of induction}
{In Corollary \ref{lem:source-bound}, the assumption that each active modes in the source term satisfy the target bounds is justified by a nested induction mentioned in Subsection \ref{subsec: nested induction}. Given that the comparison principle and the control of nonlinear source term have been established, we are now ready to close the induction. The following proposition summarizes the choice of parameters and subsequently implements a nested induction, resulting in uniform bounds for all active normalized modes $U_{i,k}$.}
\begin{proposition}[Nested induction]\label{thm:closure}
{Let \(\Lambda_+,C_{0,+}\) be the universal constants from
Lemma~\ref{lem:positive-gap} and $\mu_0$ be given in Lemma \ref{lem:linear-comparison}. {Fix $a=3$, and} choose any
\(\Lambda\ge\Lambda_+\), \(C_0\ge C_{0,+}\) and let \(\mu_0\) universal and $s_{\mathrm{lin}}(\Lambda)$ be the constants supplied by Lemma~\ref{lem:linear-comparison}.  Next, fix any
\(s_1\ge s_{\mathrm{ent}}(\Lambda)\) ($ s_{\mathrm{ent}}(\Lambda)$ defined in \eqref{eq:s1-requirements}), and let
\(B_1=B_1(s_1,\Lambda)\) be defined by \eqref{eq:Bone-def}. Then, choose any
\( s_0\ge\max\bigl\{s_1,s_{\mathrm{lin}}{(\Lambda}),
  s_{\mathrm{rad}}(C_0,\Lambda)\bigr\},\)
where we recall that $s_{\text{rad}}$ is defined in \eqref{eq:bound on analytical radius}. {Finally, choose \(M>0\) small enough, such that \(CM\le\mu_0/4\),} where \(C\) is the universal constant in
Corollary~\ref{lem:source-bound}, {and subsequently choose \(0<\rho_0\le1\) sufficiently small so that
\(C\rho_0(B_1+B_1^2/M)\le\mu_0/4\).}  
Assume that at $s=s_0$, the active initial data satisfy
\begin{equation}\label{eq:uniform-initial}
  \norm{U_{i,k}(s_0)}_{L^\infty_z}\le M,
  \qquad i+k\ge2.
\end{equation}
Then every active mode satisfies
\begin{equation}\label{eq:theorem-target}
  \norm{U_{i,k}(s)}_{L^\infty_z}\le M,
  \qquad i+k\ge2,
  \qquad s\ge s_0.
\end{equation}
}
\end{proposition}
\begin{proof}
{ The choices of \(a\), \(M\) and \(\rho_0\) imply
\(C\left[
     { M
     +\rho_0B_1
     +\rho_0\frac{B_1^2}{M}}
  \right]
  \le \frac{\mu_0}{2}.\)
Since
\(s_0\ge s_{\mathrm{rad}}(C_0,\Lambda)\),
\eqref{eq:bound on analytical radius} gives
\( C_0\Lambda^2
  \frac{(\log s_0)^2+2\log s_0+2}{s_0}
  \le\log2,\)
and hence \( \rho_0/2\le \rho(s)\le \rho_0,\,\text{for }
s\ge s_0.\) Proceeding in the order of Subsection~\ref{subsec: nested induction}, assuming that the desired estimate holds for all the preceding modes, for the current mode, the hypothesis of Corollary~\ref{lem:source-bound} is thus satisfied, which yields the smallness of nonlinear source. In light of Lemma~\ref{lem:linear-comparison}, with $R:=M$, the assumption on initial smallness and the smallness of nonlinear source are thus satisfied, which results a bound of the current mode, for all time. The rest is simply bookkeeping.}

We prove the statements \(\mathsf P_m\) (defined in \eqref{eq:Pm}), \(m\ge2\), using the nested induction
from Subsection~\ref{subsec: nested induction}.
First consider the base shell \(m=2\), where the modes are treated in the
order
\[
  (0,2)\longrightarrow(1,1)\longrightarrow(2,0).
\]
For \((0,2)\), there is no earlier same-shell mode. In particular, its source is
therefore controlled by the shell-one estimates \eqref{eq:available-base-pointwise}, which have already been covered in the proof of Lemma {~} \ref{lem:shell-one-smallness}.  Hence Corollary {~} \ref{lem:source-bound}
applies, and it gives
\[
  |S_{0,2}|
  \le \frac{\mu_0}{2}M\mathcal D_{0,2}.
\]
Combining with the bound on the initial data, the linear comparison lemma (Lemma \ref{lem:linear-comparison}) implies \(\|U_{0,2}\|_{L^\infty}\le M\), which makes the bound on \((0,2)\) available. Repeating the same argument gives the bound on \((1,1)\), and
then, with the bounds on \((0,2)\) and \((1,1)\) available, gives the bound on \((2,0)\). Thus,
\(\mathsf P_2\) is proved.\par

Let \(m\ge3\), we now prove the modes in \(\mathcal S_m\) in the inner order
\eqref{eq:shell-order}. The inner base case is \((0,m)\): as there is no earlier same-shell mode, using only the estimates on lower-shell modes (induction assumption) and shell-one modes (background information),  Corollary~\ref{lem:source-bound} and Lemma~\ref{lem:linear-comparison} apply, yielding a bound for \((0,m)\).  At the step \((i,k)\), the lower-shell estimates and the 
estimates of the same-shell modes with smaller first index and larger second index, and the shell-one
background information are exactly the hypotheses of Corollary {~} \ref{lem:source-bound}.  Therefore,
\[
 |S_{i,k}(s,z)|
  \leq C \Bigl[ M^2
     +\rho_0(B_1M+B_1^2)
  \Bigr]\mathcal D_{i,k}(s,z)
\]
which is at most
\((\mu_0/2)M\mathcal D_{i,k}\), by the parameter choice above. Lemma {~} \ref{lem:linear-comparison} then gives \(\|U_{i,k}\|_{L^\infty}\le M\).  This advances the inner induction, up and until
the last pair \((m,0)\).
Strong induction in \(m\) gives \eqref{eq:theorem-target} for every active
shell, with shell number \(m\ge2\).
\end{proof}
{The following corollary translates the closure result on the normalized hierarchy $(U_{i,k})$, into a result that controls the analytical growth of the original hierarchy $(F_{i,k})$, as well as establishes a lower bound for the analytic radius on the in self-similar coordinates.}
\begin{corollary}[Lower bound on the analytic radius in self-similar variable]
\label{cor:pointwise-step2}
Let $M>0$. There exist constants
{\( \Lambda\ge1,
  \, C_0>0,\,s_0\ge1 \text{ sufficiently large},\text{ as well as, }
  \, M>0 \) sufficiently small and then \(0<\rho_0\le1 \)\text{ sufficiently small},
}constituting a decreasing function
\begin{equation}\label{eq:rho-corollary}
  \rho(s)
  =\rho_0\exp\left(
    -C_0\Lambda^2\int_{s_0}^{s}
      \frac{(\log\sigma)^2}{\sigma^2}\,d\sigma
  \right),
  \qquad s\ge s_0,
\end{equation}
such that the following holds.  Let
\begin{equation}\label{eq:h0-hplus-corollary}
  \hh_0(s,z)=e^{-z}+s^{-1},
  \qquad
  \hh_+(s,z)=e^{-z}+\Lambda^2\frac{(\log s)^2}{s^2}.
\end{equation}
If the active initial coefficients satisfy
\begin{equation}\label{eq:raw-initial}
\begin{aligned}
\left\|
  \frac{F_{i,0}(s_0,\cdot)}{\hh_0(s_0,\cdot)}
\right\|_{L^\infty_z}
&\le
M\frac{(2i+1)!}{(1+i)^{{4}}}\rho_0^{\,1-2i},
&& i\ge2,\\[2mm]
\left\|
  \frac{F_{i,k}(s_0,\cdot)}{\hh_+(s_0,\cdot)}
\right\|_{L^\infty_z}
&\le
M\frac{(2i+k+1)!}{(1+i+k)^{{4}}}\rho_0^{\,1-2i-k},
&& k\ge1,\quad i+k\ge2,
\end{aligned}
\end{equation}
then, for every \(s\ge s_0\) and \(0\le z\le s\),
\begin{equation}\label{eq:raw-conclusion}
\begin{aligned}
|F_{i,0}(s,z)|
&\le
M\frac{(2i+1)!}{(1+i)^{{4}}}\rho(s)^{\,1-2i}\hh_0(s,z),
&& i\ge2,\\[2mm]
|F_{i,k}(s,z)|
&\le
M\frac{(2i+k+1)!}{(1+i+k)^{{4}}}\rho(s)^{\,1-2i-k}\hh_+(s,z),
&& k\ge1,\quad i+k\ge2.
\end{aligned}
\end{equation}
More precisely,
\begin{equation}\label{eq:rho-infinity}
\rho(s)\longrightarrow\rho_\infty
:=\rho_0\exp\left(
-C_0\Lambda^2
\frac{(\log s_0)^2+2\log s_0+2}{s_0}
\right), \text{ and }
\rho_\infty\ge\frac{\rho_0}{2}>0.
\end{equation}
\end{corollary}
\begin{proof}
Applying Proposition~\ref{thm:closure}, with the definition of $U_{i,k}$ \eqref{eq:Aik-def}, \(U_{i,k}=A_{i,k}\wh_kF_{i,k}\), the definition of $A_{i,k}$ \eqref{eq:Aik-def}, as well as, the definition of weights \eqref{eq:hk-wk} \eqref{eq:mixed-h} gives 
\eqref{eq:raw-conclusion}.
While $\rho$ is decreasing by definition and its limit can be explicitly computed, the lower-boundedness is implied by the choice of $s_0\geq s_\text{rad}$ in Proposition~\ref{thm:closure}, which guarantees $\frac{\rho_0}{2}\leq \rho_\infty\leq \rho(s)\leq \rho_0$ (See \eqref{eq:bound on analytical radius}).
\end{proof}
\subsection{Reconstruction of the 2D blowup solution and its validity}\label{sec:physical-taylor-radius}
Recall that the physical coordinates $(t,Z)$ and self-similar coordinates $(s,z)$ are related by 
\begin{equation}
    s=-\log(T-t),z=\frac{Z}{\abs{\log(T-t)}^{-1}}
\end{equation}
{Whenever the solution $u$ to \eqref{2d primitive boundary condition} exists, we set $u_{i,k}=\partial_X^{2i+1}\partial_Z^k u (X=0)$.} Recall also that the derivatives of \(u\) are related to the rescaled hierarchy by
\begin{equation}\label{eq:xi-F-relation}
  \partial_X^{2i+1}\partial_Z^k u(t,0,Z)
  =(T-t)^{-3i-1}s(t)^k
  F_{i,k}\bigl(s(t),s(t)Z\bigr),
  \qquad 0\le Z\le1.
\end{equation}

The following theorem translates Corollary~\ref{cor:pointwise-step2}
into the original variables and then supplements additional properties for the analytical reconstruction. In particular, the first part gives bounds for the physical derivatives restricted on the symmetry axis, under the assumption of the boundedness of initial condition, and the second part computes a lower bound on the radius of convergence, thus establishes criterion on the validity of analytical extension.

\begin{proposition}
[Physical derivative bounds and time-dependent \(X\)-Taylor radius]
\label{thm:physical-x-radius}
{Let $u_{0,0}=-a$ and $u_{1,0}=b$ satisfies the conclusion of Theorem \ref{refined C1 expansion} and Theorem \ref{first shell mode expansions} respectively. Recall the definition of self-similar time $ s(t):=|\log(T-t)|,$ and that whenever the solution $u$ to \eqref{2d primitive boundary condition} exists, we set $u_{i,k}=\partial_X^{2i+1}\partial_Z^k u (X=0)$.}
{There exist constants $\Lambda\ge1, \, C_0>0,$ then a time \(t_0<T\), sufficiently close to \(T\), as well as \(M>0 \text{ sufficiently small}\text{ and then }0<\rho_0\le1 \text{ sufficiently small},\)}
such that the following
conditional statement holds.  Set $s_0:=s(t_0),$
and define, for \(t\in[t_0,T)\) and \(0\le Z\le1\), the physical weights
\begin{equation}\label{eq:physical-weights}
\begin{aligned}
  \mathcal H_0(t,Z)
  &:= (T-t)^Z+\frac{1}{s(t)},\\
  \mathcal H_+(t,Z)
  &:= (T-t)^Z
      +\Lambda^2\frac{\bigl(\log s(t)\bigr)^2}{s(t)^2}.
\end{aligned}
\end{equation}
Equivalently,
\(
\mathcal H_0(t,Z)=\hh_0\bigl(s(t),s(t)Z\bigr)
\)
and
\(
\mathcal H_+(t,Z)=\hh_+\bigl(s(t),s(t)Z\bigr).
\)
Define also
\begin{equation}\label{eq:physical-rho}
  \rho(\sigma)
  =\rho_0\exp\left(
    -C_0\Lambda^2\int_{s_0}^{\sigma}
      \frac{(\log\tau)^2}{\tau^2}\,d\tau
  \right),
  \quad \rho_0=\rho(s_0),\qquad \sigma\ge s_0,
\end{equation}
{Assume, in addition, that for every \(0\le Z\le1\), the active initial
coefficients satisfy the following two bounds.
\begin{equation}\label{eq:physical-initial}
\begin{aligned}
\left|u_{i,0}(t_0,Z)\right|
&\le
M\frac{(2i+1)!}{(1+i)^{4}}
(T-t_0)^{-3i-1}\rho_0^{\,1-2i}
\mathcal H_0(t_0,Z),
&& i\ge2,\\[2mm]
\left|u_{i,k}(t_0,Z)\right|
&\le
M\frac{(2i+k+1)!}{(1+i+k)^{4}}
(T-t_0)^{-3i-1}s_0^k\rho_0^{\,1-2i-k}
\mathcal H_+(t_0,Z),
&& k\ge1,\quad i+k\ge2.
\end{aligned}
\end{equation}
}
Then the following conclusions hold.
\begin{enumerate}[label=\textup{(\roman*)},leftmargin=2.3em]
\item \emph{Physical derivative bounds.}
For every \(t\in[t_0,T)\) and \(0\le Z\le1\),
{
\begin{equation}\label{eq:physical-derivative-conclusion}
\begin{aligned}
\left|u_{i,0}(t,Z)\right|
&\le
M\frac{(2i+1)!}{(1+i)^{4}}
(T-t)^{-3i-1}\rho\bigl(s(t)\bigr)^{\,1-2i}
\mathcal H_0(t,Z),
&& i\ge2,\\[2mm]
\left|u_{i,k}(t,Z)\right|
&\le
M\frac{(2i+k+1)!}{(1+i+k)^{4}}
(T-t)^{-3i-1}s(t)^k
\rho\bigl(s(t)\bigr)^{\,1-2i-k}
\mathcal H_+(t,Z),
&& k\ge1,\quad i+k\ge2.
\end{aligned}
\end{equation}
}
Moreover,
\begin{equation}\label{eq:physical-rho-limit}
  \rho\bigl(s(t)\bigr)\longrightarrow\rho_\infty
  \quad\text{as }t\uparrow T,
  \qquad
  \rho_\infty\ge\frac{\rho_0}{2}>0.
\end{equation}
\item \emph{Analytic continuation in \(X\).}
Assume, in addition, that the function \(u(t_0)\) is analytic in some
neighborhood of \(X=0\).  Define
{
\begin{equation}\label{eq:u-taylor-series}
  \widetilde u(t,X,Z)
  :=\sum_{i=0}^{\infty}
  \frac{u_{i,0}(t,Z)}{(2i+1)!}X^{2i+1}.
\end{equation}
}
This series has radius of convergence at least
\begin{equation}\label{eq:physical-radius}
  R_X(t):=(T-t)^{3/2}\rho\bigl(s(t)\bigr) \ge \rho_\infty(T-t)^{3/2}
  \ge \frac{\rho_0}{2}(T-t)^{3/2},
\end{equation}
converging absolutely for every \(|X|<R_X(t)\).  Moreover, for each fixed
\(t\in[t_0,T)\) and each fixed \(r\) with \(0<r<R_X(t)\), the series
converges uniformly in \((X,Z)\) on the domain
\[
  |X|\le r,
  \qquad 0\le Z\le1.
\]
\end{enumerate}
\end{proposition}

\begin{proof}
Choose the parameters $(\Lambda,C_0,M,\rho_0)$ and the initial time \(s_0\) supplied by
Corollary~\ref{cor:pointwise-step2}, and set \(s(t_0)=s_0\). 
The analytical bounds \eqref{eq:physical-derivative-conclusion} and  the convergence/lower-boundedness of radius \eqref{eq:physical-rho-limit} are direct consequences of Corollary \ref{cor:pointwise-step2}.
Fix \(t\in[t_0,T)\), write \(s=s(t)\), and fix \(0\le Z\le1\).
For \(i\ge2\), the first estimate in
\eqref{eq:physical-derivative-conclusion} and
\(\mathcal H_0(t,Z)\le2\) give
\[
  \frac{|u_{i,0}(t,Z)|}{(2i+1)!}
  \le
  \frac{2M}{(1+i)^{4}}
  (T-t)^{-3i-1}\rho(s)^{1-2i}.
\]
The identity
\begin{equation}\label{eq:three-halves-identity}
  (T-t)^{-3i-1}\rho(s)^{1-2i}
  =(T-t)^{1/2}\rho(s)^2
  \left(\frac{1}{(T-t)^{3/2}\rho(s)}\right)^{2i+1}
\end{equation}
therefore yields
\begingroup\small
\begin{align*}
\left(
  \frac{|u_{i,0}(t,Z)|}{(2i+1)!}
\right)^{1/(2i+1)}
\le{}&
\left(2M(T-t)^{1/2}\rho(s)^2\right)^{1/(2i+1)}
(1+i)^{-\frac{4}{2i+1}}
\frac{1}{(T-t)^{3/2}\rho(s)}.
\end{align*}
\endgroup
For fixed \(t\), we have
\(\left(2M(T-t)^{1/2}\rho(s)^2\right)^{1/(2i+1)}\longrightarrow1,\)
and
\((1+i)^{-\frac{4}{2i+1}}
=
\exp\left(-\frac{4\log(1+i)}{2i+1}\right)
\longrightarrow1.\)
Consequently,
\[
\limsup_{i\to\infty}
\left(
  \frac{|u_{i,0}(t,Z)|}{(2i+1)!}
\right)^{1/(2i+1)}
\le
\frac{1}{(T-t)^{3/2}\rho(s)}=\frac{1}{R_X(t)},
\]
which implies the radius of convergence is at least $R_X(t)$.
For a fixed \(0<r<R_X(t)\),
\begin{equation}
\left|
\frac{u_{i,0}(t,Z)}{(2i+1)!}X^{2i+1}
\right|
\le
\frac{2M(T-t)^{1/2}\rho(s)^2}{(1+i)^{4}}
\left(\frac{|X|}{R_X(t)}\right)^{2i+1}\leq C(t)\frac{1}{(1+i)^{4}}\left(\frac{r}{R_X(t)}\right)^{2i+1},
\end{equation}
By the Weierstrass \(M\)-test, we also have the uniform convergence of $\tilde u$ on \(|X|\le r\), \(0\le Z\le1\). 
\end{proof}
{
\begin{remark}\label{rmk:local agreement}
Let $u_0$ satisfy
the analyticity and compatibility assumptions of Theorem
\ref{thm:physical-x-radius_intro}. By local well-posedness, the initial datum
$u_0$ generates a unique two-dimensional solution $u$ of
\eqref{2d primitive intro}, with $u(t_0)=u_0$ and
the maximal existence interval is defined to be $[t_0,T_u)$ with $t_0<T_u\leq T$. In view of Proposition \ref{thm:physical-x-radius}, let $\widetilde u$ be the reconstructed local analytic solution, with $\widetilde u(t_0)=u_0$ and $u_{i,k}$ identified through its Taylor series around $X=0$, satisfying the initial bounds \eqref{eq:physical-initial}. Then, by uniqueness, we conclude that \[
u(t,X,Z)=\widetilde u(t,X,Z)
\]
near the symmetry axis for every $t\in[t_0,T_u)$.
The reconstructed solution exists locally up to $T$, although
the full solution may become singular away from the axis earlier.
\end{remark}
}
{
\begin{remark}\label{rmk:redundancy of initial bound}
   Fix $t_0<T$ sufficiently close to $T$, and let $u_0$ satisfy
the analyticity and compatibility assumptions of Theorem
\ref{thm:physical-x-radius_intro}. We regard $u_0$ as an initial
datum prescribed at time $t_0$. There exists
$\rho_0^*(u_0,M,t_0)>0$ such that $u_0$ satisfies the initial
bound \eqref{eq:physical-initial} at $t=t_0$ whenever
$0<\rho_0\leq\rho_0^*$. 
\end{remark}
\begin{proof}
   Since $u_0$ is odd and analytic, there exist a constant \(C_*(u_0)>0\) and a radius
\(r(u_0)>0\) such that, for all \(i,k\ge0\) and \(Z\in[0,1]\),
\begin{equation}\label{eq:interpolant-cauchy}
 |\partial_X^{2i+1}\partial_Z^k u_0(0,Z)|
 \le C_*(2i+1)!k!r^{-2i-1-k}.
\end{equation}
Since for $s_0=s(t_0)\geq e$,  both weights have the positive lower bound
\[ \mathcal H_{k}(t_0,Z)\ge h^*(t_0):=\min\left\{\frac1{s_0},
       \frac{\Lambda^2(\log s_0)^2}{s_0^2}\right\}>0,
\]
where $\mathcal H_{k}=\mathcal H_+$ for $k\geq 1$.
Using \eqref{eq:interpolant-cauchy}, the normalized left side of
\eqref{eq:physical-initial} at the initial time \(t_0\) is bounded by
\begin{equation}\label{eq:initial-weighted-smallness}
\begin{aligned}
&\frac{(1+i+k)^4(T-t_0)^{3i+1}\rho_0^{2i+k-1}}
 {(2i+k+1)!s_0^k\mathcal H_k(T_0,Z)}
 |\partial_X^{2i+1}\partial_Z^k u_0(0,Z)|\leq\frac{C_*}{h_*r^2}\left(\frac{\rho_0}{r}\right)^{2i+k-1}(2i+k+1)^4,
\end{aligned}
\end{equation}
where we have used the previous lower bound on $\mathcal H_k$, the factorial inequality $a! b!\leq (a+b)!$, and $t_0$ is chosen sufficiently close to $T$, or $s_0$ is sufficiently large. Setting $m=2i+k\geq 2$, and letting $\rho_0\leq \frac{r}{2}\implies\left(\frac{\rho_0}{r}\right)^{m-1}
 \le\frac{\rho_0}{r}\,2^{-(m-2)}$, the right hand side of \eqref{eq:initial-weighted-smallness} is further bounded by
$ \frac{C_*\rho_0}{h_*r^3}
 \sup_{m\ge2}(1+m)^4\,2^{-(m-2)}\lesssim \frac{C_*\rho_0}{h_*r^3}.$ Therefore, in view of \eqref{eq:initial-weighted-smallness}, taking $\rho_0$ sufficiently small enough ensures the left hand side of \eqref{eq:initial-weighted-smallness} is smaller than $M$, which is exactly \eqref{eq:physical-initial} at $t=t_0$.
\end{proof}
Combining Proposition \ref{thm:physical-x-radius}, Remark \ref{rmk:redundancy of initial bound} and Remark \ref{rmk:local agreement}, upon rebasing the start time to zero, yields Theorem \ref{thm:physical-x-radius_intro}.
}
\section*{Acknowledgment}

 C.C funded by the European Union. Views and opinions expressed are however
those of the authors only and do not necessarily reflect those of the European Union or the European
Research Council Executive Agency. Neither the European Union nor the granting authority can be held
responsible for them.
C.C. work is supported by ERC grant (FloWAS, No. 101117820, DOI 10.3030/101117820). C. Collot
was supported by the CY Initiative of Excellence Grant ``Investissements d’Avenir” ANR-16-IDEX0008 via Labex MME-DII, the grant ``Chaire Professeur Junior” ANR-22-CPJ2-0018- 01 of the French
National Research Agency, and the grant BOURGEONS ANR-23-CE40-0014-01 of the French National
Research Agency.

S.I.and L.Q. are partially supported by NSERC grant No. 371637-2025. L. Q. is grateful for the support provided by the Mitacs Globalink Research Award.
\printbibliography
\end{document}